\documentclass[11pt,reqno]{amsart}
\usepackage{xifthen,setspace,bbm}
\usepackage{pkgfile}
\newtheorem{example}{Example}

\allowdisplaybreaks

\title[Bernoulli Quadratic Forms]{Asymptotic Distribution of Bernoulli Quadratic Forms}

\author{Bhaswar B. Bhattacharya}
\address{Department of Statistics, University of Pennsylvania, Philadelphia, USA,
{\tt bhaswar@wharton.upenn.edu}}

\author{Somabha Mukherjee}
\address{Department of Statistics, University of Pennsylvania, Philadelphia, USA,
{\tt somabha@wharton.upenn.edu}}

\author[Sumit Mukherjee]{Sumit Mukherjee\textsuperscript{*}}\thanks{\textsuperscript{*}Research partially supported by NSF grant DMS-1712037}
\address{Department of Statistics, Columbia University, New York, USA, {\tt  sm3949@columbia.edu}}

\begin{document}

\begin{abstract} Consider the random quadratic form $T_n=\sum_{1 \leq u < v \leq n} a_{uv} X_u X_v$, where $((a_{uv}))_{1 \leq u, v \leq n}$ is a $\{0, 1\}$-valued symmetric matrix with zeros on the diagonal, and $X_1,$ $X_2, \ldots, X_n$ are i.i.d. $\dBer(p_n)$. In this paper, we prove various characterization theorems about the limiting distribution of $T_n$, in the sparse regime, where $0 < p_n \ll 1$ such that $\E(T_n)=O(1)$. The main result is a decomposition theorem showing that distributional limits of $T_n$ is the sum of three components:  a mixture which consists of a quadratic function of independent Poisson variables; a linear Poisson mixture, where the mean of the mixture is itself a (possibly infinite) linear combination of independent Poisson random variables; and another independent Poisson component. This is accompanied with a universality result which allows us to replace the Bernoulli distribution with a large class of other discrete distributions. Another consequence of the general theorem is a necessary and sufficient condition for Poisson convergence, where an interesting second moment phenomenon emerges. 
\end{abstract}

%\subjclass[2010]{05C15, 60C05,  60F05, 05D99}
%\keywords{}

\maketitle

\section{Introduction}

Let $X_1, X_2, \ldots, X_n$ be i.i.d. $\dBer(p_n)$, where $0< p_n \ll 1$.\footnote{For two sequences $\{a_n\}_{n \geq 1}$ and $\{b_n\}_{n \geq 1}$ of non-negative reals, $a_n \sim b_n$ means $a_n = (1+o(1))b_n$, and $a_n \ll b_n$ means $a_n= o(b_n)$.} Then the well-known Poisson approximation to the Binomial distribution shows that, given a $\{0, 1\}$-valued sequence $a_1, a_2, \ldots, a_n$, the {\it linear statistic} $$L_n=\sum_{i=1}^n a_i X_i \dto \dPois(\lambda),$$ whenever the mean $\E L_n=  p_n\sum_{i=1}^n a_i \rightarrow \lambda$. Conversely, if $0 < p_n \ll 1$ is such that $p_n\sum_{i=1}^n a_i =O(1)$, then whenever $L_n$ converges in distribution to a finite random variable, there exists $\lambda \geq 0$, such that $L_n$ converges to $\dPois(\lambda)$. In other words, in the sparse regime, where $0 < p_n \ll 1$ is chosen such that $\E(L_n)=O(1)$,  the Poisson distribution characterizes the limiting distribution of linear forms in Bernoulli variables. 

In this paper we address the analogous question for quadratic forms in Bernoulli random variables: Given a $\{0, 1\}$-valued symmetric matrix $((a_{uv}))_{1 \leq u, v \leq n}$ with zeros on the diagonal, consider the {\it Bernoulli quadratic form},  
\begin{align}\label{eq:T0}
T_n= \sum_{1 \leq u < v \leq n} a_{uv}X_u X_v,
\end{align}
where, as before, $X_1, X_2, \ldots, X_n$ are i.i.d. $\dBer(p_n)$. In this case, the {\it sparse regime} corresponds to choosing $0< p_n \ll 1$, such that 
\begin{align}\label{eq:sparse}
\E(T_n)=p_n^2 \sum_{1 \leq u < v  \leq n} a_{uv} =O(1).
\end{align} 
In this regime the random variable $T_n=O_P(1)$, therefore, it has distributional limits along subsequences. In fact, using Stein's method for Poisson approximation \cite{goldstein,arratia,barbourholstjanson,CDM}, it is easy to obtain various sufficient conditions on the matrix $((a_{uv}))_{1 \leq u, v \leq n}$ for which $T_n$ is asymptotically Poisson. However, unlike in the linear case, it is easy to construct matrices $((a_{uv}))_{1 \leq u, v \leq n}$ for which $T_n$ has a non-Poisson limit:

\begin{enumerate}

\item[(1)] Take $a_{uv}=1$, for all $1 \leq u \ne v \leq n$, and choose $p_n=\lambda/n$ (for some $\lambda >0$). Then $S_n=\sum_{u=1}^n X_u \dto N\sim \dPois(\lambda)$, and 
\begin{align}\label{eq:Kn}
T_n=\frac{1}{2} \sum_{1\leq u \ne v \leq n} X_u X_v={S_n \choose 2} \dto {N \choose 2},
\end{align}
which is a quadratic function of a Poisson random variable. 

\item[(2)] Take $b_n=\lfloor \sqrt n \rfloor $ and let $a_{uv}=a_{vu}=1$, for $1 \leq u \leq b_n$ and $u b_n+1 \leq v \leq u b_n+b_n$. Then 
$$T_n=\sum_{u=1}^{b_n} X_u \sum_{v=u b_n+1}^{u b_n+b_n} X_v.$$
Here, choosing $p_n = \lambda/\sqrt n$ (for some $\lambda >0$) ensures $\E (T_n) \to \lambda^2$. Then the random variables $J_u=\sum_{v=u b_n+1}^{ub_n+b_n} X_v \sim \dBin(\lfloor \sqrt n \rfloor, \frac{\lambda}{\sqrt n})$, are independent for $1 \leq u \leq b_n$. This implies, 
\begin{align}\label{eq:disjointK1n}
T_n = \sum_{u=1}^{b_n} X_u J_u \stackrel{D}= \dBin\left(\lfloor \sqrt n \rfloor \sum_{u=1}^{b_n} X_u, \frac{\lambda}{\sqrt n}\right) \dto \dPois(\lambda N),
\end{align}
where $N \sim\dPois(\lambda)$ (because  $\sum_{u=1}^{b_n} X_u \dto \dPois(\lambda)$). In this case, the limit is a Poisson distribution with a random mean, that is, it is a {\it Poisson mixture} \cite{lecam}.\footnote{Given a discrete random variable $X$, we denote by $Z \sim \dPois(X)$ a Poisson random variable with a random mean $X$. More precisely, for $z \in \{0, 1, \ldots, \}$, $\P(Z=z) =\E[\frac{e^{-X} X^z}{z!}]$.}  
\end{enumerate}

The different limits obtained in the examples above raise the question: {\it What are the possible limiting distributions of the Bernoulli quadratic form $T_n$ in the sparse regime} \eqref{eq:sparse}? In this paper, we prove a general decomposition theorem which allows us to express the limiting distribution of $T_n$ as the sum of three components:
 a `quadratic component', which is a mixture driven by a bivariate Poisson stochastic integral; a `linear component' which is a Poisson mixture, where the mean of the mixture is itself a univariate Poisson stochastic integral; and an independent Poisson component (Theorem \ref{thm:poisson_quadratic}). Moreover, any distributional limit of $T_n$ must belong to the closure of the class defined by the above decomposition (Theorem \ref{thm:distributional_limit}). This general result has several interesting consequences, such as a characterization theorem for dense matrices (Corollary \ref{cor:poisson_quadratic_W}), a second moment phenomenon for Poisson convergence (Corollary \ref{cor:trunc}), and a universality phenomenon which allows us to replace the Bernoulli distribution with other discrete distributions (Corollary \ref{cor:quadratic_general}). In Section \ref{sec:examples} we use these results to compute the limit of $T_n$ in various natural examples.

\subsection{Limiting Distribution of Bernoulli Quadratic Forms}

Hereafter, without loss of generality, we adopt the language of graph theory, and think of the matrix $((a_{uv}))_{1 \leq u, v \leq n}$ as the adjacency matrix of an undirected simple graph on $n$ vertices. To this end, let $\sG_n$ denote the space of all simple undirected graphs on $n$ vertices labeled by $[n] := \{1, 2, \ldots, n\}$. Given a graph $G_n \in \sG_n$ with adjacency matrix $A(G_n) = ((a_{uv}(G_n)))_{1 \leq u, v \leq n}$, denote by $V(G_n)$ the set of vertices, and by $E(G_n)$ the set of edges of $G_n$, respectively.  Then the  Bernoulli quadratic form \eqref{eq:T0} (indexed by the graph $G_n$) becomes 
\begin{align}\label{eq:TGn}
T_n =\frac{1}{2} \sum_{1 \leq u, v \leq n} a_{u v}(G_n)X_u X_v=\frac{1}{2}\bm X' A(G_n) \bm X,
\end{align}
where $X_1,X_2,\ldots, X_n$ are i.i.d. $\dBer(p_n)$ and $\bm X=(X_1,X_2,\ldots, X_n)'$. The sparse regime \eqref{eq:sparse} translates to $0< p_n \ll 1$ such that\footnote{For two non-negative sequences $(a_n)_{n\geq 1}$ and $(b_n)_{n \geq 1}$, $a_n = \Theta(b_n)$ means that there exist positive constants $C_1, C_2$,
such that $C_1 b_n \leq a_n \leq C_2b_n$, for all $n$ large enough.} 
\begin{align}\label{eq:ETn}
\E[T_n]=|E(G_n)| p_n^2=\Theta(1).
\end{align}
(Note that if $\E[T_n]=o(1)$, then $T_n \pto 0$, hence, to obtain non-degenerate limiting distributions it suffices to consider the case $\E[T_n]=\Theta(1)$.) 

\begin{remark}
The statistic \eqref{eq:TGn}  arises naturally in several contexts, such as non-parametric two-sample tests \cite{fr}, understanding coincidences \cite{diaconismosteller}, and motif frequency estimation in large networks \cite{kw1}. For instance, in the study of coincidences $T_n$ arises as a generalization of the birthday paradox \cite{birthdayexchangeability,dasguptasurvey,diaconisholmes}, where the matrix $((a_{uv}))_{1 \leq u, v \leq n}$ corresponds to the adjacency matrix of a friendship-network graph $G_n$, and one wishes to estimate the probability that there are two friends with birthday on a particular day (say January 31). Then taking $X_1, X_2, \ldots, X_n$ i.i.d. $\dBer(1/365)$ (assuming birthdays are uniformly distributed over the year), $T_n$ counts the number of pairs of friends with birthdays on January 31. This statistic also arises in the problem of estimating frequencies of motifs (small subgraphs) in large graphs \cite{kw1,subgraph_science}. Here, given a large graph $G_n$, the goal is to efficiently estimate (without
storing or searching over the entire graph) global characteristics, such as, the number of edges of $G_n$, by making local queries on $G_n$. In the subgraph sampling model \cite{kw1,degree_distribution}, where one has access to the random induced subgraph obtained by sampling each vertex of $G_n$ independently with probability $p_n$, the statistic $T_n/p_n^2$, by \eqref{eq:ETn},  is an unbiased estimate of the number of edges in $G_n$. 
\end{remark}

Hereafter, we denote $r_n=1/p_n$, and assume that the vertices of $G_n$ are labelled in the non-increasing order of the degrees $d_{1} \geq d_{2} \geq \ldots \geq d_{n}$, where $d_v$ denotes the degree of the vertex labelled $v$. To describe the limiting distribution of $T_n$ we need to consider limits of the sequence of matrices $((a_{uv}))_{1 \leq u, v \leq n}$. This can be done using the framework of graph limit theory \cite{graph_limits_I,graph_limits_II,lovasz_book}. To this end, let $\cW$ be the space of all symmetric measurable functions from $[0, \infty)^2 \rightarrow [0, 1]$. 
Given a graph $G_n$ (and a sequence $r_n \rightarrow \infty$), define the function $W_{G_n} \in \cW$ as follows: 
\begin{align}\label{eq:WGn}
W_{G_n}(x, y):= 
\left\{
\begin{array}{cc}
\bm  1\{(\ceil{xr_n}, \ceil{yr_n}) \in E(G_n)\}  & \text{for }  x, y \in \left[0, \frac{n}{r_n}\right]^2 \\
0  &      \text{otherwise.}
\end{array}
\right.
\end{align}
Moreover, for a graph $G_n$, define the {\it normalized degree-function} as $d_{W_{G_n}}(x)=\int_0^\infty W_{G_n}(x, y) \mathrm dy$. Note that 
\begin{align}\label{eq:normalized_degree}
d_{W_{G_n}}(x):= 
\left\{
\begin{array}{cc}
\frac{1}{r_n} \sum_{j=1}^n a_{\ceil{xr_n} j}(G_n)& \text{for }  x \in \left[0, \frac{n}{r_n}\right] \\
0  &      \text{otherwise.}
\end{array}
\right.
\end{align} 

\begin{defn}\label{defn:Wconvergence}\cite{lovasz_book}
For $K > 0$, the {\it cut-distance} between two functions $W_1, W_2 \in \cW$, restricted to the domain $[0, K]^2$, is defined as,
\begin{align}\label{eq:Wconvergence}
||W_1-W_2||_{\square([0, K]^2)}:=\sup_{f, g: [0, K] \rightarrow [-1, 1]}\left|\int_{[0, K]^2} \left(W_1(x, y)-W_2(x, y)\right) f(x) g(y) \mathrm dx \mathrm dy \right|. 
\end{align} 
The {\it cut-metric} between two functions $W_1, W_2 \in \cW$, restricted to the domain $[0, K]^2$, is defined as,  
\begin{align}\label{eq:Wdelta}
\delta_{\square([0, K]^2)}(W_1, W_2):= \inf_{\psi}||W_1^{\psi}-W_2||_{\square([0, K]^2)}, 
\end{align} 
with the infimum taken over all measure-preserving bijections $\psi: [0, K] \rightarrow [0, K]$, and  $W_1^\psi(x, y):= W_1(\psi(x), \psi(y))$, for $x, y \in [0, K]$. 
\end{defn}

Equipped with the definitions above we can now state our main theorem. To this end, for $p \geq 1$ and a Borel set $\cK \subseteq \R^d$ denote by $L_p(\cK)$ the set of all measurable functions from $\cK \rightarrow \R$ such that $\int_{\cK} |f(\bm x)|^p \mathrm d \bm x < \infty$.

\begin{thm} Let $X_1,X_2,\ldots, X_n$ be i.i.d. $\dBer(p_n)$ and suppose $\{G_n\}_{n \geq 1}$ is a sequence of graphs such that \eqref{eq:ETn} is satisfied. Assume that the vertices of $G_n$ are labelled  $ \{1, 2, \ldots, n\}$ in non-increasing order of the degrees and the following hold: 

\begin{enumerate}[(a)]

\item $\lim_{K \rightarrow \infty} \lim_{n\rightarrow \infty} \frac{1}{2}\int_K^\infty \int_K^\infty W_{G_n}(x, y)  \mathrm dx \mathrm dy = \lambda_0$.

\item There exists a function $W \in \cW$, such that, for $K>0$ large enough,  
\begin{align}\label{eq:WK_condition_I}
\lim_{n \rightarrow \infty}||W_{G_n}-W||_{\square([0, K]^2)} = 0. 
\end{align}

\item There exists a function $d:[0, \infty) \rightarrow [0, \infty)$ in $L_1([0,\infty))$, such that, for $K, M  > 0$ large enough, 
\begin{align}\label{eq:dK_condition_II}
\lim_{n \rightarrow \infty} \int_0^K \left|d_{W_{G_n}} (x) \bm 1\{ d_{W_{G_n}} (x)  \leq M \} - d(x) \bm 1\{d(x) \leq M \}\right| \mathrm dx = 0.
\end{align}
\end{enumerate}
Then 
\begin{align}\label{eq:Q1Q2Q3}
T_n:=\frac{1}{2} \sum_{1 \leq u, v \leq n} a_{u v}(G_n)X_u X_v \dto Q_1+Q_2+Q_3,
\end{align} 
where
\begin{itemize}

\item[--] $Q_3 \sim \dPois(\lambda_0)$ and $Q_3$ is independent of  $(Q_1, Q_2)$. 

\item[--] The joint moment generating function of $(Q_1, Q_2)$ is given by: For $t_1, t_2 \geq 0$, 
\begin{align}\label{eq:Q1Q2}
\E \exp&\Big\{-t_1 Q_1-t_2 Q_2\Big\}\nonumber \\
&=\E \exp\left\{\frac{1}{2}\int_0^\infty \int_0^\infty \phi_{W, t_1}(x, y) \mathrm d N(x)\mathrm dN(y)-(1-e^{-t_2}) \int_0^\infty \Delta(x)\mathrm d N(x)\right\},
\end{align}
with
\begin{itemize}
\item $\int_{[0,\infty)^2}W(x,y)\mathrm dx \mathrm dy<\infty$, 

\item $\Delta(x):=d(x)-\int_0^\infty W(x, y) \mathrm dy$,  

\item $\{N(t), t \geq 0\}$ is a homogenous Poisson process of rate $1$,\footnote{Given a function $f\in L_1([0,\infty)^d)$, $\int f(x_1, x_2, \ldots, x_d) \prod_{a=1}^d \mathrm d N(x_a)$, denotes the multiple It\^{o} stochastic integral of $f$ with respect to the homogeneous Poisson process of rate 1, $\{N(t), t \geq 0\}$. The precise definition of stochastic integration with respect to a Poisson process and methods for computing them are given in Appendix \ref{sec:integral}.} and 

\item $\phi_{W, t_1}(x, y):=\log(1-W(x, y)+W(x, y)e^{-t_1})$.
\end{itemize}
\end{itemize} 
\label{thm:poisson_quadratic}
\end{thm}

The proof of this result is given in Section \ref{sec:pf_quadratic}. The proof proceeds by decomposing the graph into three parts (based on the degree of the vertices), and a truncated moment-comparison argument, which shows that the moments of a truncated version of $T_n$ are close to the moments of another `approximating' variable, for which the asymptotic distribution can be easily  computed. The three parts give rise to the following three components in the limiting distribution of $T_n$: 

\begin{itemize}

\item A {\it quadratic component} $Q_1$ whose moment generating function is given in terms of a bivariate stochastic integral. This is the contribution to $T_n$ from the `dense core' of the graph, that is, edges between the `high-degree' vertices (degree greater than $\frac{r_n}{K}$) of $G_n$. 

\item A {\it linear component} $Q_2$, which is the contribution to $T_n$ from the edges between the `high-degree' and `low-degree' vertices (degree less than $\frac{r_n}{K}$) of $G_n$. Note that the marginal moment generating function of $Q_2$ is 
\begin{align}\label{eq:Q2}
\E \exp \Big\{-t_2 Q_2\Big\} &=\E \exp\left\{-(1-e^{-t_2}) \int_0^\infty \Delta(x)\mathrm d N(x)\right\}. 
\end{align}
By comparing moment generating functions, it is easy to see that $Q_2 \sim \dPois(R_2)$, where $R_2=\int_0^\infty \Delta(x)\mathrm d N(x)$ is a univariate Poisson stochastic integral. This shows that marginally $Q_2$ is a Poisson mixture, where the mixing distribution is a  (possibly) infinite linear combination of independent Poisson random variables. 

\item An {\it independent Poisson component} $Q_3$, which is the contribution from the edges between the `low-degree' vertices of $G_n$.  

\end{itemize}

\begin{remark}\label{rem:poisson_condition} Even though \eqref{eq:Q1Q2} often characterizes the limit of $T_n$ (as shown in Theorem \ref{thm:distributional_limit}, Corollary \ref{cor:poisson_quadratic_W}, and Corollary \ref{cor:trunc} below), the conditions in Theorem \ref{thm:poisson_quadratic}  can be slightly relaxed in a few  ways: 

\begin{itemize} 

\item[(1)] It will be evident from the proof of Theorem \ref{thm:poisson_quadratic} that it suffices to assume \eqref{eq:WK_condition_I} holds, not for all $K$ large enough, but along any diverging sequence $K_s \rightarrow \infty$. Similarly, condition \eqref{eq:dK_condition_II} only needs to hold along diverging sequences of $K$ and $M$. In fact, we show later in Observation \ref{obs:degree_M} that an easy sufficient condition for \eqref{eq:dK_condition_II} to hold along a certain diverging sequence of $M$ is 
$$\lim_{n \rightarrow \infty }||d_{W_{G_n}}-d||_{L_1([0, K])}=0.$$
We will often use the condition above to verify \eqref{eq:dK_condition_II}.  However, the truncated condition in \eqref{eq:dK_condition_II} is, in general, necessary to include graphs with a few high-degree vertices. 

\item[(2)] Another relaxation, which will again be clear from the proof of Theorem \ref{thm:poisson_quadratic}, is to assume \eqref{eq:WK_condition_I} and \eqref{eq:dK_condition_II} hold along  a common bijection (permutation of the vertices) from $[0, K] \rightarrow [0, K]$ (see Lemma \ref{lm:moment_limit} for a precise statement). 
Marginally, this allows one to replace the cut-distance $||\cdot||_{\square([0, K]^2)}$ in 
\eqref{eq:WK_condition_I} with the cut-metric $\delta_{\square([0, K]^2)}$. This generalization will be important for establishing the necessity of the conditions and  characterizing the limits of $T_n$ (in Theorem \ref{thm:distributional_limit} and Corollary \ref{cor:poisson_quadratic_W} below). Nevertheless, to avoid notational clutter, we present Theorem \ref{thm:poisson_quadratic} under the slightly weaker condition, and discuss this generalization as part of the proof in Section \ref{sec:pfquadratic4}. 
\end{itemize}

\end{remark}

Given the above discussion, it is natural to wonder whether the conditions \eqref{eq:WK_condition_I} and \eqref{eq:dK_condition_II} are necessary for the convergence of $T_n$. More generally, one can ask what are the possible limiting distributions of $T_n$? It is easy to construct examples where $T_n$ does not converge in distribution, when the conditions of Theorem \ref{thm:poisson_quadratic} are not satisfied (see Example \ref{ex:starcomplete_II}). However, the question of determining all possible limiting distributions of $T_n$ is more delicate. In the theorem below, we answer this question by showing that whenever $T_n$ has  a distributional limit, it must belong to the closure of limits of the form \eqref{eq:Q1Q2}. To make this precise, denote by $\cF$ the collection of all functions  $d:[0, \infty) \rightarrow [0, \infty)$ in $L_1([0,\infty))$, and consider the following definition: 

\begin{defn}\label{defn:Wf} For $\cW$ and $\cF$ as above, define $\cP(\cW, \cF)$ to be the collection all probability measures $\mu$ on $\Z_{+}\cup \{0\}$, such that if $J \sim \mu$, then
$$J\stackrel{D}= J_1 + J_2,$$
where the joint moment generating function of $(J_1, J_2)$ is given by the RHS of \eqref{eq:Q1Q2}, for some function $W \in \cW$ with $\int_0^\infty \int_0^\infty W(x, y) \mathrm d x \mathrm dy < \infty$ and some function $d \in \cF$, such that $\Delta(x)=d(x)-\int_0^\infty W(x, y) \mathrm d y \geq 0$, for all $x \in [0, \infty)$.  Finally, denote by $\overline{\cP}(\cW, \cF)$ the closure of $\cP(\cW, \cF)$ under weak convergence.\footnote{More precisely,  a probability measure $\mu$ on  $\Z_{+}\cup \{0\}$ belongs to $\overline{\cP}(\cW, \cF)$ if and only if there exists a sequence of probability measures $\{\mu_s\}_{s \geq 1}$, with $\mu_s \in {\cP}(\cW, \cF)$, such that $\mu_s$ converges weakly (in distribution) to $\mu$, as $s \rightarrow \infty$.}
\end{defn}

The following theorem shows that whenever $T_n$ has a distributional limit, it has component which belongs to $\overline{\cP}(\cW, \cF)$ plus an independent Poisson random variable.

\begin{thm}\label{thm:distributional_limit} 
Suppose \eqref{eq:ETn} holds and the random variable $T_n$ converges in distribution to a  random variable $T$. Then $T\stackrel{D}=J+J_0$, where $J \in \overline{\cP}(\cW, \cF)$, $J_0 \sim \dPois(\lambda)$, for some $\lambda \geq 0$, and $J_0$ is independent of $J$. 
\end{thm}

The proof of the above theorem is given in Section \ref{sec:pf_distributional_limit}. We compute the limit of $T_n$ in different examples in Section \ref{sec:examples}. Interestingly, in all the examples constructed in Section \ref{sec:pf_distributional_limit} the limiting distribution of $T_n$ belongs to the class ${\cP}(\cW, \cF)$ itself. This leaves open the intriguing question of  whether there are distributional limits of $T_n$ which are in $\overline{\cP}(\cW, \cF)$ but not in ${\cP}(\cW, \cF)$.

\subsection{Consequences of Theorem \ref{thm:poisson_quadratic}}

The limiting distribution in Theorem \ref{thm:poisson_quadratic} simplifies if the graph sequence $\{G_n\}_{n \geq 1}$ has some special structures. 

We begin with the case when the graph is dense. Recall a sequence of graphs $\{G_n\}_{n \geq 1}$ is said to be {\it dense}, if $|E(G_n)| \geq C n^2$, for some constant $C>0$, when $n$ is large enough. In this case, the assumption \eqref{eq:WK_condition_I} characterizes all limits of $T_n$. Here, the linear mixture and the Poisson components vanish, and the limit of $T_n$ is determined  by the quadratic component. 

\begin{cor}[Dense Graphs] Let $X_1,X_2,\ldots, X_n$ be i.i.d. $\dBer(p_n)$ and suppose $\{G_n\}_{n \geq 1}$ is a sequence of dense graphs such that \eqref{eq:ETn} holds. 

\begin{enumerate}
\item[(a)] Suppose there exists a function $W \in \cW$, such that, for $K>0$ large enough,  
$\lim_{n \rightarrow \infty}||W_{G_n}-W||_{\square([0, K]^2)} = 0$.  Then $W$ vanishes outside a compact rectangle $[0,a]^2$ for some finite $a\ge 0$, and $T_n \dto Q_1$, where 
\begin{align}\label{eq:Q1W}
\E \exp&\Big\{-t_1 Q_1\Big\}=\E \exp\left\{\frac{1}{2}\int_0^a \int_0^a \phi_{W, t_1}(x, y) \mathrm d N(x)\mathrm dN(y) \right\},
\end{align}
with $t_1\geq 0$, $\phi_{W, t_1}(x, y):=\log(1-W(x, y)+W(x, y)e^{-t_1})$, and $\{N(t), t \geq 0\}$ is a homogenous Poisson process of rate $1$. 

\item[(b)]
Conversely, suppose $\{G_n\}_{n \geq 1}$ is a sequence of dense graphs such that \eqref{eq:ETn} holds, and $T_n$ converges in distribution. Then the limit is necessarily of the form  \eqref{eq:Q1W}, for some function $W \in \cW$ which vanishes outside $[0,a]^2$ for some finite $a\ge 0$.
\end{enumerate}
\label{cor:poisson_quadratic_W}
\end{cor} 

The proof of Corollary \ref{cor:poisson_quadratic_W} is given in Section \ref{sec:pfcorollary}. In Section \ref{sec:examples}, we compute the limit in \eqref{eq:Q1W} in various examples.  

Another consequence of Theorem \ref{thm:poisson_quadratic}, is a characterization of when the limiting distribution of $T_n$ is a Poisson random variable. This reveals an interesting truncated {\it second moment phenomenon},  that is, the convergence of the first two moments of a truncated version of $T_n$ determines the convergence in distribution to a Poisson distribution. To this end, for any $M > 0$, define $X_{u,M}:=X_u \bm 1\{d_u\le M r_n\}$ and %$G_{n,M}$ to be the graph obtained from $G_n$ by removing all vertices $v$ such that $d_v>M r_n$, along with all the edges adjacent on $v$. 
%Set 
\begin{align}\label{eq:M}
T_{n, M}= \sum_{(u,v) \in E(G_{n})} X_{u,M}X_{v,M}.
\end{align} 

\begin{cor}[Truncated Second Moment Phenomenon for Poisson Approximation]
 Let $X_1,$ $X_2, \ldots, X_n$ be i.i.d. $\dBer(p_n)$ and suppose $\{G_n\}_{n \geq 1}$ is a sequence of graphs such that \eqref{eq:ETn} holds. Then the following are equivalent.
 \begin{enumerate}
 \item[(a)]
 $T_n\dto \dPois(\lambda)$.
 
\item[(b)] $\lim_{M\rightarrow\infty}\lim_{n\rightarrow\infty}\E T_{n,M}=\lambda$ and $\lim_{M\rightarrow\infty}\lim_{n\rightarrow\infty} \Var(T_{n,M})=\lambda$.\footnote{
For a doubly indexed sequence of real numbers $\{a_{n, m}\}_{n, m \geq 1}$, the double limit $\lim_{m \rightarrow\infty}\lim_{n\rightarrow\infty} a_{n,m}=a$, means $\limsup_{m \rightarrow\infty}\limsup_{n\rightarrow\infty} a_{n,m}= \liminf_{m \rightarrow\infty}\liminf_{n\rightarrow\infty} a_{n,m}=a$.} 
 
\item[(c)] The assumptions of Theorem \ref{thm:poisson_quadratic} hold with $W=0$, $d=0$, and $\lambda_0=\lambda$.

\end{enumerate}\label{cor:trunc}
\end{cor}

This second moment phenomenon for the Poisson distribution for random quadratic forms  complements the well-known fourth-moment phenomenon, which asserts that the limiting normal distribution of certain centered homogeneous forms is implied by the convergence of the corresponding sequence of fourth moments (refer to Nourdin et al. \cite{wchaos_I,wchaos_II} and the references therein, for general fourth-moment theorems and invariance principles and \cite{BDM,xiao} for an example of this phenomenon in random graph coloring). As in the fourth-moment phenomenon for normal approximation, this second moment phenomenon for Poisson approximation exhibits universality (see Section \ref{sec:general_quadratic} below), and we expect this phenomenon to extend beyond the quadratic to general integer-valued homogeneous sums.

\subsection{Universality}
\label{sec:general_quadratic}
 
It is natural to ask what happens if one considers quadratic forms in other integer-valued random variables (not necessarily Bernoulli). More precisely, if $X_1,X_2,\ldots, X_n$ are i.i.d. non-negative integer valued random variables with distribution function $F_n$, then (similar to \eqref{eq:TGn}) the $F_n$-quadratic form, indexed by a graph $G_n$, is defined as  
\begin{align}\label{eq:TGn_general}
T_n =\frac{1}{2} \sum_{1 \leq u, v \leq n} a_{u v}(G_n)X_u X_v=\frac{1}{2}\bm X' A(G_n) \bm X,
\end{align}
where $\bm X=(X_1,X_2,\ldots, X_n)'$. It turns out that the limiting distribution of a general $F_n$-quadratic form exhibits a universality, whenever $X_1$ has the property $\frac{\E X_1}{\P(X_1=1)}=1+o(1)$, that is, the contribution to the expectation is essentially determined by $\P(X_1=1)$.  

\begin{cor}\label{cor:quadratic_general} Suppose $\{X_v\}_{1\le v\le n}$ are i.i.d. non-negative integer valued random variables with $p_n:=\P(X_1=1) \rightarrow 0$, such that $|E(G_n)| p_n^2 =\Theta(1)$ $($as in \eqref{eq:ETn}$)$ and $\lim_{n \rightarrow \infty}\frac{1}{p_n} \E X_1 =  1$. Then if the graph sequence $\{G_n\}_{n \geq 1}$ satisfies the assumptions of Theorem  \ref{thm:poisson_quadratic},
\begin{align*} 
T_n \dto Q_1+Q_2+Q_3,
\end{align*} 
where $T_n$ is as defined in \eqref{eq:TGn_general} and $Q_1$, $Q_2$, and $Q_3$ are as in Theorem \ref{thm:poisson_quadratic}.
\end{cor}

This result shows that Theorem \ref{thm:poisson_quadratic}, and, as a consequence, Corollary \ref{cor:poisson_quadratic_W} and Corollary \ref{cor:trunc}, extend beyond the (sparse) Bernoulli, to include cases like the sparse Poisson, Binomial, Negative Binomial, and Hypergeometric, among others, and complements the well-known universality of the Weiner chaos for centered homogeneous sums \cite{wchaos_I}. 

\begin{enumerate}

\item {\it Sparse Poisson}: Suppose $X_1, X_2, \ldots, X_n$ are i.i.d. $\dPois(\theta_n)$, where $\theta_n \rightarrow 0$. In this case, $\P(X_1=1) =\theta_n e^{-\theta_n} \rightarrow 0$ and $\E X_1=\theta_n$, and so $\frac{\E X_1}{\P(X_1=1)}=e^{\theta_n} \rightarrow 1$, as required in Corollary \ref{cor:quadratic_general}.

\item {\it Sparse Binomial}: Suppose $X_1, X_2, \ldots, X_n$ are i.i.d. $\dBer(m_n, \theta_n)$, where $m_n$ and $\theta_n$ satisfy $m_n\theta_n\rightarrow 0 $. In this case, $\P(X_1=1) = m_n \theta_n (1-\theta_n)^{m_n-1} \rightarrow 0$, and $\E X_1=m_n\theta_n$, and so $\frac{\E X_1}{\P(X_1=1)}=\frac{1}{ (1-\theta_n)^{m_n-1}} \rightarrow 1$, as required in Corollary \ref{cor:quadratic_general}.

\item {\it Sparse Negative Binomial}: Suppose $X_1, X_2, \ldots, X_n$ are i.i.d. $\text{NB}(m_n, \theta_n)$ with $$\P(X_1=r)={r+m_n-1\choose r}(1-\theta_n)^{m_n}\theta_n^r, \quad \text{ for } \quad r=0, 1, \ldots~.$$ where $m_n$ and $\theta_n$ satisfy $m_n\theta_n\rightarrow 0 $. In this case, $\P(X_1=1) = m_n  \theta_n (1-\theta_n)^{m_n} \rightarrow 0$, and $\E X_1=\frac{m_n\theta_n}{1-\theta_n}$, and so $\frac{\E X_1}{\P(X_1=1)}=\frac{1}{ (1-\theta_n)^{m_n+1}} \rightarrow 1$, as required in Corollary \ref{cor:quadratic_general}.

\item {\it Sparse Hypergeometric}: Suppose $X_1, X_2, \ldots, X_n$ are i.i.d. $\text{HGeom}(N_n,K_n,m_n)$ with $$\P(X_1=r)=\frac{{{K_n}\choose{r}}{{N_n-K_n}\choose {m_n-r}}}{{N_n \choose m_n}}, \quad \text{for} \quad r \in \{\max(0,m_n+K_n-N_n),\ldots,\min(m_n,K_n)\},$$
where $(N_n, K_n,m_n)$ satisfy $N_n \rightarrow \infty$, $\frac{m_nK_n}{N_n}\rightarrow 0$, and $\min(m_n, K_n) \geq 1$. This implies $N_n-(m_n+K_n)\rightarrow\infty$, and so, for all $n$ large,  $0$ and 1 are both  in the support of $X_1$.  Further, 
\begin{align*}
\P(X_1=1) =&\frac{K_n {N_n-K_n\choose m_n-1}}{{N_n\choose m_n}}=\frac{m_n K_n}{N_n} \cdot \frac{(N_n-K_n)! (N_n-m_n)!}{(N_n-K_n-m_n+1)! (N_n-1)!}\\
=&\frac{m_nK_n}{N_n}\prod_{s=1}^{m_n-1}\frac{N_n-K_n+1-s}{N_n-s}=\frac{m_nK_n}{N_n}\cdot a_n,
\end{align*}
where
\begin{align*}
1 \geq a_n =\prod_{s=1}^{m_n-1} \left(1-\frac{K_n-1}{N_n-s}\right)\ge \left(1-\frac{K_n-1}{N_n-m_n+1}\right)^{m_n-1} \rightarrow 1,
\end{align*}
since $\frac{m_nK_n}{N_n}\rightarrow 0$. Thus, $\P(X_n=1)\rightarrow 0$ and $\frac{\E X_1}{\P(X_1=1)}=\frac{1}{a_n} \rightarrow 1$, as required in Corollary \ref{cor:quadratic_general}.
\end{enumerate}

\subsection{Organization} The rest of the paper is organized as follows: In Section \ref{sec:examples}, we compute the limiting distribution in various examples. The proofs of Theorem \ref{thm:poisson_quadratic} and Theorem \ref{thm:distributional_limit} are given in Section \ref{sec:pf_quadratic} and Section \ref{sec:pf_distributional_limit}, respectively. The proofs of Corollaries \ref{cor:poisson_quadratic_W}, \ref{cor:trunc}, and \ref{cor:quadratic_general} are given in Section \ref{sec:pfcorollary}. Details about Poisson stochastic integrals  and other technical lemmas are discussed in the appendix. 

\section{Examples} 
\label{sec:examples}

In this section we use Theorem \ref{thm:poisson_quadratic} to compute the limiting distribution of $T_n$ for various graph sequences. In the examples below, we will often construct graph sequences $G_n=(V(G_n), E(G_n))$, where $|V(G_n)| \ne n$, but $|V(G_n)| \rightarrow \infty$, as $n \rightarrow \infty$. In such cases, the definitions in \eqref{eq:WGn} and \eqref{eq:normalized_degree} have to be modified, with the number of vertices $n$ replaced by $|V(G_n)|$ appropriately, following which the results hold verbatim. 

We begin with an application of Corollary \ref{cor:poisson_quadratic_W} for dense block graphons.

\begin{example}(Dense Block Graphons) Let $X_1, X_2, \ldots, X_n$ be i.i.d. $\dBer(\lambda/n)$, for some $\lambda>0$. Fix $\kappa>0$ and consider a sequence of dense graphs $G_n$ converging in cut-metric to the $B$-block function $f: [0, \kappa]^2 \rightarrow [0, 1]$, given by 
\begin{align}\label{eq:block}
f(x, y)=\left\{
\begin{array}{cc}
b_{jj}  & \text{ if }  x, y \in [c_{j-1}, c_j],  \text{ for some } j \in [B], \\
b_{jj'}  & \text{ if }  x \in [c_{j-1}, c_j], y \in [c_{j'-1}, c_{j'}], \text{ for some } j\ne j' \in [B],
\end{array}
\right.
\end{align}
where $c_0=0$, $c_B=\kappa$, $[B] := \{1,2,\cdots,B\}$, and the constants $\{b_{jj'}, j, j' \in [B]\}$, and $c_1, c_2, \ldots, c_B$ are chosen such that $b_{jj'}=b_{j'j}$, for $j \ne j' \in [B]$ and $\int_0^\kappa \int_0^\kappa f(x, y) \mathrm d x \mathrm dy >0$.\footnote{This is obtained as the graph limit of a Stochastic Block Model (SBM) on $\ceil{n\kappa}$ vertices and $B$ blocks, where the edge $(u, v)$ exists independently with probability $b_{jj'}$, when $u \in [\ceil{nc_{j-1}}, \ceil{nc_j}]$ and $v \in [\ceil{nc_{j'-1}}, \ceil{nc_{j'}}]$.}  Now, given $t_1 \geq 0$, recall $\phi_{f, t_1}(x, y):=\log(1-f(x, y)+f(x, y)e^{-t_1})$. Then by Example \ref{integralB} and \eqref{eq:Q1W}, for $t_1\geq 0$, 
\begin{align}
\E \exp&\Big\{-t_1 Q_1\Big\}=\E \exp\left\{ \sum_{j=1}^B \psi_{f, t_1}(j, j) {N_j \choose 2} +   \sum_{1 \leq j < j' \leq B} \psi_{f, t_1}(j, j') N_j  N_{j'}\right\},
\end{align} 
where $\psi_{f, t_1}(j, j'):=\log(1-b_{jj'}+b_{jj'}e^{-t_1})$, for $j,j' \in [B]$, and $\{N_1, N_2, \ldots, N_B\}$ are independent with $N_j \sim \dPois(c_j-c_{j-1})$. Now, consider the random variable,  
\begin{align}\label{eq:Q1B}
Q_1':=\sum_{j=1}^B \eta_{jj}+  \sum_{1 \leq j < j' \leq B} \eta_{jj'},
\end{align}
where $\eta_{jj}\sim \dBin({N_j \choose 2}, b_{jj})$, $\eta_{jj'}\sim \dBin\left(N_j N_{j'}, b_{jj'}\right)$ for $j\neq j'$, and the collection $\{\eta_{jj'}: 1 \leq j, j' \leq B\}$ are independent given $\{N_1, N_2, \ldots, N_B\}$.\footnote{
Given $q \in [0, 1]$ and a discrete random variable $X$, we denote by $Z \sim \dBin(X, q)$ a Binomial distribution with a random number of trails $X$. More precisely, for $z \in \{0, 1, \ldots, \}$, $\P(Z=z) =\E[{X \choose z} q^z (1 - q)^{X-z}]$.} Then it follows that, for $t_1 \geq 0$, 
\begin{align}
\E \exp&\Big\{-t_1 Q_1'|\{N_1, N_2, \ldots, N_B\}\Big\} \nonumber \\
& = \prod_{j=1}^B (1-b_{jj}+b_{jj}e^{-t_1})^{{N_j \choose 2}} \prod_{1\leq j < j' \leq B}  (1-b_{jj'}+b_{jj'}e^{-t_1})^{N_j N_{j'}}\nonumber \\ 
& = \exp\left\{ \sum_{j=1}^B \psi_{f, t_1}(j, j) {N_j \choose 2} +  \sum_{1 \leq j < j' \leq B} \psi_{f, t_1}(j, j') N_j  N_{j'}\right\}.
\end{align}
This implies, for all $t_1 \geq 0$, $\E \exp\{-t_1 Q_1'\}=\E \exp\{-t_1 Q_1\}$, that is, $Q_1\stackrel{D}= Q_1'$, which shows, if $\{G_n\}_{n \geq 1}$ is a sequence of graphs converging to the $B$-block function $f$ (as in \eqref{eq:block}), then $T_n \dto Q_1'$, as defined in \eqref{eq:Q1B}. For specific choices of $f$ this further simplifies. For example, suppose $\{G_n\}_{n \geq 1}$ is a sequence of graphs converging to the 2-block function
$$W(x, y)=\left\{
\begin{array}{cc}
b_{11}  & \text{ for }  x, y \in [0, \alpha],   \\
b_{22}  &   \text{ for }  x, y \in [\alpha, 1],   \\
b_{12}  &    \text{ otherwise.} 
\end{array}
\right.$$
Then,
\begin{align}\label{eq:Tnb2}
T_n \dto \dBin\left({N_1\choose 2}, b_{11}\right) + \dBin\left(N_1N_2, b_{12} \right) + \dBin\left({N_2\choose 2}, b_{22} \right),
\end{align}
where $N_1 \sim \dPois(\alpha \lambda), N_2 \sim \dPois((1-\alpha) \lambda)$ are independent, and the three summands in \eqref{eq:Tnb2} are independent given $N_1, N_2$. This includes as special cases, the Erd\H os-R\'enyi graph and the random bipartite graph.\footnote{By a simple conditioning argument, Corollary \ref{cor:poisson_quadratic_W} and \ref{cor:trunc} can be extended to random graphs by conditioning on the graph, under the assumption that the graph and its coloring are jointly independent (see for example \cite[Lemma 4.1]{BMM}). In particular, the convergence of $T_n$ in Corollary \ref{cor:poisson_quadratic_W} and Corollary \ref{cor:trunc} hold whenever the required conditions hold in probability.}

\begin{itemize}

\item  {\it Dense Erd\H os-R\'enyi  Graphs}:  When $\alpha=1$, the graphon $W$ reduces to the constant function $b_{11}$. This is attained as the graphon limit when $G_n \sim G(n, b_{11})$ is a sequence of  Erd\H os-R\'enyi random graphs such that $b_{11} \in (0, 1]$ is fixed. In this case, \eqref{eq:Tnb2}  simplifies to 
\begin{align}\label{eq:random_dense}
T_n \dto \dBin\left({N_1\choose 2}, b_{11}\right),
\end{align}
where $N_1 \sim \dPois(\lambda)$. In particular, if $b_{11}=1$, that is, $G_n=K_n$ is the complete graph, then $T_n \dto {N_1 \choose 2}$ (recall \eqref{eq:Kn}).

\item  {\it Random Bipartite Graphs}:  When $b_{11}=b_{22}=0$, then this is attained as the limit of the random bipartite graph $G_n \sim G(\ceil{\alpha n}, \ceil{(1-\alpha) n}, b_{12})$, with edge probability $b_{12} \in (0, 1]$. Then, \eqref{eq:Tnb2}  simplifies to
$$T_n \dto \dBin\left(N_1N_2, b_{12}\right),$$
where $N_1\sim \dPois(\alpha \lambda)$, $N_2 \sim \dPois((1-\alpha)\lambda)$ are independent. 

\end{itemize}

\end{example}

For more sparser graphs, the limiting distribution is often a Poisson, and Corollary \ref{cor:trunc} can be applied.

\begin{example} (Non-Dense Approximately Regular Graphs)
Let $X_1,X_2,\ldots, X_n$ be i.i.d. $\dBer(p_n)$ and $\{G_n\}_{n \geq 1}$ be a sequence of graphs such that 
\begin{align}\label{eq:lambda_rn}
\lim_{n\rightarrow\infty} |E(G_n)| p_n^2 = \lambda \quad \text{and} \quad \Delta(G_n):=\max_{v \in V(G_n)} d_v = o(r_n).
\end{align}
Then for any $\varepsilon>0$ there exists $n$ large enough, such that $d_v \leq \varepsilon r_n$, for all $v \in V(G_n)$.  Hence, for any $M \geq 1$ and $n$ large enough $T_n=T_{n, M}$. This implies, 
\begin{align}\label{eq:ETnMrn}
\lim_{M\rightarrow\infty}\lim_{n\rightarrow\infty} \E T_{n, M}=\lim_{n\rightarrow\infty} \E T_n=\lim_{n\rightarrow\infty} |E(G_n)| p_n^2 \rightarrow \lambda.
\end{align}
Moreover, for all large $n$, 
\begin{align}\label{eq:VTnMrn}
\Var(T_{n, M})=\Var(T_n)=|E(G_n)| \Var(X_1 X_2) + 2 N(K_{1, 2}, G_n)    \Cov(X_1 X_2, X_1X_3),
\end{align}
where $N(K_{1, 2}, G_n) =\sum_{v=1}^n {d_v \choose 2}$ denotes the number of 2-stars in the graph $G_n$. Note that $\Var(X_1 X_2) = p_n^2-p_n^4$ and $\Cov(X_1 X_2, X_1X_3)=p_n^3-p_n^4$. Therefore, $\lim_{n\rightarrow\infty} |E(G_n)| \Var(X_1 X_2) = \lambda$, and using $N(K_{1, 2}, G_n) \leq \varepsilon |E(G_n)| r_n$,  gives $\limsup_{n\rightarrow\infty} N(K_{1, 2}, G_n)    \Cov(X_1 X_2, X_1X_3) \leq \varepsilon$. Then \eqref{eq:VTnMrn} implies, $$\lim_{M\rightarrow\infty}\lim_{n\rightarrow\infty} \Var(T_{n, M}) = \lambda,$$ since $\varepsilon$ is arbitrary. This combined with \eqref{eq:ETnMrn} and Corollary \ref{cor:trunc} shows that $T_n \dto \dPois(\lambda)$, whenever \eqref{eq:lambda_rn} holds. This derives the limiting distribution of non-dense (that is, $|E(G_n)|=o(n^2)$), `approximately' regular graphs.

\begin{itemize} 

\item {\it Non-Dense Regular Graphs}: Let $G_n$ be a sequence of $d$-regular graphs  such that $d=o(n)$ and $\frac{nd}{2} p_n^2 \rightarrow \lambda$. Then $r_n=1/p_n=\Theta(\sqrt{nd})$ and the maximum degree $d=o(r_n)$. Therefore, by the argument above, $T_n \dto \dPois(\lambda)$.

\item  {\it Non-Dense Erd\H os-R\'enyi  Graphs}:   Let $G_n \sim G(n, q_n)$ be a sequence of  Erd\H os-R\'enyi random graphs such that $\frac{\log n}{n} \ll  q_n  \ll 1$ and $\frac{n^2 q_n}{2} p_n^2 \rightarrow \lambda$. Then $r_n=1/p_n=\Theta( n \sqrt{q_n})$ and the maximum degree $\Delta(G_n)=(1+o_P(1))nq_n=o(r_n)$ \cite{KS}. Therefore, by the argument above, $T_n \dto \dPois(\lambda)$. 
\end{itemize}
\end{example}

In the example above, the maximum degree of $G_n$ is `small', and, as a result, condition  (b) in Corollary \ref{cor:trunc} holds for the original (un-truncated) random variable $T_n$, as well (see \eqref{eq:ETnMrn} and \eqref{eq:VTnMrn}). However, the truncation is necessary when there are few vertices with `large' degree, as illustrated below.

\begin{example} Let $X_1,X_2,\ldots, X_n$ be i.i.d. $\dBer(\gamma/\sqrt n)$. We consider two examples where truncation matters: 

\begin{itemize}

\item[(1)] Let $G_n=K_{1, n}$ be the $n$-star. Then $|E(G_n)|=n$ and \eqref{eq:ETn} is satisfied. In this case, since the degree of the central vertex of the star is $n \gg M \sqrt n$, for any $M\geq 1$, $T_{n,M}$ is identically zero. Hence, condition (b) in Corollary \ref{cor:trunc} holds with $\lambda=0$, which implies $T_n \pto 0$. 

\item[(2)] To get a non-zero limiting distribution, take $G_n$ to be the disjoint union of a $n$-star $K_{1, n}$ and $n$ disjoint edges ($\{a_1, b_1\}, \{a_2, b_2\}, \ldots, \{a_n, b_n\})$. As before, there is  no contribution to $T_{n,M}$ from the star-graph, and 
$$T_{n,M}=\sum_{j=1}^n X_{a_j} X_{b_j}.$$
This is the sum of independent indicators $Z_j=X_{a_j} X_{b_j}\sim \dBer(\gamma^2/n)$, and hence $\E T_{n,M}= \gamma^2$ and $\Var(T_{n,M})\rightarrow \gamma^2$. Then, by Corollary \ref{cor:trunc}, $T_n \dto \dPois(\gamma^2)$. 
\end{itemize}
Note that, as expected, in both the examples above the convergence is not in $L_1$: in (1) $\E T_n=\gamma^2$ and in (2) $\E T_n=2\gamma^2$.
\end{example}

The Poisson mixture arises in the limit of $T_n$ for bipartite graph which have many `high' degree vertices on one of the sides, and is best illustrated by considering a disjoint union of star graphs. 

\begin{example}\label{example:sI} (Disjoint Union of Stars) Let $G_n$ be the disjoint union of $n$ isomorphic copies of the $n$-star $K_{1, n}^{(1)}, \ldots, K_{1, n}^{(n)}$. Note that, $|V(G_n)|=n^2 + n$ and $|E(G_n)|=n^2$. Label the central vertices of the stars $1, 2, \ldots, n$, the leaves of the vertex 1 as $n+1, \ldots, 2n$, the leaves of the vertex 2 as $2n+1, \ldots, 3n$, and so on. Let $X_1, X_2, \ldots, X_n$ be i.i.d. $\dBer(1/n)$, which ensures $\E(T_n)=\frac{|E(G_n)|}{n^2}=1$. Fix $K \geq 1$, denote by $G_{n,K}$ the induced subgraph of $G_n$ on the first $\ceil{Kn}$ vertices. Then 
$$\int_0^K \int_0^K W_{G_n}(x, y) \mathrm d x\mathrm dy \leq \frac{2|E(G_{n, K})|}{n^2} \lesssim \frac{K}{n} \rightarrow 0,$$
as $n \rightarrow \infty$.\footnote{For $a, b \in \R$, $a \lesssim b$, $a \gtrsim b$, and $a \sim b$ means $a \leq C_1 b$,  $a \geq C_2 b$, and $C_2 b \leq a \leq C_1 b$, for some universal constants $C_1, C_2 >0$, respectively.} Therefore, $||W_{G_n}||_{\square([0, K]^2)}  \lesssim ||W_{G_n}||_{L_1([0, K]^2)} \rightarrow 0$, that is, condition \ref{eq:WK_condition_I} holds with $W=0$. Moreover, for every $K \geq 1$, there is no edges in $G_n$ between the vertices $\{\ceil{Kn}+1, \ldots, n^2\}$, which means  
$\lim_{K \rightarrow \infty}\lim_{n \rightarrow \infty} \int_K^\infty \int_K^\infty W_{G_n}(x, y) \mathrm d x\mathrm dy=0.$ Finally, the normalized degree-functional is (recall \eqref{eq:normalized_degree}),
\begin{align*}%\label{eq:starIc}
d_{W_{G_n}}(x):= 
\frac{1}{n} \sum_{j=1}^{n^2+n} a_{\ceil{xn} j}(G_n) = \left\{
\begin{array}{cc} 
1 & \text{ for }  x \in \left[0, 1\right] \\
\frac{1}{n}  &      \text{ for }  x \in  \left(1, n+1 \right].  
\end{array}
\right.
\end{align*}
This converges to the function $d(x)= \bm  1\{x \in [0, 1]\}$ in $L_1([0, K])$.  To see this, fixing $K \geq 1$,  note that $\int_0^K |d_{W_{G_n}}(u) - d(u) | \mathrm du = \frac{K-1}{n} \rightarrow 0$. Therefore, the conditions of  Theorem \ref{thm:poisson_quadratic} hold with $\lambda_0=0$, $W=0$, and $d(x) = \bm 1\{x\in [0, 1]\}$ (by the discussion in Remark \ref{rem:poisson_condition} and Observation \ref{obs:degree_M}). Hence, 
$$T_n \dto \dPois(N), \quad \text{where} \quad N\sim \dPois(1),$$
This is a type of {\it compound Poisson distribution}: a special case of the Poisson mixture, where the mean itself is a Poisson random variable (recall \eqref{eq:disjointK1n} with $\lambda=1$). 
\end{example}

%%%%%%%%%%%%%%%%%%%%%%%%%%%%%%%%%%%%%%%%%%%%%%%%%%%%%%%%%%%%%%%%%%%%%%%%%%%%%%%%
\begin{figure*}[h]
\centering
\begin{minipage}[c]{1.0\textwidth}
\centering
\includegraphics[width=3.5in]
    {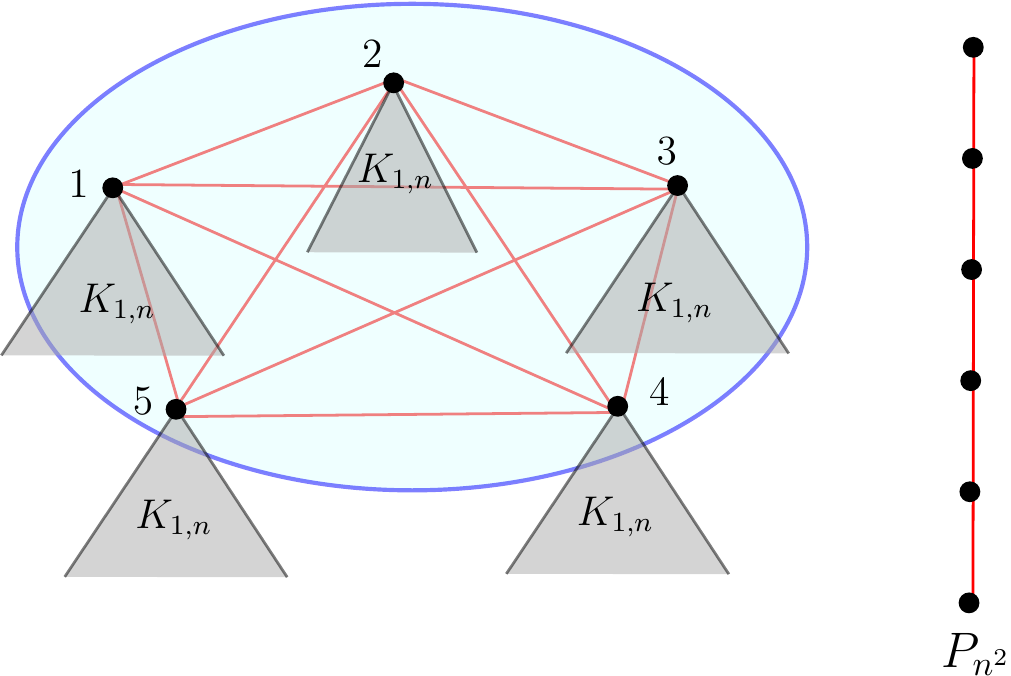}\\
%\small{(a)}
\end{minipage}
\caption{\small{Illustration for Example \ref{ex:starcomplete_I}.}}
\vspace{-0.1in}
\label{fig:component}
\end{figure*}
%%%%%%%%%%%%%%%%%%%%%%%%%%%%%%%%%%%%%%%%%%%%%%%%%%%%%%%%%%%%%%%%%%%%%%%%%%%%%%%%%%%%%%%%%%%%%%%%%%%%%%%%%%

One can easily modify the example above to construct graph sequences for which the quadratic component and the Poisson mixture component appear together in the limit:

\begin{example}\label{ex:starcomplete_I} (Coexistence I) Let $X_1, X_2, \ldots, X_n$ be i.i.d. $\dBer(1/n)$. Construct the graph $G_n$ as follows (see Figure \ref{fig:component}): 
\begin{itemize}
\item Consider disjoint union of $n$ isomorphic copies of the $n$-star $K_{1, n}^{(1)}, \ldots, K_{1, n}^{(n)}$, with vertices labeled as in Example \ref{example:sI} above. 

\item Place a complete graph $K_n$ on the vertices labeled $1,2, \ldots, n$. 

\item Place a path of length $n^2$ with vertices labelled $n^2+n+1, \ldots, 2n^2+n$, disjoint from everything else. 
\end{itemize} 
Here, $|V(G_n)|=2n^2+n$ and $|E(G_n)|={n \choose 2} + n^2 + n^2-1 \sim \frac{5}{2}n^2$, and, hence, \eqref{eq:ETn} holds. Then by arguments similar to Example \ref{example:sI} above, it is easy to check that the conditions of Theorem \ref{thm:poisson_quadratic} hold with 
\begin{align}\label{eq:Wdlambda}
W(x, y) = \bm 1 \{ (x, y) \in [0, 1]^2 \}, \quad d(x)=2 \cdot \bm  1\{x \in [0, 1]\}, \quad \text{and} \quad \lambda_0=1.
\end{align} 
Then $\Delta(x)=d(x)-\int_0^\infty W(x, y) \mathrm dy = \bm  1\{x \in [0, 1]\}$, and by Theorem \ref{thm:poisson_quadratic}, 
\begin{align}\label{eq:starcomplete}
T_n  \dto Q_1+Q_2+Q_3,
\end{align} 
where $Q_3 \sim N_1\sim \dPois(1)$ is independent of  $(Q_1, Q_2)$, and the joint moment generating function of $(Q_1, Q_2)$ is: 
\begin{align}
\E \exp \Big\{-t_1 Q_1-t_2 Q_2\Big\} =\E \exp\left\{- t_1{N_2 \choose 2}-(1-e^{-t_2}) N_2 \right\},
\end{align}
where $N_2 \sim \dPois(1)$. In other words, with a slight abuse of notation, we can write 
$$T_n  \dto {N_2 \choose 2} + \dPois(N_2) + N_1,$$ 
where $N_1$ is independent of ${N_2 \choose 2} + \dPois(N_2)$.  
\end{example}

By repeating the constructions above, it is possible to have distributions where the range of the integrals in \eqref{eq:Q1Q2} are infinite (unlike in the example above, where the range of the integral reduces to $[0, 1]$ because of \eqref{eq:Wdlambda}):

\begin{example} (Coexistence II) Suppose $X_1, X_2, \ldots, X_n$ be i.i.d. $\dBer(1/n)$. For $s \geq 1$, let $a_s=\frac{1}{16^s}$, $b_s=4^s$, and $c_s=\frac{1}{32^s}$. Now, construct the graph $G_n$ as follows: 
\begin{itemize}
\item For each $s \in [\ceil{\log_4 n}]$, take $\ceil{b_s n}$ disjoint isomorphic copies of  the $\ceil{a_s n}$-stars $K_{1, \ceil{a_s n}}^{(1)}$, $K_{1, \ceil{a_s n}}^{(2)}$, \ldots, $K_{1, \ceil{a_s n}}^{(\ceil{b_s n})}$. Label the central vertices of the  $\ceil{a_s n}$-stars, $\ceil{b_{s-1} n}+1, \ldots, \ceil{b_s n}$, where $b_0=0$. 

\item For each $s \in [\ceil{\log_4 n}]$, place a Erd\H os-R\'enyi  random graph $G(\ceil{b_s n}, c_s)$ on the vertices labeled $\ceil{b_{s-1} n}+1, \ldots, \ceil{b_s n}$. 

\end{itemize}
Note that $|V(G_n)|=\sum_{s=1}^{\ceil{\log_4 n}} \ceil{b_s n} (\ceil{a_s n}+1)  =\Theta(n^2)$, and 
$$\E |E(G_n)| \sim \frac{1}{2} \sum_{s=1}^{\ceil{\log_4 n}}   \ceil{b_s n}^2 c_s + \sum_{s=1}^{\ceil{\log_4 n}} \ceil{b_s n} \ceil{a_s n} =\Theta(n^2).$$ (Note that the choice $p_n=1/n$ implies \eqref{eq:ETn} holds.) As before, it can be verified that the conditions of Theorem \ref{thm:poisson_quadratic} hold with $\lambda_0=0$, 
\begin{align}%\label{eq:normalized_degree}
W(x, y) = \left\{
\begin{array}{cc} 
c_s & \text{ for }  x, y \in \left[ r_{s-1},  r_s \right] \\ 
0 & \text{ otherwise,}
\end{array}
\right. 
\quad 
\text{and} \quad  
d(x)=\left\{
\begin{array}{cc} 
c_s + a_s & \text{ for }  x \in \left[ r_{s-1},  r_s \right] \\ 
0 & \text{ otherwise.}
\end{array}
\right.
\end{align}
for $s \geq 1$ and $r_s=\sum_{i=0}^s  b_i$ is the $s$-th partial sum of the sequence $\{b_i\}_{i \geq 1}$.
Now, for $t_1 \geq 0$ and recalling $\phi_{W, t_1}(x, y):=\log(1-W(x, y)+W(x, y)e^{-t_1})$, it follows from Example \ref{integralB} that,  
\begin{align}\label{eq:Winfinite}
\int_0^\infty \int_0^\infty \phi_{W, t_1}(x, y) \mathrm d N(x) \mathrm d N(y) = \sum_{s=1}^\infty \psi_{W, t_1}(s) {N_s \choose 2},
\end{align}
where $\psi_{W, t_1}(s):=\log(1-c_s+c_se^{-t_1})$, for $s \geq 1$, and $N_s \sim \dPois(b_s)$ and $\{N_1, N_2, \ldots  \}$ are independent. Moreover, 
\begin{align}\label{eq:Delta_infinite}
\int_0^\infty \Delta(x)\mathrm d N(x) = \int_0^\infty \left(d(x)-\int_0^\infty W(x, y)\mathrm dy\right)\mathrm d N(x)= \sum_{s=1}^\infty a_s N_s.
\end{align}
Combining \eqref{eq:Winfinite} and \eqref{eq:Delta_infinite}, with Theorem \ref{thm:poisson_quadratic}, it follows that $T_n \dto Q_1+ Q_2$, where 
\begin{align}\label{eq:R2}
\E \exp\Big\{-t_1Q_1-t_2 Q_2\Big\} &= \E \exp\left\{-t_1 \sum_{s=1}^\infty \psi_{W, t_1}(c_s) {N_s \choose 2}  - (1-e^{-t_2})   \sum_{s=1}^\infty a_s N_s \right\}. 
\end{align}
This can be rewritten, by comparing moment generating functions, as 
$$T_n \dto \sum_{s=1}^\infty \dBin\left( {N_s \choose 2}, \frac{1}{32^s} \right)+\dPois\left(\sum_{s=1}^\infty \frac{N_s}{16^s} \right),$$ 
where, as before, $N_s \sim \dPois(4^s)$ are independent, and conditional on the sequence $\{N_1, N_2, \ldots  \}$, the Poisson and the Binomials above are independent.
\end{example}

We conclude with an example where $T_n$ does not have a limit in distribution, showing the necessity of the conditions in Theorem \ref{thm:poisson_quadratic}.

\begin{example}\label{ex:starcomplete_II} (Non-Existence of Limit) Let $X_1, X_2, \ldots, X_n$ be i.i.d. $\dBer(1/n)$. We will construct a graph sequence $\{G_n\}_{n \geq 1}$ for which $T_n$ does not converge in dsitribution. Let $G_n$ be defined as: 
\begin{itemize}

\item Consider a Erd\H os-R\'enyi random graph $G(n, \frac{1}{4})$ on the vertices labelled $1,2, \ldots, n$, and another independent Erd\H os-R\'enyi random graph $G(n, \frac{1}{2})$ on the vertices labelled $n+1, n+2, \ldots, 2n$.

\item For $n$ odd, attach $n$ disjoint $n$-stars $K_{1, n}^{(1)}, K_{1, n}^{(2)}, \ldots, K_{1, n}^{(n)}$, with central vertices at $1, 2, \ldots, n$ respectively. 

\item For $n$ even, attach $n$ disjoint $n$-stars $K_{1, n}^{(1)}, K_{1, n}^{(2)}, \ldots, K_{1, n}^{(n)}$, with central vertices at $n+1, n+2 \ldots, 2n$, respectively. 

\end{itemize}
Here, $|V(G_n)|=\Theta(n^2)$ and $\E |E(G_n)|=\Theta(n^2)$, hence, \eqref{eq:ETn} holds. Now, from the arguments in \eqref{eq:random_dense} and Example \ref{example:sI}, it follows that 
the contribution to $T_n$ from the $G(n, \frac{1}{4})$ and $G(n, \frac{1}{2})$ components converge to $\dBin({N_1 \choose 2}, \frac{1}{4})$ and $\dBin({N_2 \choose 2}, \frac{1}{2})$, respectively, where $N_1, N_2$ are independent $\dPois(1)$. Moreover, the contribution of the $n$ disjoint stars converge to $\dPois(N_1)$ along the odd subsequence and $\dPois(N_2)$ along the even subsequence. Therefore, along the  odd subsequence, 
\begin{align}\label{eq:limitI}
T_n \dto \dBin\left({N_1 \choose 2}, \frac{1}{4}\right) +\dPois(N_1) + \dBin\left({N_2 \choose 2}, \frac{1}{2}\right),
\end{align}
and along the even subsequence
\begin{align}\label{eq:limitII}
T_n \dto \dBin\left({N_1 \choose 2}, \frac{1}{4}\right) +\dPois(N_2) + \dBin\left({N_2 \choose 2}, \frac{1}{2}\right),
\end{align}
where $N_1, N_2$ are independent $\dPois(1)$, and the conditional on $N_1, N_2$, the Poisson and the 2 binomials are independent.  Clearly, the distributions in \eqref{eq:limitI} and \eqref{eq:limitII} are not the same (this can be easily seen by computing their second moments), that is, $T_n$ does not converge in distribution. This is because, for all $K \geq 1$, the function $d_{W_{G_n}}$ converges in $L_1([0, K])$, to the function $d_+(x)= \bm 1\{ x \in [0, 1] \}$ along the odd subsequence, and to the function $d_-(x)= \bm 1\{ x \in [1, 2] \}$ along the even subsequence, respectively. This shows condition \eqref{eq:dK_condition_II} in Theorem \ref{thm:poisson_quadratic} does not hold. In fact, in this case it can be shown that there is no permutation of the vertices $\{1, 2, \ldots, 2n\}$ for which conditions \eqref{eq:WK_condition_I} and \eqref{eq:dK_condition_II} simultaneously hold, in the permuted graph (recall the discussion in Remark \ref{rem:poisson_condition}). 
\end{example}

\section{Proof of Theorem \ref{thm:poisson_quadratic}}
\label{sec:pf_quadratic}

For positive integers $a< b$, denote by $[a, b] :=\{a, a+1, \ldots, b\}$.\footnote{We will often slightly abuse notation and also use $[a, b]$ to denote the closed interval with points $a, b \in \R$, whenever it is clear from the context.} Throughout we assume that the vertices of $G_n$ are labelled  $ \{1, 2, \ldots, n\}$ in non-increasing order of the degrees. Recall that $d_v$ denotes  the degree of the vertex labelled $v$.

\begin{obs}\label{obs:label} If the vertices of $G_n$ are labelled $\{1, 2, \ldots, n\}$ in the non-increasing order of the degrees $d_{1} \geq d_{2} \geq \ldots \geq d_{n}$, then 
\begin{align}\label{eq:label}
\lim_{K \rightarrow \infty}\lim_{n \rightarrow \infty} \frac{d_{\ceil{K r_n}}}{r_n}=0. 
\end{align}
\end{obs}

\begin{proof} Note that $$2|E(G_n)| =\sum_{v=1}^n  d_{v} \geq \sum_{v=1}^{\ceil{K r_n}}d_{v} \geq \ceil{K r_n} d_{\ceil{K r_n}},$$
which implies, by \eqref{eq:ETn}, $d_{\ceil{K r_n}} \lesssim \frac{r_n}{K}$, hence \eqref{eq:label} holds.
\end{proof} 

The first step in the proof of Theorem \ref{thm:poisson_quadratic} is a truncation argument, which shows that vertices with `large' degree have negligible contribution to $T_n$.  To this end, recall the definition of $T_{n,M}$ from \eqref{eq:M}. We begin by showing that the difference between $T_{n}$ and the truncation $T_{n, M}$ above, goes to zero in probability.  
\begin{lem}\label{Tnapprox} Let $T_{n}$ and $T_{n,M}$ be as in defined in \eqref{eq:TGn} and \eqref{eq:M}, respectively. Then 
$$\lim_{M \rightarrow \infty}  \lim_{n \rightarrow \infty} \P\left(T_{n} \ne T_{n, M}\right) = 0.$$  
\end{lem}

\begin{proof} Fix $n \geq 1$ and $M > 1$. Then 
\begin{align*}
\P\left(T_{n} \ne T_{n, M} \right) \leq & \P\left(\exists a \in V(G_n): d_a> M r_n \text{ and } X_a = 1\right)\\
\leq & \sum_{a \in V(G_n): d_a>  M r_n} \P(X_a  = 1)\\
=& \sum_{a \in V(G_n)} p_n \bm 1 \left\{d_a>  M/p_n\right\} \tag*{(recall $r_n=1/p_n$).} \\
\leq &  \sum_{a \in V(G_n)} \frac{p_n^2 d_a}{M}= \frac{2|E(G_n)| p_n^2}{M},
\end{align*}
which goes to zero under the double limit, by assumption \eqref{eq:ETn}.
\end{proof}

This shows that it suffices to derive the limiting distribution of $T_{n, M}$. Now, fix $K \geq 1$, and define $$V_{G_n, K}^+:=[\ceil{Kr_n}] \quad \text{and} \quad V_{G_n, K}^-:=[\ceil{Kr_n}+1, n],$$
the first $\ceil{Kr_n}$ vertices and the last $n-\ceil{Kr_n}$ vertices, respectively. Denote by  $$G_{n, K}^+:=G_n[V_{G_n, K}^+]  \quad \text{and} \quad G_{n, K}^-:=G_n[V_{G_n, K}^-],$$ the subgraphs of $G_n$ induced by $V_{G_n, K}^+$ and $V_{G_n, K}^-$, respectively.\footnote{For $S \subseteq V(G_n)$, $G_n[S]$ denotes induced sub-graph of $G_n$ with vertex set $S$.} Finally, let $G_{n, K}^\pm$ be the subgraph of $G_n$ formed by the union of edges with one end point in $V_{G_n, K}^+$ and the other in $V_{G_n, K}^-$. Note that by definition the subgraphs $G_{n, K}^+$, $G_{n, K}^{\pm}$, and $G_{n, K}^-$ partition the edges of $G_n$, that is, $E(G_n) = E(G_{n, K}^+) \bigcup E(G_{n, K}^\pm) \bigcup E(G_{n, K}^-)$  is a disjoint partition of $E(G_n)$. Therefore, we can decompose $T_{n, M}$ as follows: 
\begin{align}\label{eq:Tn123}
T_{n, M}=\sum_{(u, v) \in E(G_n)}X_{u, M} X_{v, M}=T_{n, K, M}^+ + T_{n, K, M}^\pm + T_{n, K, M}^-,
\end{align}
where 
\begin{align}\label{eq:Tn13defn}
T_{n, K, M}^+:=\sum_{(u, v) \in E(G_{n,K}^+)}X_{u, M} X_{v, M} \quad \text{and} \quad T_{n, K, M}^-:=\sum_{(u, v) \in E(G_{n,K}^-)}X_{u, M} X_{v, M}, 
\end{align}
and 
\begin{align}\label{eq:Tn2defn}
T_{n, K, M}^\pm:=\sum_{u \in  V_{G_n, K}^+} \sum_{v \in  V_{G_n, K}^-} a_{uv}(G_n) X_{u, M} X_{v, M}. 
\end{align}

The proof of Theorem \ref{thm:poisson_quadratic} involves deriving the joint distribution of the three terms in \eqref{eq:Tn123}. It has the following main steps: 

\begin{enumerate}
\item[(1)] In Section \ref{sec:pfquadratic1} we show that the moments of $T_{n, K, M}^-$ are close to an approximating variable $W_{n, K, M}$ obtained by replacing the dependent collection of random variables $\{X_uX_v : (u, v) \in E(G_{n, K}^-)\}$ with a collection of independent $\dBer(p_n^2)$ random variables $\{ R_{uv} : (u, v)\in E(G_{n, K}^-)\}$.

\item[(2)] In Section \ref{sec:pfquadratic2} we show that $T_{n, K, M}^-$ is asymptotically independent of $(T_{n, K, M}^+, T_{n, K, M}^\pm)$ in moments. 

\item[(3)] In Section \ref{sec:pfquadratic3} we show that the moments of $T_{n, K, M}^\pm$ are close to an approximating variable $Z_{n, K, M}$, which has  more independence structure than $T_{n, K, M}^-$. 

\item[(4)] In Section \ref{sec:pfquadratic4} we show that the joint distribution of $(T_{n, K, M}^+, Z_{n, K, M})$ converges in distribution and in moments under the assumptions of  Theorem \ref{thm:poisson_quadratic}, as $n \rightarrow \infty$ and for all fixed $K, M$ large enough.  

\item[(5)] To compute the joint distribution of $(T_{n, K, M}^+, Z_{n, K, M})$ we replace the graph in $G_{n, K}^+$ by an inhomogeneous random graph which has the same graph limit as $G_{n, K}^+$ (Section \ref{sec:pfquadratic5}). In this case, limiting moment generating function can be explicitly computed by first taking the expectation with respect to the randomness of the graph. The existence of the limit proved in the earlier section can then be used to show that this has the same limit as $(T_{n, K, M}^+, Z_{n, K, M})$. 

\item[(6)] The proof of \eqref{eq:Q1Q2Q3} is completed in Section \ref{sec:pfquadratic6}, which entails moving from the joint distribution of $(T_{n, K, M}^+, Z_{n, K, M}, W_{n, K, M})$ to that of the actual variables $(T_{n, K, M}^+, T_{n, K, M}^\pm, T_{n, K, M}^-)$, by verifying the Stieltjes moment condition \cite{Stieltjes} and taking limits in the various parameters.

\end{enumerate}

\subsection{Moment Approximation for $T_{n,K,M}^-$}
\label{sec:pfquadratic1}

Define 
\begin{align}\label{eq:Wn-}
W_{n, K}^-:=\sum_{(u, v) \in E(G_{n, K}^-)} R_{uv},
\end{align} 
where $\{ R_{uv} \}_{(u, v)\in E(G_{n, K}^-)}$ is a collection of independent Bernoulli$(p_n^2)$ random variables. In the following lemma, we show that $T_{n,K,M}^-$ and $W_{n, K}^-$ are close in moments. To this end, we need a few notations: For any two graphs $H$ and $G$, let $N(H, G)$ denote the number of isomorphic copies of $H$ in $G$.

\begin{lem}\label{lm1} Fix $M \geq 1$. Then, under the assumptions of Theorem \ref{thm:poisson_quadratic}, for every positive integer $a \geq 1$, 
$$\lim_{K \rightarrow \infty} \limsup_{n \rightarrow \infty}  \Big| \E \left[({T_{n,K,M}^-})^a\right]  - \E \left[({W_{n, K}^-})^a\right]\Big|=0.$$
Moreover, $\limsup_{K \rightarrow \infty} \limsup_{n \rightarrow \infty}\E \left[({T_{n,K,M}^-})^a\right]< C(a)$, for some constant $C(a)>0$.
\end{lem}

\begin{proof} Fix $\varepsilon \in (0, 1)$. Choose $n, K$ large enough so that $\max_{v \in V_{G_n, K}^-} d_v = d_{\ceil{Kr_n}+1} \leq \varepsilon r_n$, which can be done by Observation \ref{obs:label}. This implies $X_{v, M}=X_v$, for all $v \in V_{G_n, K}^-$, when $n, K$ are large enough, and by the the multinomial expansion,
\begin{eqnarray}
\E[({T_{n,K,M}^-})^a]&=&\sum_{(u_1, v_1)\in E(G_{n, K}^-)}\sum_{(u_2, v_2)\in E(G_{n, K}^-)}\cdots\sum_{(u_a, v_a)\in E(G_{n, K}^-)}\E\left[\prod_{s=1}^a X_{u_s} X_{v_s} \right],
\label{eq:t-a}\\
\E[({W_{n, K}^-})^a]&=&\sum_{(u_1, v_1)\in E(G_{n, K}^-)}\sum_{(u_2, v_2)\in E(G_{n, K}^-)}\cdots\sum_{(u_a, v_a)\in E(G_{n, K}^-)}\E\left[\prod_{s=1}^a R_{u_s v_s}\right].
\label{eq:w-m}
\end{eqnarray}
Now, let $H$ be the graph formed by the union of the edges $(u_1,v_1), (u_2, v_2), \ldots, (u_a, v_a)$. Then
\begin{align}\label{eq:tw-}
\E\left[\prod_{s=1}^a X_{u_s} X_{v_s} \right]=p_n^{|V(H)|} \quad \text{ and } \quad \E\left[\prod_{s=1}^{a} R_{u_s v_s}\right]=p_n^{2 |E(H)|}.
\end{align}
If $\cH_a$ denotes the set of all non-isomorphic graphs with at most $a$ edges and no isolated vertex, then \eqref{eq:t-a}, \eqref{eq:w-m}, and \eqref{eq:tw-} combined gives\footnote{For $a, b \in \R$, $a \lesssim_{\circ} b$ means that $a \leq C(\circ) b$, for some constant $C(\circ)$ depending only on the subscripted quantities in $\circ$.} 
\begin{align}\label{one}
\left|\E \left[({T_{n,K,M}^-})^a\right]  - \E \left[({W_{n, K}^-})^a\right]\right| & \lesssim_a \sum_{H \in \cH_a} N(H,G_{n, K}^-)\left|p_n^{|V(H)|}-p_n^{2|E(H)|}\right| \nonumber \\ 
& = \sum_{\substack{H \in \cH_a \\ |V(H)| < 2 |E(H)|}} N(H,G_{n, K}^-)\left|p_n^{|V(H)|}-p_n^{2|E(H)|}\right| \nonumber \\ 
& \le  \sum_{\substack{H \in \cH_a \\ |V(H)| < 2 |E(H)|}} N(H,G_{n, K}^-) p_n^{|V(H)|},
\end{align}
where the second and third steps use the fact that $|V(H)| \leq 2|E(H)|$, since graphs in $\cH_a$ have no isolated vertex.

Now, for any connected graph $H \in \cH_a$, 
\begin{align*}
N(H, G_{n, K}^-) \lesssim_a  |E(G_{n, K}^-)| \left(\max_{v \in V_{G_n, K}^-} d_v \right)^{|V(H)|-2} & \leq |E(G_{n, K}^-)|  \varepsilon^{|V(H)|-2}  r_n^{|V(H)|-2}  \tag*{(by Observation \eqref{obs:label})}\nonumber \\
& \lesssim \varepsilon^{|V(H)|-2} r_n^{|V(H)|},
\end{align*}
where the last step uses $|E(G_{n, K}^-)|\lesssim r_n^2$ by \eqref{eq:ETn}. 

Therefore, if $H \in \cH_a$ has $\nu(H)$ connected components, then using the above bound separately on each of the connected components gives,
\begin{align}\label{eq:NHGn-}
N(H,G_{n, K}^-) \lesssim_a \varepsilon^{|V(H)|-2\nu(H)}p_n^{-|V(H)|}.
\end{align}
Using the estimate above and (\ref{one}) gives, 
\begin{align}\label{two}
\left|\E \left[({T_{n,K,M}^-})^a\right]  - \E \left[({W_{n, K}^-})^a\right]\right| \lesssim_a  \sum_{\substack{H \in \cH_a \\ |V(H)| < 2 |E(H)|}} \varepsilon^{|V(H)|-2\nu(H)}. 
\end{align} 
Now, suppose $H \in \cH_a$ is such that $|V(H)| < 2 |E(H)|$. Note that all connected components of $H$ contains least 2 vertices, and at least one connected component of $H$ must contain at least 3 vertices (otherwise $H$ is a disjoint union of edges, and $|V(H)|=2 |E(H)|$). This implies, $|V(H)| > 2 \nu(H)$, where $\nu(H)$ is the number of connected components of $H$. Since $\varepsilon>0$ is arbitrary and cardinality of the set $\cH_a$ is fixed (free of $n$),  the RHS of \eqref{two} can be made arbitrarily small, and so the LHS of \eqref{two} converges to 0 under the double limit of $n$ goes to infinity followed by $K$ goes to infinity, which is the first desired result.

Finally, from \eqref{eq:t-a}, \eqref{eq:tw-}, and \eqref{eq:NHGn-} we have 
\begin{align*}
\E \left[({T_{n,K,M}^-})^a\right]  & \lesssim_a \sum_{H \in \cH_a} N(H,G_{n, K}^-)p_n^{|V(H)|}  \lesssim_a \sum_{H \in \cH_a} \varepsilon^{|V(H)|-2\nu(H)} \lesssim_a 1, \nonumber 
\end{align*}
because, as before, the sum is over a finite index set free of $n$. \end{proof}

\subsection{Independence in Moments of $(T_{n,K,M}^+, T_{n,K,M}^\pm)$ and $T_{n,K,M}^-$}
\label{sec:pfquadratic2}

In this section, we will show that the mixed moments $(T_{n,K,M}^+, T_{n,K,M}^\pm)$ and $T_{n,K,M}^-$ factorize in the limit.   

\begin{lem}\label{lm:independent_moment} Fix non-negative integers $a, b, c$ and $M > 0$. Then under the assumptions of Theorem \ref{thm:poisson_quadratic}, 
$$\lim_{K \to \infty} \limsup_{n\to \infty} \Big| \E\left[(T_{n,K,M}^+)^a (T_{n,K,M}^\pm)^b (T_{n,K,M}^-)^c\right] - \E \left[(T_{n,K,M}^+)^a (T_{n,K,M}^\pm)^b\right] \E\left[(T_{n,K,M}^-)^c\right] \Big| = 0.$$
\end{lem}

\begin{proof} Note that there is nothing to prove if $c=0$. Moreover, since $T_{n,K,M}^+$ and $T_{n,K,M}^-$ are independent for each $n$ and $K$ (they are defined on disjoint sets of vertices of $G_n$), the case $b=0$ follows trivially. Therefore, we assume $b$ and $c$ are both positive.

Let $\cE_{n,K}^{a, b, c}$ be the collection of $(a+b+c)$-tuples of the form $$\bm e=((u_1, v_1),\ldots ,(u_a, v_a), (u_1', v_1'), \ldots, (u_b', v_b') ,(u_1'', v_1''), \ldots , (u_c'', v_c'')),$$ where  
$(u_1, v_1),\ldots ,(u_a, v_a) \in E(G_{n, K}^+)$, $(u_1', v_1'),\ldots ,(u_b', v_b') \in E(G_{n, K}^\pm)$ (with $u_1',\ldots ,u_b' \in V_{G_n, K}^+$ and $v_1',\ldots , v_b' \in V_{G_n, K}^-$), and $(u_1'', v_1''),\ldots ,(u_c'', v_c'') \in E(G_{n, K}^-)$. Further, define $\cD_{n,K}^{a, b, c}$ as the set of all $\bm e \in \cE_{n,K}^{a, b, c}$ such that the sets $\{v_1', \ldots, v_b'\}$ and $\{u_1'', v_1'',\ldots ,u_a'', v_a''\}$ are disjoint. Then 
\begin{align}\label{eq:abc1}
&\E\left[(T_{n,K,M}^+)^a (T_{n,K, M}^\pm)^b (T_{n,K,M}^-)^c\right]\nonumber\\
&= \sum_{\bm e \in \cE_{n,K}^{a, b, c}} \E\left[ \prod_{s=1}^a X_{u_s, M} X_{v_s, M}  \prod_{s=1}^b X_{u_s', M} X_{v_s', M}   \prod_{s=1}^c X_{u_s'', M} X_{v_s'', M}  \right]\nonumber\\
&= \sum_{\bm e \in \cD_{n,K}^{a, b, c}} \E\left[ \prod_{s=1}^a X_{u_s, M} X_{v_s, M}  \prod_{s=1}^b X_{u_s', M} X_{v_s', M}  \right] \E \left[ \prod_{s=1}^c X_{u_s'', M} X_{v_s'', M}  \right]  \nonumber \\
&\;\;\;\;\;\;\;\;\;\;\;\;\;\;\;\;\;\;\;\;\;\;\;\;\;+\sum_{\bm e \in \cE_{n,K}^{a, b, c}\backslash \cD_{n,K}^{a, b, c}} \E\left[ \prod_{s=1}^a X_{u_s, M} X_{v_s, M}  \prod_{s=1}^b X_{u_s', M} X_{v_s', M}   \prod_{s=1}^c X_{u_s'', M} X_{v_s'', M}  \right]. 
\end{align}
On the other hand,
\begin{align}\label{eq:abc2}
&\E \left[(T_{n,K,M}^+)^a (T_{n,K, M}^\pm)^b\right] \E\left[(T_{n,K,M}^-)^c\right] \nonumber \\ 
&= \sum_{\bm e \in \cE_{n,K}^{a, b, c}} \E\left[ \prod_{s=1}^a X_{u_s, M} X_{v_s, M}  \prod_{s=1}^b X_{u_s', M} X_{v_s', M}  \right] \E\left[ \prod_{s=1}^c X_{u_s'', M} X_{v_s'', M}  \right] \nonumber\\
&= \sum_{\bm e \in \cD_{n,K}^{a, b, c}} \E\left[ \prod_{s=1}^a X_{u_s, M} X_{v_s, M}  \prod_{s=1}^b X_{u_s', M} X_{v_s', M}  \right] \E\left[ \prod_{s=1}^c X_{u_s'', M} X_{v_s'', M}  \right] \nonumber\\
&\;\;\;\;\;\;\;\;\;\;\;\;\;\;\; + \sum_{\bm e \in \cE_{n,K}^{a, b, c}\backslash \cD_{n,K}^{a, b, c}} \E\left[ \prod_{s=1}^a X_{u_s, M} X_{v_s, M}  \prod_{s=1}^b X_{u_s', M} X_{v_s', M}  \right] \E\left[ \prod_{s=1}^c X_{u_s'', M} X_{v_s'', M}  \right].
\end{align}
By taking the difference of \eqref{eq:abc1} and \eqref{eq:abc2} it follows that, in order to prove the lemma, it suffices to show the following two statements:
\begin{equation}\label{suff11}
\lim_{K \rightarrow \infty} \limsup_{n \rightarrow \infty} \sum_{\bm e \in \cE_{n,K}^{a, b, c}\backslash \cD_{n,K}^{a, b, c}} \E\left[ \prod_{s=1}^a X_{u_s, M} X_{v_s, M}  \prod_{s=1}^b X_{u_s', M} X_{v_s', M}   \prod_{s=1}^c X_{u_s'', M} X_{v_s'', M}  \right] = 0, 
\end{equation}
and 	 	
\begin{equation}\label{suff12}
\lim_{K \rightarrow \infty} \limsup_{n \rightarrow \infty} \sum_{\bm e \in \cE_{n,K}^{a, b, c}\backslash \cD_{n,K}^{a, b, c}} \E\left[ \prod_{s=1}^a X_{u_s, M} X_{v_s, M}  \prod_{s=1}^b X_{u_s', M} X_{v_s', M}  \right] \E\left[ \prod_{s=1}^c X_{u_s'', M} X_{v_s'', M}  \right] = 0.
\end{equation}

To this end, define $\cE_{n, K, M}^{a, b, c}$ to be set of all $\bm e \in \cE_{n, K}^{a, b, c}$, such that the following three conditions hold: (1) $\max\{d_{u_s}, d_{v_s}\} \leq  M r_n,~\textrm{for all}~s \in [a]$, (2) $\max\{d_{u_s'}, d_{v_s'}\} \leq  M r_n,~\textrm{for all}~s \in [b]$, and (3) $\max\{d_{u_s}'', d_{v_s}''\} \leq  M r_n,~\textrm{for all}~s \in [c]$. Let $\cD_{n,K, M}^{a,b,c} = \cE_{n, K, M}^{a, b, c} \bigcap \cD_{n,K}^{a,b,c} $. Then, (\ref{suff11}) becomes:
\begin{equation}\label{suff101}
\lim_{K \rightarrow \infty} \limsup_{n \rightarrow \infty}  \sum_{\bm e \in \cE_{n,K, M}^{a, b, c}\backslash \cD_{n,K, M}^{a, b, c}} \E\left[ \prod_{s=1}^a X_{u_s} X_{v_s}  \prod_{s=1}^b X_{u_s'} X_{v_s'}   \prod_{s=1}^c X_{u_s''} X_{v_s''}  \right] = 0, 
\end{equation}

If $H$ is the graph formed by the union of the edges $(u_1, v_1),\ldots ,(u_a, v_a)$, $(u_1', v_1'), \ldots, (u_b', v_b')$, $(u_1'', v_1''), \ldots , (u_c'', v_c'')$, then 
$$\E\left[ \prod_{s=1}^a X_{u_s} X_{v_s}  \prod_{s=1}^b X_{u_s'} X_{v_s'}   \prod_{s=1}^c X_{u_s''} X_{v_s''}  \right]=p_n^{|V(H)|}.$$
Note that this graph $H$ must have at  least two edges $(u_i', v_i')$ and $(u_j'', v_j'')$, such that $v_i'=u_j''$ or $v_i' = v_j''$, since for any $\bm e \in  \cE_{n,K, M}^{a, b, c}\backslash \cD_{n,K, M}^{a, b, c}$, the set $\{v_1', \ldots, v_b'\}$  intersect the set $\{u_1'', v_1'',\ldots ,u_a'', v_a''\}$, that is, $H$ has a 2-star $K_{1, 2}$, with central vertex in $V_{G_n,K}^-$. Let $V_{M}:=\{v \in V(G_n): d_v \leq Mr_n\}$ and 
$N_{a, b, c}(H, G_n[V_M])$ be the number ways of forming a graph isomorphic to $H$, with $a$ edges from $G_{n, K}^+[V_{M}]$, $b$ edges from $G_{n, K}^\pm[V_{M}]$, and $c$ edges from $G_{n, K}^-[V_{M}]$, such that the resulting graph contains a $K_{1, 2}$, with central vertex in $V_{G_n,K}^-$ (and one edge in $E(G_{n, K}^-[V_{M}])$ and the other in $E(G_{n, K}^\pm[V_{M}])$). Then
\begin{align}\label{eq:NH_III}
\sum_{\bm e \in \cE_{n,K, M}^{a, b, c}\backslash \cD_{n,K, M}^{a, b, c}} \E\left[ \prod_{s=1}^a X_{u_s} X_{v_s}  \prod_{s=1}^b X_{u_s'} X_{v_s'}   \prod_{s=1}^c X_{u_s''} X_{v_s''}  \right] \lesssim \sum_{H \in \cH_{a,b,c}} N_{a, b, c}(H, G_n[V_M]) p_n^{|V(H)|},
\end{align}
where $\cH_{a,b,c}$ is the set of all non-isomorphic graphs with at most $2(a+b+c)$ vertices, none of which is isolated, which contains at least one $K_{1, 2}$ (the 2-star) as a subgraph.

Now, we proceed to bound $N_{a, b, c}(H, G_n[V_M])$: Note that for any connected $F \in \cH_{a, b, c}$, 
\begin{align}
N_{a,b,c}(F, G_n[V_M])  \lesssim_{a, b, c}  |E(G_{n}[V_M])| (M r_n)^{|V(F)|-2} 
 \leq |E(G_n)| (M r_n)^{|V(F)|-2} 
\label{eq:NHabc} \lesssim_{M, F}  r_n^{|V(F)|}, 
\end{align} 
using \eqref{eq:ETn}.  Now, suppose $H\in \cH_{a, b, c}$ has connected components $H_1, H_2, \ldots, H_{\nu(H)}$, and without loss of generality, assume $H_1$ has a 2-star $K_{1, 2}$, with central vertex in $V_{G_n, K}^-$ (and one edge in $E(G_{n, K}^-[V_{M}])$ and the other in $E(G_{n, K}^\pm[V_{M}])$). Therefore, choosing this 2-star in at most $|E(G_{n}[V_M])|\cdot \max\{d_v: v \in V_{G_n, K}^- \}$ ways and each of the remaining $|V(H_1)|-3$ vertices in at most $M r_n$ ways gives the bound 
\begin{align}
N_{a, b, c}(H_1, G_n[V_M]) & \lesssim_{a, b, c} |E(G_{n}[V_M])| \left(\max_{v \in V_{G_n, K}^-} d_v\right) (M r_n)^{(|V(H_1)|-3)}  \nonumber \\
\nonumber& \lesssim_{a,b,c,M} |E(G_n)| r_n^{(|V(H_1)|-3)} \left(\max_{v \in V_{G_n, K}^-} d_v\right) \\ 
\label{eq:NH1abc} & \lesssim_{a,b,c,M} r_n^{(|V(H_1)|-1)} d_{\ceil{Kr_n}+1}, 
\end{align}
using \eqref{eq:ETn}. Now, combining \eqref{eq:NHabc} and \eqref{eq:NH1abc} gives 
\begin{align}\label{eq:NHabcII}
N_{a, b, c}(H, G_n[V_M]) \leq \prod_{j=1}^{\nu(H)} N_{a, b, c}(H_j, G_n[V_M]) \lesssim_{a, b, c, M, H} r_n^{(|V(H)|-1)}  d_{\ceil{Kr_n}+1}.
\end{align}
This implies 
$$\lim_{K\rightarrow \infty}\lim_{n \rightarrow \infty}\frac{1}{r_n^{|V(H)|}}N_{a, b, c}(H, G_n[V_M])=0,$$ by Observation \ref{obs:label}. Thus \eqref{suff101} follows, because the sum in the right hand side of \eqref{eq:NH_III} is over a finite set ($|\cH_{a,b,c}\big| \leq 2^{\binom{2(a+b+c)}{2}}$). The limit in  (\ref{suff12}) follows similarly, completing the proof of the lemma.
\end{proof}

\subsection{Moment Approximation for $T_{n,K,M}^\pm$}
\label{sec:pfquadratic3}

Let $\{J_{uv}\}_{(u, v)\in E(G_{n, K}^\pm)}$ be a collection of independent Bernoulli$(p_n)$ random variables, independent of the  collection $\{X_v\}_{v \in V(G_n)}$. Define 
\begin{align}\label{eq:Z}
Z_{n,K,M} = \sum_{u \in  V_{G_n, K}^+} \sum_{v \in  V_{G_n, K}^-} a_{uv}(G_n)  J_{uv} X_{u,M}.
\end{align}

\begin{lem}\label{lm:TZ} Fix non-negative integers $a, b$ and $M > 0$. Then under the assumptions of Theorem \ref{thm:poisson_quadratic}, 
$$\lim_{K \to \infty} \limsup_{n\to \infty} \Big|\E\left[(T_{n,K,M}^+)^a (T_{n,K,M}^\pm)^b\right] - \E \left[(T_{n,K,M}^+)^a Z_{n,K,M}^b\right] \Big| = 0.$$ 
Moreover, $\limsup_{K \rightarrow \infty} \limsup_{n \rightarrow \infty} \E \left[(T_{n,K,M}^+)^a Z_{n,K,M}^b\right] \lesssim_{a, b, M} 1$.
\end{lem}

\begin{proof} Note that there is nothing to prove if $b=0$, so we  assume that $b > 0$. Let $\cM_{n,K}^{a, b}$ the collection of $(a+b)$-tuples of the form $$\bm e=((u_1, v_1),\ldots ,(u_a, v_a), (u_1', v_1'), \ldots, (u_b', v_b')),$$ where  
$(u_1, v_1),\ldots ,(u_a, v_a) \in E(G_{n, K}^+[V_M])$ and $(u_1', v_1'),\ldots ,(u_b', v_b') \in E(G_{n, K}^\pm[V_M])$ (with $u_1',\ldots ,u_b' \in V_{G_n, K}^+$ and $v_1',\ldots , v_b' \in V_{G_n, K}^-$).

\begin{align}\label{eq:diff_Z}
&\left|\E\left[(T_{n,K,M}^+)^a (T_{n,K,M}^\pm)^b\right] - \E \left[(T_{n,K,M}^+)^a Z_{n,K,M}^b\right]\right| \nonumber\\
& \leq  \sum_{\bm e \in \cM_{n,K}^{a, b}} \left|\E\left[ \prod_{s=1}^a X_{u_s} X_{v_s}  \prod_{s=1}^b X_{u_s'} X_{v_s'}   \right] - \E\left[ \prod_{s=1}^a X_{u_s} X_{v_s}   \prod_{s=1}^b X_{u_s'} \prod_{s=1}^b J_{u_s'v_s'}   \right] \right| \nonumber \\
& =  \sum_{\bm e \in \cM_{n,K}^{a, b}} \left| p_n^{|V(H_1 \bigcup H_2)|}-  p_n^{|\{u_1, v_1, u_2, v_2, \ldots, u_a, v_a, u_1', u_2', \ldots, u_b'\}|+|E(H_2)|} \right|,
\end{align}
where $H_1$ is the graph formed by the union of the edges $(u_1, v_1),\ldots ,(u_a, v_a)$, and  $H_2$  is the graph formed by the union of the edges $(u_1', v_1'),\ldots, (u_b', v_b')$.

Note that 
\begin{align}\label{eq:VHbound}
|\{u_1, v_1, &u_2, v_2, \ldots, u_a, v_a, u_1', u_2', \ldots, u_b'\}|+|E(H_2)| \nonumber \\
& = |V(H_1 \bigcup H_2)|-|\{v_1', v_2', \ldots, v_b'\}|+|E(H_2)| \geq |V(H_1 \bigcup H_2)|,
\end{align}
using $|E(H_2)|\geq |\{v_1', v_2', \ldots, v_b'\}|$, since each distinct element in $\{v_1', v_2', \ldots, v_b'\}$ contributes an edge to $E(H_2)$. This implies 
\begin{align*} 
\left|\E\left[(T_{n,K,M}^+)^a (T_{n,K,M}^\pm)^b\right] - \E \left[(T_{n,K,M}^+)^a Z_{n,K,M}^b\right]\right|  \leq \sum_{\bm e \in \overline{\cM}_{n,K}^{a, b}}  p_n^{|V(H_1 \bigcup H_2)|},
\end{align*}
where $\overline{\cM}_{n,K}^{a, b} \subseteq \cM_{n,K}^{a, b}$  is the collection of all tuples in $\cM_{n,K}^{a, b}$ such that $$|\{u_1, v_1, u_2, v_2, \ldots, u_a, v_a, u_1', u_2', \ldots, u_b'\}|+|E(H_2)|> |V(H_1 \bigcup H_2)|.$$  

Now, suppose that for every $s$, $t \in [b]$ such that $v_s' = v_t'$, we also have $u_s' = u_t'$. Then, $|E(H_2)| = |\{v_1',\ldots,v_b'\}|$, and hence, equality holds in \eqref{eq:VHbound}. Therefore, $e \in \overline{\cM}_{n,K}^{a, b}$ implies that there exist $s,t \in [b]$ such that $v_s' = v_t'$ and $u_s' \neq u_t'$, that is, the graph $H:= H_1\bigcup H_2$ must have a $K_{1, 2}$ with central vertex in $V_{G_n, K}^-$. Recall that $V_{M}:=\{v \in V(G_n): d_v \leq Mr_n\}$ and denote by $N_{a, b}(H, G_n[V_M])$ the number of ways of forming a graph isomorphic to $H$, with $a$ edges from $G_{n, K}^+[V_{M}]$ and $b$ edges from $G_{n, K}^\pm[V_{M}]$, such that the result graph contains a $K_{1, 2}$, with central vertex in $V_{G_n,K}^-$ (and both edges in $E(G_{n, K}^\pm[V_{M}])$). Then
\begin{align}\label{eq:diff_H}
\left|\E\left[(T_{n,K,M}^+)^a (T_{n,K,M}^\pm)^b\right] - \E \left[(T_{n,K,M}^+)^a Z_{n,K,M}^b\right]\right|  \lesssim_{a,b} \sum_{H \in \cH_{a, b}} N_{a, b}(H, G_n[V_M]) p_n^{|V(H)|},
\end{align}
where $\cH_{a,b}$ is the set of all non-isomorphic graphs with at most $2(a+b)$ vertices, none of which is isolated, which contains at least one $K_{1, 2}$ (the 2-star) as a subgraph. Now, as in \eqref{eq:NHabcII},
$$N_{a, b}(H, G_n[V_M]) \lesssim_{a, b, M, H} r_n^{(|V(H)|-1)} d_{\ceil{K r_n}+1}.$$  This implies 
$$\lim_{K\rightarrow \infty}\lim_{n \rightarrow \infty}\frac{1}{r_n^{|V(H)|}}N_{a, b}(H, G_n[V_M])=0,$$ by Observation \ref{obs:label}. This completes the proof of the lemma, because the sum in the right hand side of \eqref{eq:diff_H} is over a finite set ($|\cH_{a,b}\big| \leq 2^{\binom{2(a+b)}{2}}$).

Finally, by similar arguments as above and \eqref{eq:diff_H}, 
$$\E \left[(T_{n,K,M}^+)^a Z_{n,K,M}^b\right]  \leq \sum_{H \in \cH_{a, b}} N_{a, b}(H, G_n[V_M]) p_n^{|V(H)|} \lesssim_{a, b, M} 1,$$
using the bound \eqref{eq:NHabc} and because the sum is over a finite set. 
\end{proof}

Combining the above results, we get the following proposition which shows that $(T_{n,K,M}^+, T_{n,K,M}^\pm, T_{n,K,M}^-)$ and $((T_{n,K, M}^+, Z_{n,K,M}), W_{n, K}^-)$ are close in moments. 
 	 
\begin{ppn}\label{ppn:abc}
Fix non-negative integers $a, b, c$ and for $M > 0$. Then under the assumptions of Theorem \ref{thm:poisson_quadratic}, 
$$\lim_{K \to \infty} \limsup_{n\to \infty} \left| \E\left[(T_{n,K,M}^+)^a (T_{n,K,M}^\pm)^b (T_{n,K,M}^-)^c\right]- \E\left[(T_{n,K, M}^+)^a Z_{n,K, M}^b\right]\E\left[(W_{n, K}^-)^c\right]\right| = 0.$$
\end{ppn}
	 	
\begin{proof} Note that 
\begin{align*}
\left| \E\left[(T_{n,K,M}^+)^a (T_{n,K,M}^\pm)^b (T_{n,K,M}^-)^c\right]- \E\left[(T_{n,K, M}^+)^a Z_{n,K, M}^b\right]\E\left[(W_{n, K}^-)^c\right]\right| \leq T_1+T_2+T_3,
\end{align*} 
where 
\begin{eqnarray*}
T_1&:=& \left|\E\left[(T_{n,K,M}^+)^a (T_{n,K,M}^\pm)^b (T_{n,K,M}^-)^c\right] - \E \left[(T_{n,K,M}^+)^a (T_{n,K,M}^\pm)^b\right] \E\left[(T_{n,K,M}^-)^c\right] \right| \\
T_2&:=& \left|  \E \left[(T_{n,K,M}^+)^a (T_{n,K,M}^\pm)^b\right] \E\left[(T_{n,K,M}^-)^c\right]   - \E\left[(T_{n,K, M}^+)^a Z_{n,K, M}^b\right]\E\left[(T_{n,K,M}^-)^c\right]\right|  \\
T_3 &:=& \left| \E\left[(T_{n,K, M}^+)^a Z_{n,K, M}^b\right] \E\left[(T_{n,K,M}^-)^c\right]  - \E\left[(T_{n,K, M}^+)^a Z_{n,K, M}^b\right] \E\left[(W_{n, K}^-)^c\right] \right|. 
\end{eqnarray*}			
Now, $T_1$ goes to zero (under the double limit) by Lemma \ref{lm:independent_moment}, $T_2$ goes to zero by Lemma \ref{lm:TZ} (and using $\limsup_{K \rightarrow \infty} \limsup_{n \rightarrow \infty}\E \left[({T_{n,K,M}^-})^a\right] \lesssim_a 1$ by Lemma \ref{lm1}), and $T_3$ goes to zero by Lemma \ref{lm1} (and using $\limsup_{K \rightarrow \infty} \limsup_{n \rightarrow \infty} \E \left[(T_{n,K,M}^+)^a Z_{n,K,M}^b\right]  \lesssim_{a, b, M} 1$ by Lemma \ref{lm:TZ}). 
\end{proof}

\subsection{Convergence of Moments and Existence of Limiting Distribution}
\label{sec:pfquadratic4}

Recall from \eqref{eq:Z}
\begin{align}\label{eq:Z_I}
Z_{n,K,M} = \sum_{u \in  V_{G_n, K}^+} \sum_{v \in  V_{G_n, K}^-} a_{uv}(G_n)  J_{uv} X_{u,M}\stackrel{D}=  \sum_{u \in  V_{G_n, K}^+} J_u X_{u, M},
\end{align}
where $J_u := \sum_{v \in  V_{G_n, K}^-} a_{uv}(G_n)  J_{uv} \sim \dBin(d_u^\pm, p_n)$ and $d_u^\pm:=\sum_{v \in  V_{G_n, K}^-} a_{uv}(G_n)$, is the number of edges between $u \in V_{G_n, K}^+$ and some vertex in $V_{G_n, K}^-$. Note that, by definition, $\{J_u\}_{u \in V_{G_n, K}^+}$ is a collection of independent Binomial random variables, independent of $\{X_v\}_{v \in V(G_n)}$. Next, define,  $V_{n, K, M}^+ := \{ v \in V_{G_n,K}^+ : d_v\leq M r_n\}$. Then \eqref{eq:Z} becomes (recall $X_{v, M} = X_v \bm{1}\{d_v \leq M r_n\}$ and $p_n=1/r_n$), 
\begin{align}\label{eq:Z_II}
Z_{n,K,M} =  \sum_{u \in  V_{n, K, M}^+} J_u X_u \sim \dBin\left(\sum_{u \in V_{n, K, M}^+} d_u^{\pm} X_u, \frac{1}{r_n}\right). 
\end{align} 
Define $Y_{n,K,M} = \frac{1}{r_n} \sum_{u \in V_{n, K, M}^+} d_u^{\pm} X_u$. The lemma below shows the existence of the limiting mixed moments of $(T_{n,K,M}^+, Y_{n,K,M})$. We begin with the following definition. Hereafter, we will assume that $K \geq 1$ is an integer.  Also, denote by $\cW_K$ the set of all symmetric measurable functions from $[0, K]^2 \rightarrow [0, 1]$. With these definitions, we now have the convergence of the mixed moments.

\begin{lem}\label{lm:moment_limit}
Fix $K \geq 1$ and $M>0$. Suppose there exist functions $W_K \in \cW_K$, $d_K: [0, K] \mapsto [0,\infty)$, and measure-preserving bijections $\{\phi_{n, K}\}_{n=1}^\infty$ from $[0,K] \rightarrow [0, K]$, such that 
\begin{align}\label{eq:W_phi_K}
\lim_{n \rightarrow \infty}||W_{G_{n}}^{\phi_{n, K}}-W_K||_{\square([0, K]^2)}=0 
\end{align} 
and 
\begin{align}\label{eq:L1_M}
\lim_{n \rightarrow \infty}\int_0^K \left| d_{W_{G_n}}^{\phi_{n, K}}(u) \bm 1\{d_{W_{G_n}}^{\phi_{n, K}}(u) \leq M\} - d_K(u) \bm 1\{d_K(u) \leq M\} \right| \mathrm d u=0,
\end{align}
where $W_{G_{n}}^{\phi_{n, K}}(x, y):=W_{G_n}(\phi_{n, K}(x), \phi_{n, K}(y)) \bm 1\{ x, y \in [0, K]\}$ and  $d_{W_{G_n}}^{\phi_{n, K}}(x) := d_{W_{G_n}}(\phi_{n, K}(x)) \bm 1\{x \in [0, K]\}$. Then $\mu_{a, b, K, M}:=\lim_{n\to \infty} \E\left[(T_{n,K,M}^+)^a (Y_{n,K,M})^b\right]$ exists and is finite, for all non-negative integers $a, b$.\footnote{
In the proof of Theorem \ref{thm:poisson_quadratic}, this lemma will be used with $\phi_{n, K}$ as the identity map from $[0, K] \rightarrow [0, K]$, for all $n$. However, we will need the lemma in its generality for proving Theorem \ref{thm:distributional_limit}.} 
\end{lem}

\subsubsection{Proof of Lemma \ref{lm:moment_limit}} We begin with the following definition:

\begin{defn}\label{defn:tHWu}
Given a graph $H=(V(H), E(H))$ (with possible isolated vertices), a function $W \in \cW$, and $\bm u \in \R^{|V(H)|}$ define  
$$t(H, W, \bm u)=\prod_{(a, b) \in E(H)} W(u_a, u_b),$$
where $\bm u=(u_1, u_2, \ldots, u_{|V(H)|})$. 
\end{defn}

Now, recall from \eqref{eq:Tn13defn}, $T_{n, K, M}^+:=\sum_{u, v \in V_{n, K, M}^+}a_{uv}(G_n) X_u X_v$. Let $\cN_{n,K}^{a, b}$ be the collection of all $(a+b)$-tuples of the form
$$\bm e=((u_1, v_1),\ldots ,(u_a, v_a), u_1', \ldots, u_b'),$$ where  $u_1, v_1,\ldots ,u_a, v_a, u_1', \ldots, u_b' \in V_{G_n, K}^+$, and $u_i \ne v_i$, for $1 \leq i \leq a$. Define the event, 
$$\chi_{\bm e}=\bm  1\{d_{u_i} \leq M r_n, d_{v_i} \leq Mr_n, d_{u_j'} \leq Mr_n, \text{ for all } 1 \leq i \leq a, 1 \leq j \leq b\},$$ 
where $d_v$ is the degree of the vertex labelled $v \in V(G_n)$. Then, recalling $Y_{n,K,M} = \frac{1}{r_n} \sum_{u \in V_{n, K, M}^+} d_u^{\pm} X_u$,   
\begin{align}\label{eq:diff_Z_I}
\E\left[(T_{n,K,M}^+)^a (Y_{n,K,M})^b \right] &= \sum_{\bm e \in \cN_{n,K}^{a, b}} \E\left[ \prod_{s=1}^a a_{u_s v_s}(G_n) X_{u_s} X_{v_s}  \prod_{s=1}^b \frac{d_{u_s'}^{\pm}}{r_n} X_{u_s'}   \right] \chi_{\bm e}\nonumber \\
&= \sum_{\bm e \in \cN_{n,K}^{a, b}} \frac{\chi_{\bm e}}{r_n^{|V(H)|}} \prod_{s=1}^a a_{u_sv_s}(G_n) \prod_{s=1}^b \frac{d_{u_s'}^{\pm}}{r_n}, 
\end{align} 
where $H$ is the graph formed by the union of the edges $(u_1, v_1), (u_2, v_2), \ldots, (u_a, v_a)$  and the vertices $u_1', u_2', \ldots, u_b'$. 

Observe that 
$$d_{u_s'}^{\pm}= d_{u_s'}-\sum_{j\in V_{G_n, K}^+} a_{u_s'j}(G_n).$$
Since $u_s' \in V_{G_n, K}^+$, there exists $x_s' \in [0, K]$ such that $u_s'=\ceil{x_s' r_n}$. This implies by \eqref{eq:WGn} and \eqref{eq:normalized_degree}, 
\begin{align}\label{eq:diff_degree}
\frac{1}{r_n}d_{u_s'}^{\pm}&= d_{W_{G_n}}(x_s')-\int_0^K W_{G_n}(x_s', y)\mathrm dy + R_n(x_s')=\zeta_{n, K}(x_s')  + R_n(x_s'),  
\end{align}
where $\zeta_{n, K}(x):=(d_{W_{G_n}}(x)-\int_0^K W_{G_n}(x, y) \mathrm dy)$ and 
\begin{align}\label{eq:integral_xs}
R_n(x_s') := - \int_{K}^{\frac{\ceil{Kr_n}}{r_n}} W_{G_n}(x_s',y) \mathrm dy.
\end{align} 
Note that $\sup_{x}|R_n(x)| \leq p_n .$ Similarly, let  $x_s, y_s \in [0, K]$ be such that $u_s=\ceil{x_s r_n}$ and $v_s=\ceil{y_s r_n}$. Then (recall \eqref{eq:WGn})
\begin{align}\label{eq:adj_Gn}
 \prod_{s=1}^a a_{u_sv_s}(G_n) = \prod_{s=1}^a a_{\ceil{x_s r_n}\ceil{y_s r_n}}(G_n)=\prod_{s=1}^a W_{G_n}(x_s, y_s).
\end{align}

Now, observe that the union of the edges $(u_1, v_1),\ldots, (u_a, v_a)$ forms a graph $H_1=(V(H_1), E(H_1))$, where $V(H_1)=\{u_1, \ldots, u_a, v_1, \ldots, v_a \}$ and $E(H_1)=\{(u_1, v_1),\ldots, (u_a, v_a)\}$. Let $H=(V(H), E(H))$ be the graph obtained by the union of $H_1$ and the set of vertices $\{u_1', \ldots, u_b'\}$, that is, $V(H)=V(H_1) \cup \{u_1', \ldots, u_b'\}$ and $E(H)=E(H_1)$. Note that $H$ has at most $a$ edges and at most $b$ isolated vertices. Let $\{w_1, w_2, \ldots, w_{|V(H)|}\}$ be any labeling of the vertices in $V(H)$, and   $\eta_{j} \in [0, b]:=\{0, 1, 2,  \ldots, b\}$, for $1 \leq j \leq |V(H)|$, be the number of times the vertex $w_j$ appears in the multi-set $\{u_1', u_2', \ldots, u_b'\}$. Finally, let $z_j$ be such that $w_j=\ceil{z_j r_n}$. Then using  \eqref{eq:diff_degree} and \eqref{eq:adj_Gn}, for every graph $H$ with at most $a$ edges and at most $b$ isolated vertices and every vector $\bm \eta=(\eta_1, \eta_2, \ldots, \eta_{|V(H)|})$, there is a non-negative constant $c(H, \bm \eta)$, such that the sum in \eqref{eq:diff_Z_I} can be rewritten as  
\begin{align}\label{eq:diff_Z_II}
&\E\left[(T_{n,K,M}^+)^a (Y_{n,K,M})^b \right] \nonumber \\
&= \sum_{\substack{H \in \sG_{a, b} \\ \bm \eta \in [0, b]^{|V(H)|}}} c(H, \bm \eta)  \int_{\sB_{K, n}} t(H, W_{G_n}, \bm z) \prod_{j=1}^{|V(H)|} (\zeta_{n, K}(z_j) + R_n(z_j))^{\eta_j}\bm  1\{d_{W_{G_n}}(z_j) \leq M\} \mathrm d z_j,
\end{align} 
where 
\begin{itemize}

\item[--] $\sB_{K, n}:=[0,\frac{1}{r_n} \ceil{K r_n}]^{|V(H)|}$,  

\item[--] $\bm z=(z_1, z_2, \ldots, z_{|V(H)|})$, 
 
\item[--] $\zeta_{n, K}(\cdot)$ is as in \eqref{eq:diff_degree}, $R_n(\cdot)$ as in \eqref{eq:integral_xs}, and  $t(H, W_{G_n}, \bm z)$ as in Definition \ref{defn:tHWu},

\item[--] $\sG_{a, b}$ is the collection of all graphs with at most $a$ edges and at most $b$ isolated vertices and $[0, b]^{|V(H)|}:=\{0, 1, 2, \ldots, b\}^{|V(H)|}$.
\end{itemize}

Note, since the sum in \eqref{eq:diff_Z_II} is over  a finite set (not depending on $n$) and  each term in the integrand is bounded (by a function of $H$, $K$, and $M$), the integral over $\sB_{K, n}$ can be replaced by the integral over $\sB_K=[0, K]^{|V(H)|}$, as $n \rightarrow \infty$. Moreover, for every $1 \leq j \leq |V(H)|$, expanding the term $(\zeta_{n, K}(z_j) + R_n(z_j))^{\eta_j}$ in \eqref{eq:diff_Z_II} by the binomial theorem, and using the fact $\sup_{x}|R_n(x)| \leq p_n =o(1)$, the proof of Lemma \ref{lm:moment_limit} follows from Lemma \ref{lm:WH} below. 

\begin{lem} Under the assumptions of Lemma \ref{lm:moment_limit}, given a finite simple graph $H=(V(H), E(H))$ (with possible isolated vertices), and the non-negative integers $s_1, s_2, \ldots, s_{|V(H)|}$, 
\begin{align*}
\lim_{n \rightarrow \infty}\int_{[0,K]^{|V(H)|}} & t(H, W_{G_n}, \bm u) \prod_{a=1}^{|V(H)|}\zeta_{n, K}(u_a)^{s_a}\bm  1\{d_{W_{G_n}}(u_a) \leq M\} \mathrm d u_a \\
& = \int_{[0,K]^{|V(H)|}}  t(H, W_K, \bm u) \prod_{a=1}^{|V(H)|}\zeta_{K}(u_a)^{s_a}\bm  1\{d_K(u_a) \leq M\} \mathrm d u_a,
\end{align*}
where $\zeta_{n, K}(x):=(d_{W_{G_n}}(x)-\int_0^K W_{G_n}(x, y) \mathrm dy)$ and $\zeta_{K}(x):=(d_K(x)-\int_0^K W_K(x, y) \mathrm dy)$. 
\label{lm:WH}
\end{lem}

\begin{proof} Define $d_{W_{G_n}, K}(x):=\int_0^K W_{G_n}(x, y) \mathrm dy$ and $d_{W, K}(x):=\int_0^K W_K(x, y) \mathrm dy$. Expanding $\zeta_{n, K}(u_a)^{s_a}=(d_{W_{G_n}}(u_a)-d_{W_{G_n}, K}(u_a))^{s_a}$ by the binomial theorem, for every $1 \leq a \leq |V(H)|$, we see it suffices to show that (recall $\sB_K:=[0, K]^{|V(H)|}$), 
\begin{align}\label{eq:integral_I}
\lim_{n \rightarrow \infty}& \int_{\sB_K} t(H, W_{G_n}, \bm u) \prod_{a=1}^{|V(H)|}d_{W_{G_n}}(u_a)^{\kappa_a} d_{W_{G_n}, K}(u_a)^{\lambda_a}  \bm  1\{d_{W_{G_n}}(u_a) \leq M\} \mathrm d u_a \nonumber \\
& = \int_{\sB_K} t(H, W_K, \bm u) \prod_{a=1}^{|V(H)|}d_K(u_a)^{\kappa_a} d_{W, K}(u_a)^{\lambda_a}  \bm  1\{d_K(u_a) \leq M\} \mathrm d u_a,
\end{align}
for non-negative integers $\kappa_1, \kappa_2, \ldots, \kappa_{|V(H)|}$ and $\lambda_1, \lambda_2, \ldots, \lambda_{|V(H)|}$.

To begin with, define $d_{W_{G_n}, K}^{\phi_{n, K}}(x):=\int_0^K W_{G_n}(\phi_{n, K}(x), y) \mathrm dy=\int_0^K W_{G_n}(\phi_{n, K}(x), \phi_{n, K}(z)) \mathrm dz$, where the last equality follows by the change of variable $y=\phi_{n, K}(z)$. This implies, 
\begin{align}\label{eq:integral_permutation}
& \int_{\sB_K} t(H, W_{G_n}, \bm u) \prod_{a=1}^{|V(H)|}d_{W_{G_n}}(u_a)^{\kappa_a} d_{W_{G_n}, K}(u_a)^{\lambda_a}  \bm  1\{d_{W_{G_n}}(u_a) \leq M\} \mathrm d u_a \nonumber \\ 
& = \int_{\sB_K}  t(H, W_{G_n}^{\phi_{n, K}}, \bm z) \prod_{a=1}^{|V(H)|}d_{W_{G_n}}^{\phi_{n, K}}(z_a)^{\kappa_a} d_{W_{G_n}, K}^{\phi_{n, K}}(z_a)^{\lambda_a}  \bm  1\{d_{W_{G_n}}^{\phi_{n, K}}(z_a) \leq M\} \mathrm d z_a, 
\end{align}
by changes of variables $u_a=\phi_{n, K}(z_a)$, for $1 \leq a \leq |V(H)|$, where $\bm u=(u_1, u_2, \ldots, u_{|V(H)|})$ and  $\bm z=(z_1, z_2, \ldots, z_{|V(H)|})$. Therefore, by \eqref{eq:integral_I}, it suffices to show that 
\begin{align}\label{eq:integral_II}
\lim_{n \rightarrow \infty}&  \int_{\sB_K}  t(H, W_{G_n}^{\phi_{n, K}}, \bm z) \prod_{a=1}^{|V(H)|}d_{W_{G_n}}^{\phi_{n, K}}(z_a)^{\kappa_a} d_{W_{G_n}, K}^{\phi_{n, K}}(z_a)^{\lambda_a}  \bm  1\{d_{W_{G_n}}^{\phi_{n, K}}(z_a) \leq M\} \mathrm d z_a \nonumber \\ 
& = \int_{\sB_K} t(H, W_K, \bm u) \prod_{a=1}^{|V(H)|}d_K(u_a)^{\kappa_a} d_{W, K}(u_a)^{\lambda_a}  \bm  1\{d_K(u_a) \leq M\} \mathrm d u_a. 
\end{align}

Now, denote $d_{W_{G_n}}^{\phi_{n, K}}(u_a|M):= d_{W_{G_n}}^{\phi_{n, K}}(u_a) \bm  1\{d_{W_{G_n}}^{\phi_{n, K}}(u_a) \leq M\}$. Then by a telescoping argument similar to the proof of \cite[Theorem 3.7(a)]{graph_limits_I}, it follows that 
\begin{align}\label{eq:W}
& \left|\int_{\sB_K} \Big( t(H, W_{G_n}^{\phi_{n, K}}, \bm u) -  t(H, W_K, \bm u) \Big) \prod_{a=1}^{|V(H)|} d_{W_{G_n}, K}^{\phi_{n, K}}(u_a)^{\lambda_a}  d_{W_{G_n}}^{\phi_{n, K}}(u_a|M)^{\kappa_a} \mathrm d u_a \right| \nonumber \\
& \lesssim_{M,H,K} ||W_{G_n}^{\phi_{n, K}}-W_K||_{\square([0, K]^2)}. 
\end{align} 
Moreover, recalling $d_{W_{G_n}, K}^{\phi_{n, K}}(x):=\int_0^K W_{G_n}(\phi_{n, K}(x), y) \mathrm dy=\int_0^K W_{G_n}(\phi_{n, K}(x), \phi_{n, K}(z)) \mathrm dz$, and from the definition of the cut-distance, 
$$\sup_{f:[0,K]\rightarrow [-1,1]}\left|\int_0^K (d_{W_{G_n}, K}^{\phi_{n, K}}(u)-d_{W, K}(u))f(u)\mathrm du \right| \leq ||W_{G_n}^{\phi_{n, K}}-W_K|||_{\square([0, K]^2)}.$$
Then, for any integer $a \geq 1$ and all $f:[0,K]\rightarrow [-1,1]$, 
\begin{align}
& \left|\int_0^K (d_{W_{G_n}, K}^{\phi_{n, K}}(u)^a-d_{W, K}(u)^a)f(u)\mathrm du \right| \nonumber \\ 
& \hspace{0.5in}  \leq \left|\int_0^K (d_{W_{G_n}, K}^{\phi_{n, K}}(u)-d_{W, K}(u)) d_{W_{G_n}, K}^{\phi_{n, K}}(u)^{a-1} f(u)\mathrm du \right| \nonumber \\
& \hspace{0.95in} + \left|\int_0^K (d_{W_{G_n}, K}^{\phi_{n, K}}(u)^{a-1}-d_{W, K}(u)^{a-1}) d_{W, K}(u) f(u)\mathrm du \right| \nonumber \\ 
& \hspace{0.5in} \lesssim_{M, H, K} ||W_{G_n}^{\phi_{n, K}}-W_K|||_{\square([0, K]^2)}.
\end{align}
where the last step follows by repeating the telescoping argument $a-1$ times. Now, repeating this telescoping argument again gives, 
\begin{align}\label{eq:K}
&\left|\int_{\sB_K} t(H, W_K, \bm u) \left(\prod_{a=1}^{|V(H)|} d_{W_{G_n}, K}^{\phi_{n, K}}(u_a)^{\lambda_a} -  \prod_{a=1}^{|V(H)|} d_{W, K}(u_a)^{\lambda_a} \right) \prod_{a=1}^{|V(H)|}d_{W_{G_n}}^{\phi_{n, K}}(u_a|M)^{\kappa_a}   \mathrm d u_a \right| \nonumber \\
& \lesssim_{M,H,K} ||W_{G_n}^{\phi_{n, K}}-W_K||_{\square([0, K]^2)}.
\end{align} 
Hence, combining \eqref{eq:W} and \eqref{eq:K}, and the triangle inequality gives, 
\begin{align}\label{eq:WK_I}
& \left|\int_{\sB_K} \left( t(H, W_{G_n}^{\phi_{n, K}}, \bm u)  \prod_{a=1}^{|V(H)|} d_{W_{G_n}, K}^{\phi_{n, K}}(u_a)^{\lambda_a} - t(H, W_K, \bm u)  \prod_{a=1}^{|V(H)|} d_{W, K}(u_a)^{\lambda_a} \right) \prod_{a=1}^{|V(H)|}d_{W_{G_n}}^{\phi_{n, K}}(u_a|M)^{\kappa_a}   \mathrm d u_a \right|  \nonumber \\
& \lesssim_{M,H,K} ||W_{G_n}^{\phi_{n, K}}-W_K||_{\square([0, K]^2)}. 
\end{align} 
Note that the RHS above goes to zero as $n \rightarrow \infty$, by \eqref{eq:W_phi_K}. 

Next, define  $d_K(u_a|M):= d_K(u_a) \bm  1\{d_K(u_a) \leq M\}$, and note that 
\begin{align}\label{eq:WK_II}
& \left|\int_{\sB_K}  t(H, W_K, \bm u)  \prod_{a=1}^{|V(H)|} d_{W, K}(u_a)^{\lambda_a} \left(\prod_{a=1}^{|V(H)|}d_{W_{G_n}}^{\phi_{n, K}}(u_a|M)^{\kappa_a}  - \prod_{a=1}^{|V(H)|} d_K(u_a|M)^{\kappa_a} \right) \prod_{a=1}^{|V(H)|} \mathrm d u_a \right|  \nonumber \\ 
& \lesssim_{M, H, K} \int_{\sB_K}   \left| \prod_{a=1}^{|V(H)|}d_{W_{G_n}}^{\phi_{n, K}}(u_a|M)^{\kappa_a}  - \prod_{a=1}^{|V(H)|}d_K(u_a|M)^{\kappa_a}  \right| \prod_{a=1}^{|V(H)|} \mathrm d u_a \nonumber \\ 
& \lesssim_{M, H, K}   \int_0^K \left| d_{W_{G_n}}^{\phi_{n, K}}(u|M)  - d_K(u|M) \right|  \mathrm d u, 
\end{align} 
where the last step follows by a telescoping argument (similar to Observation \ref{obs:degree_M} below). Note that RHS above goes to zero as $n \rightarrow \infty$, by \eqref{eq:L1_M}.  

Therefore, taking limit as $n \rightarrow \infty$, and combining \eqref{eq:WK_I} and \eqref{eq:WK_II}, and the triangle inequality, gives \eqref{eq:integral_II}, as required. \end{proof}

The next observation shows that a sufficient condition for \eqref{eq:L1_M} to hold infinitely often is the $L_1$ convergence of the function $d_{W_{G_n}}^{\phi_{n, K}}$ to $d_K$.  To this end, we need a definition. 

\begin{defn}\label{defn:M_points}
Let $\cD$ denote the set of all positive reals $M$ such that for all positive integers $K$, $\P(d_K(U_K)=M)=0$, where $U_K\sim \mathrm{Unif}[0,K]$. 
\end{defn}

Note that the complement of $\cD$ in $(0,\infty)$ is countable, and so given any $M_0>0$ we can choose $M>M_0$ with $M\in \cD$. 

\begin{obs}\label{obs:degree_M} Suppose 
\begin{align}\label{eq:degree_L1}
\lim_{n \rightarrow \infty} \int_0^K \left|d_{W_{G_n}}^{\phi_{n, K}}(x) - d_K(x)\right| \mathrm dx = 0.
\end{align}
Then any integer $a \geq 1$, 
\begin{align}\label{eq:deg_M}
\lim_{n \rightarrow 0}\int_0^K \left| d_{W_{G_n}}^{\phi_{n, K}}(u)^a \bm 1\{d_{W_{G_n}}^{\phi_{n, K}}(u) \leq M\} - d_K(u)^a \bm 1\{d_K(u) \leq M\} \right| \mathrm d u=0,
\end{align}
whenever $M \in \cD$. 
\end{obs}

\begin{proof} To begin with suppose $a=1$. The assumption \eqref{eq:degree_L1} implies,  $d_{W_{G_n}}^{\phi_{n, K}}(U) \stackrel{L_1} \rightarrow d_K(U)$, where $U \sim \mathrm{Unif}[0, K]$. Note that $\E(d_K(U)) < \infty$.\footnote{To observe this, note that,  for all $K\geq 1$, $\int_0^Kd_K(x) \mathrm dx=\limsup_{n\rightarrow\infty}\int_0^K d_{W_{G_n}}^{\phi_{n, K}}(x) \mathrm dx\lesssim \limsup_{n\rightarrow\infty}\frac{|E(G_n)|}{r_n^2}$, 
which is bounded by \eqref{eq:ETn}.} Then for every sequence there is a further subsequence $\{n_s\}_{s \geq 1}$ along which $$d_{W_{G_{n_s}}}^{\phi_{n_s}, K}(U) \stackrel{a.s.} \rightarrow d_K(U).$$ Hence, along this subsequence, $d_{W_{G_{n_s}}}^{\phi_{n_s}, K}(U) \bm 1\{d_{W_{G_{n_s}}}^{\phi_{n_s}, K}(U) \leq M \}\stackrel{a.s.} \rightarrow d_K(U) \bm 1\{d_K(U) \leq M\}$, whenever $\P(d_K(U)=M)=0$, that is, $M \in \cD$. Then by the dominated convergence theorem, $$d_{W_{G_{n_s}}}^{\phi_{n_s}, K}(U) \bm 1\{d_{W_{G_{n_s}}}^{\phi_{n_s}, K}(U) \leq M\}\stackrel{L_1} \rightarrow d_K(U) \bm 1\{d_K(U) \leq M\},$$ 
proving \eqref{eq:deg_M} for $a=1$. 

For $a > 1$, a telescoping argument gives, 
\begin{align*}
& \int_0^K \left| d_{W_{G_n}}^{\phi_{n, K}}(u)^a \bm 1\{d_{W_{G_n}}^{\phi_{n, K}}(u) \leq M\} - d_K(u)^a \bm 1\{d_K(u) \leq M\} \right| \mathrm d u \\
& \leq \int_0^K d_{W_{G_n}}^{\phi_{n, K}}(u)^{a-1}  \bm 1\{d_{W_{G_n}}^{\phi_{n, K}}(u)  \leq M\}  \left| d_{W_{G_n}}^{\phi_{n, K}}(u)  \bm 1\{d_{W_{G_n}}^{\phi_{n, K}}(u) \leq M\} - d_K(u) \bm 1\{d_K(u) \leq M\} \right| \mathrm d u_a \\ 
& \hspace{0.4in} + \int_0^K \left| d_{W_{G_n}}^{\phi_{n, K}}(u)^{a-1} d_K(u)  \bm 1\{d_{W_{G_n}}^{\phi_{n, K}}(u), d_K(u) \leq M\} - d_K(u)^a \bm 1\{d_K(u) \leq M\} \right| \mathrm d u \\ 
& \lesssim_{M, a} \int_0^K \left| d_{W_{G_n}}^{\phi_{n, K}}(u) \bm 1\{d_{W_{G_n}}^{\phi_{n, K}}(u) \leq M\} - d_K(u) \bm 1\{d_K(u) \leq M\} \right| \mathrm d u \rightarrow 0,
\end{align*}
where the second inequality follows by repeating the telescoping argument from the previous  step $a-1$ times, and the last step uses  \eqref{eq:deg_M} for $a=1$. 
\end{proof}

\subsubsection{Existence of Limit of $(T_{n,K,M}^+, Y_{n,K,M})$} The existence of the limiting distribution of $(T_{n,K,M}^+,$ $Y_{n,K,M})$ follows from the above lemma and  the Stieltjes moment condition.

\begin{lem}\label{lm:limit_existence} Suppose the assumptions of Lemma \ref{lm:moment_limit} hold. Then there exists random variables $(T_{K, M}^{+}, Y_{K, M})$ such that, as $n \rightarrow \infty$, 
$$(T_{n,K,M}^+, Y_{n,K,M}) \rightarrow (T_{K, M}^{+}, Y_{K, M}),$$ 
in distribution and in all (mixed) moments. 
\end{lem}
	
\begin{proof} Recall that $V_{n, K, M}^+ := \{ v \in V_{G_n,K}^+ : d_v\leq M r_n\}$. Then from \eqref{eq:Z_II},
 $$Y_{n,K,M} = \frac{1}{r_n} \sum_{u \in V_{n, K, M}^+} d_u^{\pm} X_u \leq M \sum_{u \in V_{n, K, M}^+} X_u \sim M \dBin(|V_{n, K, M}^+|,  p_n).$$ Note that $|V_{n, K, M}^+| \leq \ceil{K/p_n}$, which implies that $Y_{n,K,M}$ is stochastically dominated by the random variable $M \dBin(\ceil{K/p_n},  p_n)$. This implies, since $\mu_{0, b, K, M}= \lim_{n \rightarrow \infty}\E[Y_{n,K, M}^b]$ exists (by Lemma \ref{lm:moment_limit}), for all $b\ge 1$, 
 \begin{equation}\label{eq:moment_b}
\mu_{0, b, K, M} = \lim_{n \rightarrow \infty}\E[Y_{n,K, M}^b] \leq (MCKb)^b,
\end{equation}
using the bound $\E[\dBin(n, p)^a] \leq C^a \left(\frac{a}{\log a} \right)^a \max \{np, (np)^a\} \leq  C^a  a^a \max \{np, (np)^a\}$, for $a \ge 3$ and some universal constant $C<\infty$, \cite[Corollary 3]{latala}.

Next, define $S_{n,K, M} = \frac{1}{2}\sum_{u\in V_{n, K, M}^+}\sum_{v\in V_{n, K, M}^+\setminus \{u\}} X_u X_v$. Then, $$S_{n,K, M} \stackrel{D}{=} \binom{R_{n,K,M}}{2}, \quad \text{ where } R_{n,K,M} \sim \dBin(|V_{n, K, M}^+|,p_n),$$ and $T_{n,K,M}^+ \leq S_{n,K, M}$. Again using $|V_{n, K, M}^+| \leq \ceil{K r_n}$ and Lemma \ref{lm:moment_limit}, it follows that,  for all $a\ge 3$, 
\begin{equation}\label{eq:moment_a}
\mu_{a, 0, K, M}= \lim_{n\rightarrow \infty} \E[({T_{n,K, M}^+})^a] \leq \limsup_{n\rightarrow \infty}\E[R_{n,K,M}^{2a}]  \leq (2 CKa)^{2a},
\end{equation}
using bounds on moments of the binomial distribution \cite[Corollary 3]{latala}, as in \eqref{eq:moment_b}.

Combining \eqref{eq:moment_b} and \eqref{eq:moment_a} gives, 
$$\sum_{a=1}^\infty \frac{1}{\left(\mu_{a, 0, K, M}+\mu_{0, a, K, M}\right)^{\frac{1}{2a}}} \geq \frac{1}{\sqrt{2}}\sum_{a=1}^\infty \frac{1}{\max\{2CKa, \sqrt{MCKa}\}}=\infty.$$
Therefore, by the Stieltjes condition for multivariate distributions \cite[Page 21]{stieltjes_multivariate} (recall that the existence of the limiting mixed moments $\mu_{a, b, K, M}$, for all positive integers $a, b$, follows from Lemma \ref{lm:moment_limit}), implies that $(T_{n,K, M}^+,Y_{n,K, M})$ converges in distribution and in all mixed moments to some random variable $(T_{K,M}^+,Y_{K,M})$. This completes the proof. 
\end{proof}

\subsection{Deriving the Limiting Distribution of $(T_{n,K,M}^+, Z_{n,K,M})$}
\label{sec:pfquadratic5}

Let $G_n$ be a sequence of graphs satisfying the assumptions of Lemma \ref{lm:moment_limit}, with the functions $W_K: [0, K]^2 \rightarrow [0, 1]$ and $d_K:[0, K] \rightarrow [0, 1]$. Denote by 
\begin{align}\label{eq:WKM}
W_{K, M}(x, y)=W_K(x, y) \bm  1\{\max\{d_K(x), d_K(y)\} \leq M\}. 
\end{align} 
For $W_{K, M}$, define its $L$-step piecewise constant approximation (note that it has $K^2L^2$ blocks) as follows: 
\begin{align*}
W_{K, M}^{(L)}(x,y) &:=\sum_{1\leq a, b \leq K L}  r^{(L)}_{K, M}(a, b) \bm 1\left\{x \in \left(\frac{a-1}{L}, \frac{a}{L}\right] \right\} \bm 1\left\{y \in \left(\frac{b-1}{L}, \frac{b}{L}\right] \right\}.
\end{align*}
where  
\begin{align}\label{eq:r}
r^{(L)}_{K, M}(a, b):= L^2 \int_{\frac{a-1}{L}}^{\frac{a}{L}} \int_{\frac{b-1}{L}}^{\frac{b}{L}} W_{K,M}(u,v)  \mathrm du \mathrm dv.
\end{align} 
By Proposition \ref{ppn:block}, $\lim_{L \rightarrow \infty}||W_{K, M}^{(L)}-W_{K,M}||_{\square([0, K]^2)} \leq \lim_{L \rightarrow \infty}||W_{K, M}^{(L)}-W_{K,M}||_{L_1([0, K]^2)} \rightarrow 0$.

\begin{defn}
Given a function $H:[0,K]^2\rightarrow [0,1]$ and a positive integer $N$, the $H$-random graph on $N$ vertices (denoted by $\cG(N, H)$) is the simple undirected labelled random graph with vertex set $[N]:=\{1,2,\ldots,N\}$ and edges are present independently, with
$$\P((u, v) \in E(\cG(N,H))=H\left( \frac{K u}{N},  \frac{K v}{N}\right), \quad \text{ for } 1\le u < v \le N.$$ 
\end{defn}

Fix $K \geq1$. Let $G_{n, K, M}^{(L)}$ be the $W_{K, M}^{(L)}$-random graph $\cG(\ceil{Kr_n}, W_{K, M}^{(L)})$, independent of $\{X_v\}_{v \in V(G_n)}$. 
For $u \in [0, K]$, define the function 
\begin{align}\label{eq:delta_KM}
\Delta_{K, M}(u):=\left(d_K(u)-\int_0^K W_K(u, v)\mathrm dv \right)  \bm  1\{d_K(u) \leq M\}.
\end{align} 
Define the $L$-step approximation of $\Delta_{K, M}^{(L)}$ as follows: 
\begin{align}\label{eq:delta_KM_II}
\Delta_{K, M}^{(L)}(x)= \sum_{a=1}^{KL} \eta_{K, M}^{(L)}(a) \bm 1\left\{x \in \left(\frac{a-1}{L}, \frac{a}{L}\right]\right\} = \eta_{K, M}^{(L)}(\ceil{L x}), 
\end{align}
where $\eta_{K, M}^{(L)}(a):= L \int_{\frac{(a-1)}{L}}^{\frac{a}{L}} \Delta_{K, M}(u) \mathrm du$. By Proposition \ref{ppn:block}, $||\Delta_{K, M}^{(L)}-\Delta_{K, M}||_{L_1([0, K])} \rightarrow 0$, as $L \rightarrow \infty$.

Recall that  $A(G_{n, K, M}^{(L)})=((A(G_{n, K, M}^{(L)})(u, v)))_{1 \leq u, v \leq \ceil{K r_n}}$ is the adjacency matrix of the graph $G_{n, K, M}^{(L)}$. Let $N =\ceil{Kr_n}$ and define 
\begin{align}\label{eq:T}
\overline{T}_{n, L, K, M}^+:=\sum_{1 \leq u < v \leq N} A(G_{n, K, M}^{(L)})(u, v)  X_u X_v, \quad 
and \quad \overline Y_{n, L, K,M} = \sum_{u=1}^{N} \eta_{K, M}^{(L)}\left(\left\lceil\frac{K L u}{N}\right\rceil\right) X_u.
\end{align} 
 
\begin{lem}\label{lm:t1t2L} Fix an integer $K \geq 1$ and $ M > 0$.  Under the assumptions of Lemma \ref{lm:moment_limit}, for $t_1, t_2 \geq 0$, 
\begin{align}
\lim_{L \rightarrow \infty} \lim_{n\rightarrow \infty}\E \exp&\Big\{-t_1 \overline{T}_{n, L, K, M}^+-t_2 \overline Y_{n,L,K,M} \big\}\nonumber \\
&=\E \exp\left\{\frac{1}{2}\int_0^K \int_0^K \phi_{t_1, K, M}(x,y) \mathrm d N(x)\mathrm dN(y)-t_2 \int_0^K\Delta_{K, M}(x)\mathrm d N(x)\right\},
\end{align}
where 
\begin{itemize}
\item[--]$\{N(t), t \geq 0\}$ is a homogenous Poisson process of rate $1$, 
\item[--] $\phi_{t_1, K, M}(x, y):=\log(1-W_{K, M}(x, y)+W_{K, M}(x, y)e^{-t_1})$, where $W_{K, M}$ is as defined in \eqref{eq:WKM}, and  
\item[--] $\Delta_{K, M}(x)$ as in \eqref{eq:delta_KM}.
\end{itemize}
\end{lem}

\begin{proof} Throughout the proof denote $N=\ceil{K r_n}$. The linear part (recall \eqref{eq:T}) can be written as 
\begin{equation}\label{eqybar}
\overline Y_{n, L, K,M} =  \sum_{a=1}^{K L} \eta_{K, M}^{(L)}(a) X_{n}(a),
\end{equation}
where $X_n(a):= \sum_{u=1}^{N} \bm 1\left\{ \left\lceil\frac{KL u}{N }\right\rceil=a \right\} X_u$. Note that 
\begin{align}\label{eq:Xna}
X_n(a)\sim \dBin\left( \sum_{u=1}^{N} \bm 1\left\{ \left\lceil\frac{KL u}{N }\right\rceil=a \right\}, p_n\right),
\end{align} 
and $\{X_n(a)\}_{1\le a\le KL}$ are mutually independent. 

For notational brevity, take $\sigma_n=K/N=K/\ceil{Kr_n}$. For the quadratic term, taking an expectation over the random graph $G_{n, L}^{(K)}$ we get,  
\begin{align*}
\E &\left[ e^{-t_1  \overline{T}_{n, L, K, M}^+ } \Big| X_1,\ldots,X_N\right]\\ 
&\;\;\;\;\;\; =\prod_{1\le u < v \le N} \E \left[e^{-t_1 A(G_{n, K, M}^{(L)})(u, v)  X_u X_v } \Big|X_1,\ldots,X_N\right]\\ 
&\;\;\;\;\;\; =\prod_{1\le u < v \le N}\left(1-W_{K, M}^{(L)}\left(\sigma_nu, \sigma_nv \right)+W_{K, M}^{(L)}\left(\sigma_nu,\sigma_nv \right)e^{-t_1 X_u X_v}\right)\\
&\;\;\;\;\;\; =\prod_{1\le u < v \le N}\left(1-W_{K, M}^{(L)}\left(\sigma_nu, \sigma_nv \right)+W_{K, M}^{(L)}\left(\sigma_nu,\sigma_nv \right)e^{-t_1}\right)^{X_u X_v},
\end{align*}
where the last equality uses the fact that $X_uX_v$ is a Bernoulli random variable. Using the definition of $W_{K, M}^{(L)}$, the RHS above equals
\begin{align}\label{eq:quad}
\prod_{1\le a < b \le KL} \left(\varphi_{t_1, L, K}(a, b) \right)^{X_n(a)X_n(b)} \prod_{a=1}^{KL} \left( \varphi_{t_1, L, K}(a, a)\right)^{{X_n(a)}\choose 2},
\end{align} 
where $\varphi_{t_1, L, K}(a, b):=1-r^{(L)}_{K, M}(a,b)+r^{(L)}_{K, M}(a,b)e^{-t_1}$ (recall \eqref{eq:r}). 

On letting $n\rightarrow\infty$, we have
$$\left\{X_n(1),X_n(2) \ldots,X_n(KL)\right\}\stackrel{D}{\rightarrow}\left\{\partial N(1), \partial N(2)\ldots,\partial N(KL)\right\},$$
where $\{N(t): 0 \leq t \leq K\}$ is a Poisson process of rate $1$ and $\partial N(a):=N(\frac{a}{L})-N(\frac{a-1}{L})\sim \dPois(1/L)$ (by \eqref{eq:Xna} and the Poisson approximation to the binomial distribution). Note that $\{\partial N(1), \partial N(2),$ $\ldots,\partial N(KL)\}$ is independent, since increments of the Poisson process are independent. Therefore, by \eqref{eqybar}, \eqref{eq:quad} and the continuous mapping theorem, as $n \rightarrow \infty$, 
\begin{align}\label{eq:YT}
\left(\overline Y_{n, L, K, M}, \E\left[e^{-t_1  \overline{T}_{n, L, K, M}^+ } \Big|X_1,\ldots,X_N\right] \right) \stackrel{D}{\rightarrow}& (\psi_{L,K, M}, \theta_{L, K, M}),
\end{align}
where $\psi_{L,K, M}:=\sum_{a=1}^{KL} \eta_{K, M}^{(L)}(a) \partial N(a)$ and 
$$\theta_{L, K, M}:= \prod_{1\le a < b \le KL}   \left(\varphi_{t_1, L, K}(a, b)\right)^{\partial N(a) \partial N(b)}  \prod_{a=1}^{KL} \left(\varphi_{t_1, L, K}(a, a) \right)^{{\partial N(a)}\choose 2}.$$ 
For $x,y\in [0,K]$, defining 
\begin{eqnarray*}
\phi_{t_1, L, K, M}(x, y):=&\log \varphi_{t_1, L, K} (\lceil Lx\rceil ,\lceil Ly\rceil ) & \text{ if }\lceil Lx \rceil\ne \lceil Ly \rceil,\\
:=&0 &\text{ if }\lceil Lx \rceil= \lceil Ly \rceil
\end{eqnarray*} 
gives, 
\begin{align}\label{eq:theta_LKM}
\log \theta_{L, K, M} &=\frac{1}{2}\sum_{1\le a\ne b\le KL}\partial N(a) \partial N(b)\log\varphi_{t_1, L, K}(a, b) + \sum_{a=1}^{KL} {{\partial N(a)}\choose 2}\log \varphi_{t_1, L, K}(a, a) \nonumber \\
&=\frac{1}{2}\int_{[0,K]^2} \phi_{t_1, L, K, M}(x, y) \mathrm dN(x)\mathrm dN(y)+  \sum_{a=1}^{KL} {{\partial N(a)}\choose 2}\log \varphi_{t_1, L, K}(a, a),
\end{align} 
using the definition of the stochastic integral for elementary functions (Definition \ref{defn:integral_elementary}).

To begin with we consider the first term in \eqref{eq:theta_LKM} above. Recall the definition of $\phi_{t_1, K, M}(x, y)$ from the statement of the Lemma \ref{lm:t1t2L}. Using 
\begin{align}\label{eq:WLKM}
\lim_{L \rightarrow \infty}||W_{K, M}^{(L)}-W_{K, M}||_{L_1([0, K]^2)}=0,
\end{align} 
and the Dominated Convergence Theorem (note that the functions $\phi_{t_1,L,K,M}$ and $\phi_{t_1,K,M}$ are bounded above by $\log(1+e^{-t_1})$ and bounded below by $-t_1$), gives $||\phi_{t_1, L, K, M}(x, y) - \phi_{t_1, K, M}(x, y)||_{L_1([0, K]^2)} \rightarrow 0$, as $L \rightarrow \infty$, for every $t_1 \geq 0$ fixed. Then, by Proposition \ref{ppn:limit},  as $L \rightarrow \infty$,
\begin{align}\label{eq:L11}
\frac{1}{2}\int_{[0,K]^2} \phi_{t_1, L, K, M}(x, y) \mathrm dN(x)\mathrm dN(y)\stackrel{P}{\rightarrow} \frac{1}{2}\int_{[0,K]^2} \phi_{t_1, K, M}(x, y) \mathrm dN(x)\mathrm dN(y).
\end{align}
Next, consider the second term in \eqref{eq:theta_LKM}: Using $\E{{\partial N(a)}\choose 2} \lesssim 1/L^2$ and $\sup_{a}|\log \varphi_{t_1, L, K}(a, a)| \lesssim_{ t_1} 1$ gives, 
\begin{align*}
\sum_{a=1}^{KL} \E {{\partial N(a)}\choose 2}\log \varphi_{t_1, L, K}(a, a) \lesssim_{K, t_1} \frac{1}{L}.
\end{align*}
Therefore, 
\begin{align}\label{eq:L12}
\sum_{a=1}^{KL} {{\partial N(a)}\choose 2}\log \varphi_{t_1, L, K}(a, a) \stackrel{L_1}\rightarrow 0,
\end{align}
as $L \rightarrow \infty$. The limits in \eqref{eq:L11} and \eqref{eq:L12} combined with \eqref{eq:theta_LKM} gives, 
\begin{align}\label{eq:L1}
\log \theta_{L, K, M} \stackrel{P}{\rightarrow} \frac{1}{2}\int_{[0,K]^2} \phi_{t_1, K, M}(x, y) \mathrm dN(x)\mathrm dN(y).
\end{align}

Similarly, as $L \rightarrow \infty$, 
\begin{align}\label{eq:L2}
\psi_{L,K, M}:=\sum_{a=1}^{KL} \eta_{K, M}^{(L)}(a) \partial N(a) \pto \int_0^K \Delta_{K, M}(x) \mathrm dN(x).
\end{align}
Combining \eqref{eq:L1} and \eqref{eq:L2} with \eqref{eq:YT}, and another application of the dominated convergence theorem completes the proof of the lemma.
\end{proof}

Next, we show that the limiting distribution of $(T_{n,K,M}^+,Y_{n,K,M})$ is same as that of 
$(\overline{T}_{n, L, K, M}^+ , \overline Y_{n,L,K,M})$ derived above.

\begin{lem}\label{lm:t1t2} Fix $K, M \geq 1$.   Under the assumptions of Lemma \ref{lm:moment_limit}, $(T_{n,K,M}^+,Y_{n,K,M})$ converge in distribution and in moments, as $n\rightarrow\infty$, to $(T_{K,M}^+,Y_{K,M})$, with joint moment generating function \begin{align}\label{eq:t1t2}
\E \exp&\Big\{-t_1 T_{K,M}^+-t_2 Y_{K,M} \big\}\nonumber \\
&=\E \exp\left\{\frac{1}{2}\int_0^K \int_0^K \phi_{t_1, K, M}(x,y) \mathrm d N(x)\mathrm dN(y)-t_2 \int_0^K\Delta_{K, M}(x)\mathrm d N(x)\right\},
\end{align}
with $t_1,t_2\ge 0$, $\Delta_{K, M}(\cdot)$, $\{N(t), t \geq 0\}$, and $\phi_{t_1, K, M}(\cdot, \cdot)$ are as defined in Lemma \ref{lm:t1t2L}. Moreover, $(T_{n,K,M}^+,Z_{n,K,M})$ converge in distribution and in moments, as $n\rightarrow\infty$, to $(T_{K,M}^+,Z_{K,M})$,  with joint moment generating function 
\begin{align}\label{eq:t1t2_II}
\E \exp&\Big\{-t_1 T_{K,M}^+-t_2 Z_{K,M} \big\}\nonumber \\
&=\E \exp\left\{\frac{1}{2}\int_0^K \int_0^K \phi_{t_1, K, M}(x,y) \mathrm d N(x)\mathrm dN(y)-(1-e^{-t_2}) \int_0^K\Delta_{K, M}(x)\mathrm d N(x)\right\}. 
\end{align}
\end{lem}

\begin{proof} We begin by computing $\lim_{L \rightarrow \infty}\lim_{n\to \infty} \E\left[(\overline T_{n,L,K,M}^+)^a (\overline Y_{n,L, K,M})^b\right]$. To this end,  let $\cN_{N}^{a, b}$ be the collection of all $(a+b)$-tuples of the form
$$\bm e=((u_1, v_1),\ldots ,(u_a, v_a), u_1', \ldots, u_b'),$$ where  $u_1, v_1,\ldots ,u_a, v_a, u_1', \ldots, u_b' \in [N]:=\{1, 2, \ldots, N\}$, and $u_i < v_i$, for $1 \leq i \leq a$. Then recalling \eqref{eq:T}, it follows that 
\begin{align}\label{eq:TY_ab_expectation}
\E & \left[(\overline T_{n,L,K,M}^+)^a (\overline Y_{n,L, K,M})^b\right]  \nonumber \\ 
&= \sum_{\bm e \in \cN_{N}^{a, b}} \E\left[ \prod_{s=1}^a A(G_{n, K, M}^{(L)})(u_s, v_s) X_{u_s} X_{v_s}  \prod_{s=1}^b  \eta_{K, M}^{(L)}\left(\left\lceil\frac{K L u_s'}{N}\right\rceil\right)  X_{u_s'}   \right]  \nonumber \\
\notag&= \sum_{\bm e \in \cN_{N}^{a, b}} \frac{1}{r_n^{|V(H)|}} \E\left[ \prod_{s=1}^a A(G_{n, K, M}^{(L)})(u_s, v_s) \right] \prod_{s=1}^b  \eta_{K, M}^{(L)}\left(\left\lceil\frac{K L u_s'}{N}\right\rceil\right)\\
&= \sum_{\bm e \in \cN_{N}^{a, b}} \frac{1}{r_n^{|V(H)|}}  \prod_{s=1}^a W_{K,M}^{(L)}\left(\frac{K u_s}{N} , \frac{K v_s}{N} \right)  \prod_{s=1}^b  \eta_{K, M}^{(L)}\left(\left\lceil\frac{K L u_s'}{N}\right\rceil\right),
\end{align} 
where $H$ is the graph formed by the union of the edges $(u_1, v_1), (u_2, v_2), \ldots, (u_a, v_a)$  and the vertices $u_1', u_2', \ldots, u_b'$. Note that since $u_s' \in [N]=[\ceil{K r_n}]$, there exists $x_s' \in [0, K]$ such that $u_s'=\ceil{x_s' r_n}$. This implies 
\begin{align}\label{eq:nkml}
\eta_{K, M}^{(L)}\left(\left\lceil\frac{K L u_s'}{N}\right\rceil\right)=\eta_{K, M}^{(L)}\left(\left\lceil\frac{K L \ceil{x_s' r_n}}{\ceil{K r_n}}\right\rceil\right):=\hat{\Delta}_{n,K,M}^{(L)}(x_s'). 
\end{align} 
Similarly, let  $x_s, y_s \in [0, K]$ be such that $u_s=\ceil{x_s r_n}$ and $v_s=\ceil{y_s r_n}$. Then 
\begin{align}\label{eq:adj_Gn_random}
W_{K,M}^{(L)}\left(\frac{K u_s}{N} , \frac{K v_s}{N} \right)=W_{K,M}^{(L)}\left(\frac{K \lceil x_sr_n\rceil }{N} , \frac{K \lceil y_s r_n\rceil}{N}\right):=\hat{W}_{n,K,M}^{(L)}(x_s,y_s). 
\end{align}

Now, let $\{w_1, w_2, \ldots, w_{|V(H)|}\}$ be any labelling of the vertices in $V(H)$ and $z_j$ be such that $w_j=\ceil{z_j r_n}$. Then, as in \eqref{eq:diff_Z_I}, using  
 \eqref{eq:nkml} and \eqref{eq:adj_Gn_random}, for every graph $H$ with at most $a$ edges and at most $b$ isolated vertices and every vector $\bm \eta=(\eta_1, \eta_2, \ldots, \eta_{|V(H)|})$, \eqref{eq:TY_ab_expectation} can be rewritten as, 
\begin{align}\label{eq:TY_LKM_ab}
&\E\left[(\overline T_{n,L,K,M}^+)^a (\overline Y_{n,L, K,M})^b\right] \nonumber \\
&= \sum_{H \in \sG_{a, b}} \sum_{\bm \eta \in [0, b]^{|V(H)|}} c(H, \bm \eta)  \int_{\sB_{K, n}} t(H, \hat{W}_{n, K, M}^{(L)}, \bm z) \prod_{j=1}^{|V(H)|} \hat{\Delta}_{n,K,M}^{(L)}(z_j)^{\eta_j} \mathrm d z_j,
\end{align}
where, as in \eqref{eq:diff_Z_II}, $\sB_{K, n}:=[0, \frac{\ceil{K r_n}}{r_n} ]^{|V(H)|}$, $\bm z=(z_1, z_2, \ldots, z_{|V(H)|})$, $t(H, \cdot, \bm z)$ is as defined in Definition \ref{defn:tHWu}, and $\sG_{a, b}$  is the collection of graphs with at most $a$ edges and at most $b$ isolated vertices, and $[0, b]^{|V(H)|}:=\{0, 1, 2, \ldots, b\}^{|V(H)|}$. 

Now, since the sum in \eqref{eq:TY_LKM_ab} is over  a finite set (not depending on $n$) and  each term in the integrand is bounded (by a function of $H$, $K$, and $M$), the integral over $\sB_{K, n}$ can be replaced by the integral over $\sB_K=[0, K]^{|V(H)|}$, as $n \rightarrow \infty$. Moreover, note that the functions $\hat{\Delta}_{n,K,M}^{(L)}$ and $\hat{W}_{n,K,M}^{(L)}$ converge in $L_1([0,K])$ and $L_1([0,K]^2)$ to $\Delta_{K,M}^{(L)}$ and $W_{K,M}^{(L)}$ respectively, as $n\rightarrow\infty$. Then using a telescoping argument as in the proof of Lemma \ref{lm:WH}, it follows that 
\begin{align*}%\label{eq:diff_TY_II}
&\lim_{n \rightarrow \infty}\E\left[(\overline T_{n,L,K,M}^+)^a (\overline Y_{n,L, K,M})^b\right] \nonumber \\
&= \sum_{H \in \sG_{a, b}} \sum_{\bm \eta = (\eta_1, \eta_2, \ldots, \eta_{|V(H)|}) \in [0, b]^{|V(H)|}} c(H, \bm \eta)  \int_{\sB_K} t(H, W_{K, M}^{(L)}, \bm z) \prod_{j=1}^{|V(H)|} \Delta_{K, M}^{(L)}(z_j)^{\eta_j}  \mathrm d z_j.
\end{align*}
Next, recall that $\Delta_{K,M}^{(L)}$ and $W_{K,M}^{(L)}$  converge in $L_1([0,K])$ and $L_1([0,K]^2)$ to $\Delta_{K,M}$ and $W_{K,M}$ respectively, as $L\rightarrow\infty$, and so we have 
\begin{align}
&\lim_{L \rightarrow \infty} \lim_{n \rightarrow \infty}\E\left[(\overline T_{n,L,K,M}^+)^a (\overline Y_{n,L, K,M})^b\right] \nonumber \\
&= \sum_{H \in \sG_{a, b}} \sum_{\bm \eta \in [0, b]^{|V(H)|}} c(H, \bm \eta)  \int_{\sB_K} t(H, W_{K, M}, \bm z) \prod_{j=1}^{|V(H)|}  \Delta_{K, M}(z_j)^{\eta_j}  \mathrm d z_j  \nonumber \\
&= \sum_{H \in \sG_{a, b}} \sum_{\bm \eta \in [0, b]^{|V(H)|}} c(H, \bm \eta)  \int_{\sB_K} t(H, W_K, \bm z) \prod_{j=1}^{|V(H)|} \left(d_K(z_j)-\int_0^K W_K(z_j, v)\mathrm dv\right)^{\eta_j}  \bm  1\{d_K(z_j) \leq M\}  \mathrm d z_j \tag*{(recall \eqref{eq:WKM} and \eqref{eq:delta_KM})}\nonumber \\
\label{eq:diff_TY_III} &=\lim_{n\to \infty} \E\left[(T_{n,K,M}^+)^a (Y_{n,K,M})^b\right], 
\end{align}
where the last step follows by combining \eqref{eq:diff_Z_II} and Lemma \ref{lm:WH}.

The equality of the limiting joint moments in \eqref{eq:diff_TY_III} and Lemma \ref{lm:t1t2L}, implies, by a diagonalization argument, that for every fixed $K, M \geq 1$ and $t_1, t_2 > 0$, we can find sequences $n_j$ and $L_j$ both increasing to $+\infty$ as $j \rightarrow \infty$, such that
\begin{align}\label{eq:TY_ab}
\lim_{j\to \infty} \E\left[(\overline T_{n_j,L_j,K,M}^+)^a (\overline Y_{n_j,L_j, K,M})^b\right]=\lim_{n\to \infty} \E\left[(T_{n,K,M}^+)^a (Y_{n,K,M})^b\right] =: \mu_{a,b},
\end{align}
for all non-negative integers $a, b$, and 
\begin{align}\label{eq:TY_exp_moment}
\lim_{j\rightarrow \infty}\E \exp&\Big\{-t_1 \overline{T}_{n_j, L_j, K, M}^+-t_2 \overline Y_{n_j,L_j,K,M} \Big\}\nonumber \\
&=\E \exp\left\{\frac{1}{2}\int_0^K \int_0^K \phi_{t_1, K, M}(x,y) \mathrm d N(x)\mathrm dN(y)-t_2 \int_0^K\Delta_{K, M}(x)\mathrm d N(x)\right\}. 
\end{align} 
By Lemma \ref{lm:limit_existence} and  \eqref{eq:TY_ab}, $\mu_{a,b} = \E\left[(T_{K,M}^+)^a(Y_{K,M})^b\right]$, and these mixed moments satisfy the Stieltjes moment condition. Hence, by  \eqref{eq:TY_ab}, $(\overline T_{n_j,L_j,K,M}^+ , \overline Y_{n_j,L_j,K,M}) \xrightarrow{D} (T_{K,M}^+,Y_{K,M})$ as $j \rightarrow \infty$. The result in \eqref{eq:t1t2} then follows from \eqref{eq:TY_exp_moment}. 

For \eqref{eq:t1t2_II} we compute the joint moment generating function of $(T_{n,K, M}^+,Z_{n,K, M})$: 
\begin{align*}
\E & \left[\exp\left\{-t_1 T_{n,K, M}^+ -t_2 Z_{n,K, M}\right\}\Big| (X_v)_{v \in V_{n, K, M}^+} \right] \\ 
& \hspace{0.5in} = \exp\left\{-t_1 T_{n,K, M}^+ \right\} \E\left[\exp\left(-t_2Z_{n,K,M}\Big| (X_v)_{v \in V_{n, K, M}^+}\right)\right] \\
&  \hspace{0.5in}  =  \exp\left\{-t_1 T_{n,K, M}^+ \right\}  \left(1-p_n(1-e^{-t_2})\right)^{\sum_{u \in V_{n,K, M}^+} X_u d_u^\pm} \tag*{(recall \eqref{eq:Z_II}).}
\end{align*}
Hence, using $(T_{n,K,M}^+,Y_{n,K,M}) \dto (T_{K,M}^+,Y_{K,M})$, as $n \rightarrow \infty$, and the dominated convergence theorem, 
\begin{align}\label{eq:cfKM}
\E\left[\exp\left\{-t_1T_{n,K, M}^+ -t_2Z_{n,K, M}\right\} \right]  & = \E\left[\exp\left\{-t_1 T_{n,KM}^+ \right\} \left(1-p_n(1-e^{-t_2})\right)^{\frac{Y_{n,K,M}}{p_n}}\right] \nonumber \\ 
& \rightarrow  \E\left[\exp\left\{-t_1 T_{K, M}^+ \right\} \exp\left\{Y_{K, M}(e^{-t_2}-1)\right\}\right]. 
\end{align} 
Note that, by \eqref{eq:t1t2}, the RHS of \eqref{eq:cfKM} is the moment generating function of $(T_{K, M}^+, Y_{K, M})$ evaluated at the points $-t_1$ and $-(1-e^{t_2})$. This implies,  $(T_{n, K, M}^+, Z_{n, K, M}) \rightarrow (T_{K, M}^+, Z_{K, M})$ in distribution and in moments (by uniform integrability, using Lemma \ref{lm:TZ}), where the joint moment generating function is given by  \eqref{eq:t1t2_II}. 
\end{proof}

\subsection{Completing the Proof of \eqref{eq:Q1Q2Q3} in Theorem \ref{thm:poisson_quadratic}} 
\label{sec:pfquadratic6}

We now combine the results from the previous sections and complete the proof of \eqref{eq:Q1Q2Q3}. 

\begin{lem}\label{lm:TY_KM} Fix $M > 0$ large enough. Under the assumptions of Theorem \ref{thm:poisson_quadratic}, $(T_{n, K, M}^+, Z_{n,K,M})$ converges to $(T_M^+, Z_M)$ under the double limit as $n \rightarrow \infty$ followed by $K\rightarrow \infty$, in distribution and in moments, where the limiting moment generating function is given by 
\begin{align}\label{eq:t1t2M}
\E \exp&\Big\{-t_1 T_{M}^+-t_2 Z_{M} \Big\}\nonumber \\
&=\E \exp\left\{ \frac{1}{2}\int_0^\infty \int_0^\infty \phi_{t_1, M}(x,y) \mathrm d N(x)\mathrm dN(y)-(1-e^{-t_2}) \int_0^\infty \Delta_{M}(x)\mathrm d N(x)\right\},
\end{align} 
for $t_1, t_2 > 0$, where 
\begin{itemize}
\item[--] $\phi_{t_1, M}(x, y):=\log(1-W^{(M)}(x, y)+W^{(M)}(x, y)e^{-t_1})$, where $W^{(M)}(x, y)=W(x, y) \bm  1\{d(x) \leq M, d(y) \leq M\}$, with $W:[0, \infty)^2 \rightarrow [0, 1]$ and $d: [0, \infty) \rightarrow [0, \infty)$ as in the statement of Theorem \ref{thm:poisson_quadratic}, and  
\item[--] $\Delta_{M}(x):=(d(x)-\int_0^\infty W(x, y) \mathrm dy) \bm 1 \{d(x) \leq M\}$.
\end{itemize} 
\end{lem}

\begin{proof} Let $W:(0, \infty)^2 \rightarrow [0, 1]$ and $d: [0, \infty) \rightarrow [0, \infty)$ be as in the statement of Theorem \ref{thm:poisson_quadratic}. Define $W_K(x, y):=W(x, y) \bm 1\{x, y \in [0, K]\}$ and $d_K(x):= d(x) \bm 1\{x \in [0, K]\}$. Then the conditions \eqref{eq:WK_condition_I} and \eqref{eq:dK_condition_II} imply that the the sequence of functions $\{W_K\}_{K \geq 1}$ and $\{d_K\}_{K \geq 1}$, satisfy the conditions \eqref{eq:W_phi_K} and \eqref{eq:L1_M}, where $\phi_{n, K}$ the identity map from $[0, K]$ to $[0, K]$, for all $n \geq 1$. Therefore, by Lemma \ref{lm:t1t2}, $(T_{n,K,M}^+,Z_{n,K,M})$ converge in distribution and in moments, as $n\rightarrow\infty$, to $(T_{K,M}^+,Z_{K,M})$. Thus, to prove the lemma it suffices to compute the limiting distribution of $(T^+_{K,M},Z_{K,M})$ as $K\rightarrow\infty$. 

To this effect, using Observation \ref{obs:Wphi_integral} below,  gives, for any $t_1 > 0$
\begin{eqnarray}\label{eq:DeltaKM_integral}
\int_0^K \int_0^K \phi_{t_1, K, M}(x,y) \mathrm dx \mathrm dy& \rightarrow& \int_0^\infty \int_0^\infty \phi_{t_1,M}(x,y)\mathrm dx \mathrm dy \nonumber \\
\int_0^K \Delta_{K, M}(x) \mathrm dx&\rightarrow& \int_0^\infty \Delta_M(x)\mathrm dx, 
\end{eqnarray} 
as $K \rightarrow \infty$. Also, noting that $\phi_{t_1, M}(x,y) \leq \phi_{t_1, K, M}(x,y)$ for all $x, y$, it follows that $\lim_{K \rightarrow \infty} ||\phi_{t_1, K, M}-\phi_{t_1, M}||_{L_1([0, \infty)^2)} = 0$. Similarly, using \eqref{eq:DeltaKM_integral}, it can be shown that $\lim_{K \rightarrow \infty} ||\Delta_{K, M}-\Delta_{M}||_{L_1([0, \infty)^2)}=0$. This implies (using the convergence of stochastic integrals in Proposition \ref{ppn:limit}), as $K  \rightarrow \infty$, that 
\begin{align*}
\int_0^K \int_0^K \phi_{t_1, K, M}(x,y) \mathrm d N(x)\mathrm dN(y) \stackrel{L_1} \rightarrow& \int_0^\infty \int_0^\infty \phi_{t_1, M}(x,y) \mathrm d N(x)\mathrm dN(y),\\
\int_0^K \Delta_{K, M}(x) \mathrm d N(x)  \stackrel{L_1} \rightarrow &\int_0^\infty \Delta_{M}(x)  \mathrm d N(x),
\end{align*}
Therefore, taking limit as $K \rightarrow \infty$ in \eqref{eq:t1t2_II} we see that the moment generating function converge to the RHS of \eqref{eq:t1t2M} (using the convergence of the stochastic integrals above and the dominated convergence theorem). This shows, $(T_{K, M}^+, Z_{K,M}) \dto (T_M^+, Z_M)$, as $K\rightarrow \infty$, with the joint moment generating function of $(T_M^+, Z_M)$ given by \eqref{eq:t1t2M}. To see that this convergence is also in moments, recall from Lemma \ref{lm:TZ} that $$\limsup_{K \rightarrow \infty}  \limsup_{n \rightarrow \infty} \E \left[(T_{n,K,M}^+)^a(Z_{n,K,M})^b \right] \lesssim_{a, b, M} 1.$$ Therefore, by uniform integrability, the convergence in moments follows. 
\end{proof}

Combining the results above we can now derive the limiting distribution of $(T_{n,K,M}^+,  Z_{n,K,M}, W_{n,K}^-)$, as $n \rightarrow \infty$ followed by $K \rightarrow \infty$. 

\begin{lem}\label{lm:limit_KM} Let  $(T_M^+, Z_M)$ be random variables with joint moment generating function as in \eqref{eq:t1t2M}. Then, under the assumptions of Theorem \ref{thm:poisson_quadratic},  
$$(T_{n,K,M}^+,  Z_{n,K,M}, W_{n,K}^-) \dto (T_{M}^+, Z_{M}, W),$$ in distribution and in all (mixed) moments, as $n \rightarrow \infty$ followed by $K \rightarrow \infty$, where $W \sim \dPois(\lambda_0)$ and $W$ is independent of $(T_{M}^+, Z_{M})$. 
\end{lem}

\begin{proof} By \eqref{eq:t1t2_II} and \eqref{eq:t1t2M}, as $n \rightarrow \infty$, followed by $K \rightarrow \infty$, for $t_1,t_2\ge 0$, 
\begin{align}\label{eq:tzcf}
\E & \left[\exp\left\{-t_1T_{n,K, M}^+ -t_2Z_{n,K, M}\right\} \right] \nonumber \\   
&  \rightarrow \E \exp\left\{ \frac{1}{2}\int_0^\infty \int_0^\infty \phi_{t_1, M}(x,y) \mathrm d N(x)\mathrm dN(y)-(1-e^{-t_2}) \int_0^\infty \Delta_{M}(x)\mathrm d N(x)\right\}, 
\end{align} 
This shows $(T_{n,K, M}^+, Z_{n,K, M}) \xrightarrow{D} (T_{M}^+,Z_{M})$, as $n \rightarrow \infty$ followed by $K \rightarrow \infty$, with joint moment generating function as above. 

Now,  recall the definition of $W_{n, K}^-$ from \eqref{eq:Wn-}. By definition, $W_{n, K}^-$ is independent of $(T_{n,K,M}^+,$ $Z_{n,K,M})$ and $W_{n, K}^- \dto W \sim \dPois(\lambda_0)$ (by condition (a) in Theorem \ref{thm:poisson_quadratic}). Therefore, $$(T_{n,K,M}^+,  Z_{n,K,M}, W_{n, K}^-) \dto (T_{M}^+, Z_{M}, W),$$ as required. 

Finally, by Lemma \ref{lm:TZ} $\limsup_{K \rightarrow \infty}  \limsup_{n \rightarrow \infty} \E \left[(T_{n,K,M}^+)^a Z_{n,K,M}^b\right] \lesssim_{a, b, M} 1$, and by Lemma \ref{lm1} $\limsup_{K \rightarrow \infty} \limsup_{n \rightarrow \infty}\E \left[({W_{n,K}^-})^c\right]  \lesssim_{c, M} 1$, for all positive integers $a, b, c$. Therefore, by the Cauchy-Schwarz inequality,
$$\E\left[({T_{n,K, M}^+})^a Z_{n,K, M}^b ({W_{n,K}^-})^c \right] \leq \left(\E\left[({T_{n,K, M}^+})^{2a} Z_{n,K, M}^{2b}\right] \E \left[ ({W_{n,K}^-})^{2c} \right] \right)^{\frac{1}{2}} \lesssim_{a, b, c, M} 1.$$ Then by uniformly integrability the convergence of the mixed moments follows. 
\end{proof}

The above lemma implies that $T_{n,K,M}^++Z_{n,K,M}+W_{n,K}^- \rightarrow T_{M}^+ +Z_{M} + W$, in distribution and in all moments. Then by Proposition \ref{ppn:abc} (recall \eqref{eq:Tn123}), 
\begin{align}\label{eq:TZW}
T_{n,M}=T_{n,K,M}^+ + T_{n,K,M}^\pm +T_{n,K,M}^- \rightarrow T_{M}^++ Z_{M} + W,
\end{align}
in all moments. Convergence in distribution follows by verifying the Stieltjes moment condition for  $T_{M}^+ + Z_{M} + W$, as follows:

\begin{lem}\label{lm:limit_M} Fix $M \geq 1$. Let  $(T_M^+, Z_M)$ be random variables with joint moment generating function as in \eqref{eq:t1t2M}. Then, under the assumptions of Theorem \ref{thm:poisson_quadratic},  
$$T_{n,M} \rightarrow T_{M}^+ + Z_{M} + W,$$
in distribution and in all (mixed) moments, as $n \rightarrow \infty$, where $W \sim \dPois(\lambda_0)$ and $W$ is independent of $(T_{M}^+, Z_{M})$. 
\end{lem}

\begin{proof} The convergence in moments follows from \eqref{eq:TZW}. To establish convergence in distribution we need to verify the  Stieltjes moment condition. To this end, 
let $G_{n,M}$ be the graph obtained from $G_n$ by removing all vertices with degree greater than $M r_n$ along with all the edges adjacent on them. Then observe that, for $a \geq 1$,  
\begin{align}
\E T_{n,M}^a =  \sum_{(u_1,v_1), (u_2,v_2), \ldots, (u_a,v_a) \in E(G_{n,M})} \frac{1}{r_n^{|V(H)|}} \leq  a^a \sum_{H \in \cH_a}  \frac{N(H, G_{n,M})}{r_n^{|V(H)|}} ,
\end{align}
where $H$ is the graph formed by the union of the edges $(u_1,v_1), (u_2,v_2), \ldots, (u_a,v_a)$, 
and $\cH_a$ the collection of all non-isomorphic graphs with at most $a$ edges and no isolated vertices.

Now, using $N(H, G_{n,M}) \leq |E(G_n)|^{\nu(H)} (M r_n)^{|V(H)|-2 \nu(H)}$, where $\nu(H)$ is the number of connected components of $H$, and \eqref{eq:ETn}, it follows that there exists some constant $C_1>0$ such that, for $n$ large enough, $\frac{N(H, G_{n,M})}{r_n^{|V(H)|}} \leq C_1^a M^{|V(H)|-2 \nu(H) } \leq C_1^a M^{2a}$, since $|V(H)| \leq 2a$ and $\nu(H) \leq a$, for $H \in \cH_a$. Finally, using $|\cH_a| \leq  C^a a^{a+1} \leq (2C)^a a^a$, for some constant $C>0$ \cite[Theorem 5]{graphs_number}, we get 
\begin{align}\label{eq:moments_TM}
\mu_a:=\lim_{n \rightarrow \infty}\E T_{n,M}^a \lesssim C_1^a (2C)^{a} M^{2a} a^{2a}.
\end{align} 
This shows that 
$$\sum_{a=1}^\infty \frac{1}{\mu_a^{\frac{1}{2a}}} \gtrsim_M \sum_{a=1}^\infty \frac{1}{a} =\infty,$$ which verifies the Stieltjes moment condition and completes the proof. 
\end{proof}

By monotonicity, as $M \rightarrow \infty$, there exist random variables $(T^+, Z)$ such that $(T_{M}^+, Z_{M}) \dto (T^+, Z)$, with joint moment generating function (which is obtained by taking limit as $M \rightarrow \infty$ in \eqref{eq:t1t2M} and using Proposition \ref{ppn:limit}), 
\begin{align}\label{eq:t1t2TY}
\E \exp&\Big\{-t_1 T^+ -t_2 Z \big\}\nonumber \\
&=\E \exp\left\{-\frac{1}{2}\int_0^\infty \int_0^\infty \phi_{W, t_1}(x, y) \mathrm d N(x)\mathrm dN(y)-(1-e^{-t_2}) \int_0^\infty \Delta(x)\mathrm d N(x)\right\},
\end{align}
where  $\phi_{W, t_1}(\cdot, \cdot)$ and $\Delta(\cdot)$ are as defined in the statement of Theorem \ref{thm:poisson_quadratic}. (Note that $$\int_0^\infty \int_0^\infty \phi_{W, t_1}(x, y) \mathrm d N(x)\mathrm dN(y) < \infty \quad \text{and} \quad \int_0^\infty \Delta(x)\mathrm d N(x) < \infty$$ almost surely, by Observation \ref{obs:Wphi_integral} and finiteness of stochastic integrals for $L_1$ integrable functions.) Thus, $T_M:=T_{M}^+ + Z_{M} + W$ converges in distribution to  $T:=T^+ + Z+W$, as $M \rightarrow \infty$, where $W \stackrel{D}= \dPois(\lambda_0)$ is independent  of $(T^+, Z)$. Therefore, using $T_n=T_{n, M}+o_P(1)$, where the $o_P(1)$-term goes to zero as $n \rightarrow \infty$ followed by $M \rightarrow \infty$ (recall Lemma \ref{Tnapprox}), and Lemma \ref{lm:limit_M}, it follows that $T_n \dto T^+ + Z+W$, where $W \sim \dPois(\lambda_0)$, $W$ is independent of $(T^+, Z)$, and the joint moment generating function of  $(T^+, Z)$ is given by \eqref{eq:t1t2TY}. This completes the proof of \eqref{eq:Q1Q2}.  \hfill $\Box$ \\

The finiteness of the integrals of $W$ and $d$, required in the proof above, is established below: 

\begin{obs}\label{obs:Wphi_integral} With $W(\cdot, \cdot), d(\cdot), \phi_{t_1}$ as in the statement of Theorem \ref{thm:poisson_quadratic}, the following hold:
\begin{enumerate}[(a)]
\item $\int_0^\infty \int_0^\infty W(x, y)  \mathrm d x \mathrm dy  <\infty$,

\item $\int_0^\infty \int_0^\infty |\phi_{W, t_1}(x, y)| \mathrm d x \mathrm dy<\infty$,

\item $\int_0^\infty d(x)  \mathrm d x <\infty.$ 

\end{enumerate} 
\end{obs}

\begin{proof}  Fixing $K\ge 1$, gives  $\frac{1}{r_n^2} |E(G_n)| \ge \frac{1}{r_n^2} |E(G_{n,K}^+)|=\frac{1}{2} \int_{[0,K]^2}W_{G_n}(x,y)\mathrm dx \mathrm dy$, 
which on letting $n\rightarrow\infty$ along with assumption \eqref{eq:WK_condition_I},  gives
$$\int_{[0,K]^2}W(x,y)\mathrm dx \mathrm d y \lesssim \limsup_{n\rightarrow\infty}\frac{1}{r_n^2} |E(G_n)|.$$
Since this holds for every $K\ge 1$, letting $K\rightarrow\infty$ along with Monotone Convergence Theorem gives
$$\int_{[0,\infty)^2}W(x,y)\mathrm dx \mathrm dy\lesssim \limsup_{n\rightarrow\infty} \frac{|E(G_n)|}{r_n^2}=O(1),$$
by \eqref{eq:ETn}. This completes the proof of (a).

The conclusion in part (b) follows immediately by invoking part (a) and noting that $0\le - \phi_{W, t_1}(x, y)\lesssim_{t_1} W(x,y)$.  

To show (c), note that by condition \eqref{eq:dK_condition_II}, for $K$, $M$ large enough, 
$$\int_0^Kd(x)\bm 1\{d(x)\le M\}\mathrm dx=\lim_{n\rightarrow\infty}\int_0^K d_{W_{G_n}}(x)\bm 1\{d_{W_{G_n}}(x)\le M\}\mathrm dx\lesssim \limsup_{n\rightarrow\infty}\frac{|E(G_n)|}{r_n^2}.$$
Taking limit $K\rightarrow\infty$ followed by $M \rightarrow\infty$ on both sides, gives $\int_0^\infty d(x)\mathrm \mathrm dx \lesssim  \limsup_{n\rightarrow\infty}\frac{ |E(G_n)|}{r_n^2}$, from which the desired conclusion follows on using \eqref{eq:ETn}. 
\end{proof}

\section{Proof of Theorem \ref{thm:distributional_limit}} 
\label{sec:pf_distributional_limit}

We begin by recalling that $\cW_K$ is the set of all symmetric measurable functions from $[0, K]^2 \rightarrow [0, 1]$. Denote by $\cM_K$ the set of all measure preserving bijections $\phi$ from $[0, K] \rightarrow [0, K]$.  Moreover, for any function $\phi \in \cM_K$, let $W^{\phi}(x, y)= W(\phi(x), \phi(y))$ and $f^\phi(x)=f(\phi(x))$, for $W \in \cW_K$ and $f: [0, K] \rightarrow [0, M]$. The following proposition shows the joint sequential compactness of the cut-metric and the $L_1$ distance. 

\begin{ppn}\label{ppn:Wnfn} Fix $K, M \geq 1$. Then given a sequence of measurable functions $W_n \in \cW_K$ and a sequence of measurable functions $f_n: [0, K] \rightarrow [0, M]$, there exists a subsequence $\{n_s\}_{s \geq 1}$ such that, 
$$\lim_{s \rightarrow \infty}\inf_{\phi \in \cM_K}\left\{ ||W_{n_s}^{\phi} -W_K||_{\square([0, K]^2)} + ||f_{n_s}^\phi -f_K||_{L_1([0, K])} \right\} = 0,$$
for some $W_K \in \cW_K$ and $f_K \in L_1([0, K])$. 
\end{ppn}

The proof of the proposition is given below in Section \ref{sec:ppnWnfn}. First, we use it to complete the proof of Theorem \ref{thm:distributional_limit}. To this end, suppose \eqref{eq:ETn} holds and $T_n$ converges in distribution to a random variable $T$. Begin by labeling the vertices of $G_n$ in non-increasing order of the degrees. Now, fix $M \in \cD$ (as in Definition \ref{defn:M_points}) and recall the definition of $T_{n,M}$ from \eqref{eq:M}, and use Lemma \ref{Tnapprox} to note that $T_{n,M}\dto T$, under the double limit as $n\rightarrow\infty$ followed by $M\rightarrow\infty$. Next, fix $K \geq 1$ and recall from \eqref{eq:Tn123}, 
\begin{align}\label{eq:Tn123_II}
T_{n, M}=T_{n, K, M}^+ + T_{n, K, M}^\pm + T_{n, K, M}^-. 
\end{align}
We will now proceed to find a subsequence $\{n_s\}_{s \geq 1}$ along which the RHS above will have a limiting distribution in the form \eqref{eq:Q1Q2Q3}. 

To begin with, observe that $\int_K^\infty \int_K^\infty W_{G_n}(x, y) \mathrm d x \mathrm d y \lesssim \frac{|E(G_n)|}{r_n^2} \lesssim 1$  by  \eqref{eq:ETn}, for $n$ large enough. Therefore, for every $K \geq 1$ fixed, there exists a subsequence depending on $K$ such that 
$$\lambda_0(K):=\lim_{n \rightarrow \infty} \frac{1}{2} \int_K^\infty \int_K^\infty W_{G_{n}}(x, y) \mathrm d x \mathrm d y$$ exists along that subsequence. Therefore, refining the subsequences at every stage and by a diagonalization argument, there exists a common subsequence $\{n_s\}_{s \geq 1}$ along which  $$\lim_{s \rightarrow \infty} \frac{1}{2}  \int_K^\infty \int_K^\infty W_{G_{n_s}}(x, y) \mathrm d x \mathrm dy=\lambda_0(K),$$ for every $K \geq 1$. Now, note that  $$\lambda_0(K+1)=\lim_{s \rightarrow \infty} \frac{1}{2}  \int_{K+1}^\infty \int_{K+1}^\infty W_{G_{n_s}}(x, y) \leq  \lim_{s \rightarrow \infty} \frac{1}{2}  \int_{K}^\infty \int_{K}^\infty W_{G_{n_s}}(x, y) \mathrm d x \mathrm dy=\lambda_0(K),$$
which implies 
\begin{align}\label{eq:lambda_s}
\lambda_0:=\lim_{K \rightarrow \infty} \lambda_0(K)=\lim_{K \rightarrow \infty}\lim_{s \rightarrow \infty} \frac{1}{2} \int_K^\infty \int_K^\infty W_{G_{n_s}}(x, y) \mathrm d x \mathrm d y
\end{align} 
exists.

Next, applying Proposition \ref{ppn:Wnfn} on the functions $$W_{G_n, K, M}(x, y):=W_{G_n}(x, y) \bm 1\{x, y \in [0, K], ~d_{W_{G_n, K}}(x) \leq M, ~d_{W_{G_n, K}}(y) \leq M\}$$ and $$d_{W_{G_n}, K}(x|M):=d_{W_{G_n}}(x) \bm 1\{x \in [0, K],~d_{W_{G_n, K}}(x) \leq M \},$$  
gives a sequence of functions $\phi_{n, K} \in \cM_K$ and a subsequence $\{n_s\}_{s \geq 1}$ such that, 
$$\lim_{s \rightarrow \infty}  ||W_{G_{n_s}, K, M}^{\phi_{n_s, K}} -W_{K,M}||_{\square([0, K]^2)} \quad \text{and} \quad || d_{W_{G_{n_s}}, K, M}^{\phi_{n_s, K}}(\cdot|M)  -d_{K, M}||_{L_1([0, K])} = 0,$$
for some $W_{K, M} \in \cW_K$ and $d_{K, M} \in L_1([0, K])$. This shows that, along this subsequence the assumptions of Lemma \ref{lm:moment_limit} are satisfied, therefore, by Lemma \ref{lm:t1t2}, along this subsequence 
\begin{align}\label{eq:T_s}
T_{n_s, K, M}^+ + Z_{n_s, K, M} \rightarrow J_{1, K, M} + J_{2, K, M}, 
\end{align}
in distribution and in moments, where the joint moment generating function of $(J_{1, K, M}, J_{2, K, M})$ is given by: For $t_1, t_2 \geq 0$, 
\begin{align*}%\label{eq:Q1Q2}
\E \exp&\Big\{-t_1 J_{1, K, M}-t_2 J_{2, K, M} \Big\}\nonumber \\
&=\E \exp\left\{\frac{1}{2}\int_0^\infty \int_0^\infty \phi_{W_K, t_1}(x, y) \mathrm d N(x)\mathrm dN(y)-(1-e^{-t_2}) \int_0^\infty \Delta_{K, M}(x)\mathrm d N(x)\right\},
\end{align*}
with
\begin{itemize}

\item $\phi_{W_{K, M}, t_1}(x, y):=\log(1-W_{K, M}(x, y)+W_{K, M}(x, y)e^{-t_1})$.

\item $\Delta_{K, M}(x):=  d_{K, M}(x)-\int_0^\infty W_{K, M}(x, y) \mathrm dy$.

\end{itemize}

Now, by the convergence in moments and Lemma \ref{lm:TZ}, for every integer $r \geq 1$, $$\E[(J_{1, K, M} + J_{2, K, M})^r] = \lim_{s \rightarrow \infty} \E [(T_{n_s, K, M}^+ + T_{n_s, K, M}^\pm)^r] \lesssim_{r, M} 1.$$ 
Therefore, there exists a subsequence $\{K_j\}_{j \geq 1}$ such that as $j \rightarrow \infty$,  $J_{1, K_j, M} + J_{2, K_j, M} \rightarrow J_M$, for some random variable $J_M   \in \overline{\cP}(\cW, \cF)$ (recall Definition \ref{defn:Wf}), in distribution and in moments. Therefore, refining the subsequence in \eqref{eq:lambda_s} and \eqref{eq:T_s}, and using the independence of $T_{n_s,K_j,M}^+ + Z_{n_s,K_j,M}$ and $W_{n_s,K_j}^-$,  it follows that, as $s \rightarrow \infty$ followed by $j \rightarrow \infty$,
$$T_{n_s,K_j,M}^++Z_{n_s,K_j,M}+W_{n_s,K_j}^- \rightarrow J_M + J_0,$$ 
in distribution and in moments, where $J_0 \sim \dPois(\lambda_0)$ and $J_{0}$ independent of $J_M$, for all $M \in \cD$. Then by the proof of Lemma \ref{lm:limit_M}, it follows that as, $s \rightarrow \infty$ followed $j \rightarrow \infty$, 
$$T_{n_s, K_j, M}^+ + T_{n_s, K_j, M}^\pm + T_{n_s, K_j, M}^- \dto J_M + J_0.$$ Now, as $T_{n_s,M}\dto T$, when $s \rightarrow \infty$ followed by $M \rightarrow \infty$,  recalling \eqref{eq:Tn123_II} we get, $ J_M + J_0 \dto J +J_0 \stackrel{D}= T$, as $M \rightarrow \infty$, where $J$ is independent of $J_0$ and $J \in \overline{\cP}(\cW, \cF)$, since $J_M \in \overline{\cP}(\cW, \cF)$.

\subsection{Proof of Proposition \ref{ppn:Wnfn}}\label{sec:ppnWnfn} Without loss of generality, we assume $K=M=1$. Hereafter, we fix $L \geq 1$. Then, we have the following: 
\begin{itemize}

\item For the graphon $W_n \in \cW_1$ by the weak regularity lemma \cite[Corollary 3.4]{graph_limits_I}, we can find a partition $\Pi_{n, L}=\{\pi_{n, L}(i)\}_{i \in [q_L]}$ of $[0, 1]$ into measurable sets, with $q_L\lesssim_L 1$ (a constant depending only on $L$), such that 
\begin{align}\label{eq:WnL}
||W_n-W_{n, L}||_{\square([0, 1]^2)} \le \tfrac{1}{L}.
\end{align}
where
$$W_{n, L}(x, y)=b_{W_n, L}(i, j)=\frac{1}{\lambda(\pi_i) \lambda(\pi_j)} \int_{\pi_i \times \pi_j}  W_n(x, y) \mathrm d x \mathrm d x,$$ for $x \in \pi_i$ and $y \in \pi_j$. (Here, $\lambda(A)$ denotes the Lebesgue measure of a measurable set $A \subset [0,  1]$.) Moreover, the partitions can be constructed in such a way that $\Pi_{n, L+1}=\{\pi_{n, L+1}(i)\}_{a \in [q_{L+1}]}$ is a refinement of $\Pi_{n, L}$ (by \cite[Corollary 3.4]{graph_limits_I}). 

\item Similarly, for the function $f_n$, there exists a partition $\Gamma_{n, L}=\{\gamma_{n, L}(i)\}_{i \in[r_L]}$ of $[0, 1]$ into $r_L \lesssim_L 1$ (a constant depending only on $L$) measurable sets and a vector ${\bm z}_{f_n, L}=(z_{f_n, L}(i))_{i \in [r_L]}$ with entries in $[0,1]$, such that the function 
\begin{align}\label{eq:fnL}
f_{n, L}(x):=z_{f_n, L}(i) := \frac{1}{\lambda(\gamma_{n, L}(i))}\int_{\gamma_{n, L}(i)} f_n(x) \mathrm d x \quad \text{ if } \quad x\in \gamma_{n, L}(i), 
\end{align}
satisfies 
\begin{align}\label{eq:fnL_II}
||f_n-f_{n, L}||_{L_1([0, 1])} \le \tfrac{1}{L}. 
\end{align}
Moreover, as before, the partitions can be constructed in such a way that $\Gamma_{n, L+1}=\{\gamma_{n, L+1}(i)\}_{a \in [r_{L+1}]}$ is a refinement of $\Gamma_{n, L}$.  
\end{itemize}

Given the partitions $\Pi_{n, L}=\{\pi_{n, L}(i)\}_{i \in[q_L]}$ and $\Gamma_{n, L}=\{\gamma_{n, L}(i)\}_{i \in[r_L]}$, the class of sets $$\left\{\theta_{n, L}(i_1, i_2):=\pi_{n, L}(i_1) \bigcap \gamma_{n, L}(i_2) \right \}_{i_1\in [q_L],~ i_2 \in [r_L]},$$ forms a partition of $[0,1]$, which refines both the partitions $\Pi_{n, L}$ and $\Gamma_{n, L}$ (with possibly some empty sets). Relabel  the sets $\{\theta_{n,L}(i_1, i_2)\}_{i_1\in [q_L],~ i_2 \in [r_L]}$ by $\{\theta_{n,L}(i)\}_{i \in [q_L r_L]}$ by taking a bijection from $[q_L] \times [r_L] \rightarrow [q_L r_L]$, and denote this partition of $[0, 1]$ by $\Theta_{n, L}:=\{\theta_{n,L}(i)\}_{i \in [q_L r_L]}$. Now, setting $\beta_{n, L}(i):=\lambda(\theta_{n, L}(i))$, there exists  a measure preserving bijection $\phi_{n, L}:[0, 1] \rightarrow [0, 1]$ such that 
the interval 
$$\left(\sum_{j=1}^{a-1} \beta_{n, L}(i), \sum_{j=1}^a \beta_{n, L}(i) \right] \text{ maps to the set } \theta_{n, L}(a), \text{ for each } 1 \leq a \leq q_L r_L.$$ 
Thus, the functions $W_{n, L}^{\phi_{n, L}}$ and $f_{n, L}^{\phi_{n, L}}$ are both step functions on $[0, 1]^2$ and $[0, 1]$ with intervals and rectangles as steps, respectively. Then, we can find a common subsequence $\{n_s\}_{s \geq 1}$ along which the sequence of vectors $$\left(\{\beta_{n_s, L}(i)\}_{i \in [q_L r_L]},~{\bm B}_{W_{n_s}, L},~{\bm z}_{f_{n_s}, L} \right) \in [0,1]^{q_L r_L + q_L^2 + r_L}$$ converge. (Here, we consider ${\bm B}_{W_{n_s}, L}$ as a vector of length $q_L^2$.) In particular, this means that along this subsequence the functions $W_{n_s, L}^{\phi_{n, L}}$ and $f_{n_s, L}^{\phi_{n, L}}$ converge almost surely to step functions $W_L:[0, 1]^2 \rightarrow [0, 1]$ and $f_L: [0, 1] \rightarrow [0, 1]$,  respectively. 
 
Now, let $(U,V) \sim \mathrm{Unif}([0,1]^2)$, and let $\cF_L$ denote the sub-sigma algebra of $\sB([0, 1]^2)$ (the Borel sigma algebra on $[0, 1]^2$) defined as 
$$\cF_L:=\sigma\left(\bm 1\left\{U \in \left(\sum_{j=1}^{i_1-1} \beta_L(j), \sum_{j=1}^{i_1} \beta_L(j)\right] \right\}, \bm 1\left\{V\in \left(\sum_{j=1}^{i_2-1} \beta_L(j), \sum_{j=1}^{i_2} \beta_L(j) \right] \right\}, i_1, i_2 \in [q_L r_L] \right),$$ where $\beta_L(i) =\lim_{s \rightarrow \infty}\beta_{n_s, L}(i)$ for $i \in [q_L r_L]$. Since the partition $\Theta_{n, L+1}=\{\theta_{n, L+1}(i)\}_{i\in [q_{L+1} r_{L+1}]}$ is a refinement of $\Theta_{n, L}=\{\theta_{n, L}(i)\}_{i\in [q_L r_L]}$, it follows that $\{\cF_L\}_{L \geq 1}$ is a filtration. Also, the construction implies that for any $(x,y)\in (0,1]^2$ such that $$(x,y)\in \left(\sum_{j=1}^{i_1-1} \beta_L(j), \sum_{j=1}^{i_1} \beta_L(j) \right] \times \left(\sum_{j=1}^{i_2-1} \beta_L(j), \sum_{j=1}^{i_2} \beta_L(j) \right], \quad \text{ where } \quad 1 \leq i_1, i_2 \leq  q_L r_L,$$ we have
\begin{align*}
W_{L}(x, y)=&\E\left(W_{L+1}(U, V) \Big| (U, V) \in \left(\sum_{j=1}^{i_1-1} \beta_L(j), \sum_{j=1}^{i_1} \beta_L(j) \right] \times \left(\sum_{j=1}^{i_2-1} \beta_L(j), \sum_{j=1}^{i_2} \beta_L(j) \right] \right),
 \end{align*}
and 
\begin{align*}
f_L(x)=& \E \left(f_{L+1}(U) \Big| U \in \left(\sum_{j=1}^{i_1-1}\beta_L(j), \sum_{a=1}^{i_1} \beta_L(j) \right]\right).
 \end{align*} 
Thus, both $W_L$ and $f_L$ are bounded martingales with respect to the filtration $\{\cF_L\}_{L \geq 1}$, and so they converge almost surely and in $L_1$ to functions $W_\infty$ and $f_\infty$, as $L \rightarrow \infty$, respectively. Therefore, by the triangle inequality, 
\begin{align}\label{eq:WnfnL}
& \inf_{\phi \in \cM_1}\left\{ ||W_{n_s}^{\phi} -W_\infty||_{\square([0, 1]^2)} + ||f_{n_s}^\phi -f_\infty||_{L_1([0, 1])} \right\} \leq S_1+S_2+S_3, 
\end{align}
where $S_1, S_2, S_3$ are defined as follows: 
$$S_1:=||W_{n_s}^{\phi_{n_s, L}}-W_{n_s, L}^{\phi_{n_s, L}}||_{\square([0, 1]^2)}+||f_{n_s}^{\phi_{n_s, L}} - f_{n_s,L}^{\phi_{n_s, L}}||_{L_1([0, 1])} \leq \tfrac{2}{L},$$
where the last inequality uses \eqref{eq:WnL} and \eqref{eq:fnL_II}. 
Next, 
\begin{align*}
S_2:= ||W_{n_s, L}^{\phi_{n_s, L}}-W_L||_{\square([0, 1]^2)} & + ||f_{n_s, L}^{\phi_{n_s, L}}-f_L||_{L_1([0, 1])}  \nonumber \\ 
& \leq ||W_{n_s, L}^{\phi_{n_s, L}}-W_L||_{L_1([0, 1]^2)} + ||f_{n_s,L}^{\phi_{n_s, L}}-f_L||_{L_1([0, 1])},
\end{align*}
which goes to zero as $s \rightarrow \infty$, using the fact that $W_{n_s, L}^{\phi_{n_s, L}}$ and $f_{n_s,L}^{\phi_{n_s, L}}$ converges in $L_1$ to $W_L$ and $f_L$, respectively.
Finally, 
$$S_3:= ||W_L-W_\infty||_{\square([0, 1]^2)} + ||f_L-f_\infty||_{L_1([0, 1])} \leq  ||W_L-W_\infty||_{L_1([0, 1]^2)} + ||f_L-f_\infty||_{L_1([0, 1])},$$
which goes to zero as $L \rightarrow \infty$, using $W_L \stackrel{L_1} \rightarrow W$ and $f_L \stackrel{L_1} \rightarrow f$. Putting together the above three bounds with \eqref{eq:WnfnL}, and taking limit as $s \rightarrow \infty$ followed by $L \rightarrow \infty$,  the result follows. \hfill $\Box$

\section{Proofs of Corollaries} 
\label{sec:pfcorollary}

In this section we prove Corollaries \ref{cor:poisson_quadratic_W}, \ref{cor:trunc}, and \ref{cor:quadratic_general}.

\subsection{Proof of Corollary \ref{cor:poisson_quadratic_W}} As $\{G_n\}_{n \geq 1}$ is a sequence of dense graphs, assumption \eqref{eq:ETn} implies that $r_n=1/p_n >  Cn$, for some constant $C>0$, when $n$ is large enough. 
Therefore, by the definition in \eqref{eq:WGn}, $W_{G_n}$ is zero is outside the box $[0, a]^2$, where $a := 1/C$. Hence, 
\begin{align}\label{eq:Q3W}
\lim_{K \rightarrow \infty} \lim_{n\rightarrow \infty} \frac{1}{2}\int_K^\infty \int_K^\infty W_{G_n}(x, y)  \mathrm dx \mathrm dy =0.
\end{align} 
We begin by showing that $W$ vanishes Lebesgue almost everywhere outside the rectangle $[0,a]^2$. To see this, let $f(x) =\bm 1\{x\geq a\}$ and $g(x) = \bm 1\{x\leq a\}$ for all $x \geq 0$. Fix $L >0$, and observe that:
\begin{align}\label{eq:intbr1}
 \int_a^{a+L}\int_0^a W(x,y)\mathrm dx \mathrm dy & = \int_{[0,a+L]^2} \left(W(x,y) - W_{G_n}(x,y)\right)f(x)g(y) \mathrm dx \mathrm dy \nonumber  \\ 
& \leq ||W-W_{G_n}||_{\square([0, a+L]^2)}. 
\end{align}
Since the RHS of \eqref{eq:intbr1} converges to $0$ and $W \geq 0$, $W$ must vanish Lebesgue almost everywhere on the rectangle $[a,a+L]\times[0,a]$. This means, since $L>0$ is  arbitrary, $W$ vanishes Lebesgue almost everywhere on $[a,\infty)\times[0,a]$. Interchanging the roles of $f$ and $g$, it follows that $W$ vanishes Lebesgue almost everywhere on $[0,a]\times [a,\infty)$. Finally, taking $f(x) = g(x) = \bm{1}\{x \geq a\}$, for all $x \geq 0$, and proceeding as above, we can show that $W$ vanishes  Lebesgue almost everywhere on $[a,\infty)\times [a,\infty)$, as desired.

Now, fix $K \geq a$ such that $||W_{G_n} -W||_{\square([0, K]^2)} \rightarrow 0$. Next,  let $d_W(x)=\int_0^\infty W(x, y)\mathrm d y = \int_0^a W(x, y)\mathrm d y $. Then, we have
$$\lim_{n \rightarrow \infty}\int_0^K (d_{W_{G_n}}(x)-d_W(x))^2 \mathrm dx \rightarrow 0,$$
This is because, $||W_{G_n} - W||_{\square([0, K]^2)} \rightarrow 0$ implies that 
$$\int_0^K d_{W_{G_n}}(x)^2 \mathrm d x \rightarrow \int_0^K d_{W}(x)^2 \mathrm d x \quad \text{and} \quad \int_0^K d_{W_{G_n}}(x) d_W(x) \mathrm d x \rightarrow \int_0^K d_{W}(x)^2 \mathrm d x.$$ Then, by the Cauchy-Schwarz inequality, $||d_{W_{G_n}} - d_W||_{L_1([0, K])} \rightarrow 0$, for all $K \geq a$ such that $||W_{G_n}-W||_{\square([0, K]^2)} \rightarrow 0$. This shows, condition \eqref{eq:dK_condition_II} holds with $d=d_W$, for $M \in \cD$ (recall Definition \ref{defn:M_points}). Therefore, $\Delta(x):=d_W(x)-\int_0^\infty W(x, y) \mathrm dy=0$, and the result in \eqref{eq:Q1W} follows from Theorem \ref{thm:poisson_quadratic}, with $\Delta=0$ and $\lambda_0=0$ (by \eqref{eq:Q3W}).  

To show (b), note that for any sequence of dense graphs $\{G_n\}_{n \geq 1}$, there exists a constant $a >0$ and a subsequence $\{n_j\}_{j \geq 1}$ along which $\lim_{j \rightarrow \infty}\delta_{\square([0, a]^2)}(W_{G_{n_j}}, W)=0$, for some $W \in \cW_a$ (by the compactness of the metric $\delta_{\square([0, a]^2)}$ in the space $\cW_a$, the space of all symmetric measurable functions from $[0, a]^2 \rightarrow [0, 1]$ \cite[Proposition 3.6]{graph_limits_I}). 
This implies, recalling \eqref{eq:Wdelta}, there exists a sequence of measure preserving bijections $\{\phi_{n_j}\}_{j \geq 1}$ such that $\lim_{j \rightarrow \infty}||W_{G_{n_j}}^{\phi_{n_j}} - W||_{\square([0, a]^2)}=0$. Therefore, by part (a) (recall the discussion in the second item in Remark \ref{rem:poisson_condition} and Lemma \ref{lm:moment_limit}), along this subsequence $T_{n_j} \dto Q_1$, where the moment  generating function of $Q_1$ is as given in \eqref{eq:Q1W}. This implies $T_n \dto Q_1$,  since, by assumption, $T_n$ converges in distribution, as required. \hfill $\Box$

\subsection{Proof of Corollary \ref{cor:trunc}} Note that $(c)\Rightarrow (a)$ is immediate from Theorem \ref{thm:poisson_quadratic}. Hence, it suffices to show that $(b)\Rightarrow (c)$ and $(a)\Rightarrow (b)$.

We begin with the proof of $(b)\Rightarrow (c)$. Denote by $K_{1,2}$, the 2-star graph and let $V_M:=\{v\in G_n:d_v\le Mr_n\}$. Then 
\begin{align*}
\Var(T_{n,M})=(1-p_n^2)\E T_{n,M}+2p_n^3(1-p_n)N(K_{1,2},G_n[V_M]),
\end{align*}
where $N(K_{1,2},G_n[V_M])$ is the number of 2-stars in the subgraph on $G_n$ induced on the vertex set $V_M$. Note that the conditions $\lim_{M \rightarrow \infty} \lim_{n \rightarrow \infty} \E(T_{n,M})=\lambda$ and $\lim_{M \rightarrow \infty} \lim_{n \rightarrow \infty} \Var(T_{n,M})=\lambda$ mean, $\lim_{M \rightarrow \infty} \lim_{n \rightarrow \infty} p_n^3 N(K_{1,2},G_n[V_M])= 0$. Letting $d_{W_{G_n[V_M]}}(x)=\int_0^\infty W_{G_n}(x,y)\bm 1\{d_{W_{G_n}}(x) \leq M, d_{W_{G_n}}(y) \leq M\}\mathrm dy$, then gives, 
\begin{align}
\int_0^\infty d_{W_{G_n[V_M]}}(x)^2 \mathrm dx  \lesssim p_n^3 N(K_{1,2}, G_n[V_M] )+ p_n^3|E(G_{n,M})|\rightarrow 0,
\label{eq:double_limit}
\end{align}
under the same double limit, where we invoke \eqref{eq:ETn} to deal with the second term.

We will now verify condition \eqref{eq:dK_condition_II} in Theorem \ref{thm:poisson_quadratic}. To see this, observe 
\begin{align*}
d_{W_{G_n}}(x)  {\bm 1\{d_{W_{G_n}}(x) \leq M\}} =&d_{W_{G_n[V_M]}}(x)+\int_0^\infty W_{G_n}(x,y)\bm 1\{d_{W_{G_n}}(y)>M\}\mathrm dy\\
\le &d_{W_{G_n[V_M]}}(x)+\frac{1}{M}\int_0^\infty d_{W_{G_n}}(y) \mathrm dy. 
\end{align*}
Therefore, 
\begin{align*}
\int_0^K d_{W_{G_n}}(x)^2{\bm 1}\{d_{W_{G_n}}(x)\le M\}  \mathrm dx & \le 2\int_0^\infty d_{W_{G_n[V_M]}}(x)^2\mathrm dx+\frac{2}{M^2} \int_0^K \left(\int_0^\infty d_{W_{G_n}}(y) \mathrm dy\right)^2 \mathrm dx  \\  
& = 2\int_0^\infty d_{W_{G_n[V_M]}}(x)^2\mathrm dx+\frac{2 K}{M^2} \left(\int_0^\infty    d_{W_{G_n}}(y)  \mathrm dy\right)^2.
\end{align*}
Now, under the double limit $n \rightarrow \infty$ followed by $M \rightarrow \infty$, first term in the RHS above goes to $0$ by \eqref{eq:double_limit}, and the second term goes to $0$ by \eqref{eq:ETn}. This gives, $\lim_{M \rightarrow \infty} \lim_{n \rightarrow \infty} \int_0^K d_{W_{G_n}}(x)^2{\bm 1}\{d_{W_{G_n}}(x)\le M\}  \mathrm dx =0$.
Therefore, by the Cauchy-Schwarz inequality,  
\begin{align}\label{eq:dl}
\lim_{M \rightarrow \infty} \lim_{n \rightarrow \infty} \int_0^K d_{W_{G_n}}(x){\bm 1}\{d_{W_{G_n}}(x)\le M\}  \mathrm dx =0. 
\end{align} 
This implies, $\limsup_{n \rightarrow \infty} \int_0^K d_{W_{G_n}}(x){\bm 1}\{d_{W_{G_n}}(x)\le M\}  \mathrm dx=0$, for all $M$, since $$\limsup_{n \rightarrow \infty} \int_0^K d_{W_{G_n}}(x){\bm 1}\{d_{W_{G_n}}(x)\le M\}  \mathrm dx$$ is non-decreasing in $M$. This shows \eqref{eq:dK_condition_II} with $d=0$.

Next, we will show that $\lim_{n \rightarrow \infty} ||W_{G_n}||_{L_1([0,K]^2)} = 0$, for  every $K>0$ (which implies \eqref{eq:WK_condition_I} holds with $W=0$, since $\lim_{n \rightarrow \infty} ||W_{G_n}||_{\square([0,K]^2)} \leq \lim_{n \rightarrow \infty} ||W_{G_n}||_{L_1([0,K]^2)} = 0$). To this end,  we have
\begin{align*}%\label{eq:dl2}
\int_{[0,K]^2}W_{G_n}(x,y)\mathrm dx \mathrm dy \le &\int_{[0,K]^2}W_{G_n}(x,y){\bm 1}\{d_{W_{G_{n}}}(x)\le M\}\mathrm dx \mathrm dy+\int_{[0,K]^2}\frac{d_{W_{G_n}}(x)}{M}\mathrm dx \mathrm dy\\
=&  \int_0^K d_{W_{G_n}}(x){\bm 1}\{d_{W_{G_{n}}}(x)\le M\}\mathrm dx+\frac{K}{M}\int_0^K d_{W_{G_n}}(x)\mathrm dx.
\end{align*}
On letting $n\rightarrow\infty$ followed by $M\rightarrow\infty$, the first term above converges to $0$ by \eqref{eq:dl},  and the second term converges to $0$ by \eqref{eq:ETn}. This implies $W_{G_n}$ converges to $0$ on $L_1([0,K]^2)$, verifying condition \eqref{eq:WK_condition_I}  in Theorem \ref{thm:poisson_quadratic} with $W=0$. 

Finally, we verify condition (a) of Theorem \ref{thm:poisson_quadratic}. Using  \eqref{eq:label} note that for all $K$ large enough there exists an integer $n(K)$, such that if $n>n(K)$, we have $d_v< r_n$, for all $v\in [ \lceil K r_n\rceil+1,n]$. In particular, for $M>1$ this implies 
\begin{align*}
0 \leq \int_0^\infty \int_0^\infty  & W_{G_{n,M}}(x,y)\mathrm dx\mathrm dy
- \int_K^\infty \int_K^\infty W_{G_n}(x,y)\mathrm dx\mathrm dy\\
 & \leq  2\int_0^\infty \int_0^K W_{G_n}(x,y){\bm 1}\{d_{W_{G_n}}(x)\le M\} \mathrm dx\mathrm dy \nonumber \\
& =  2\int_0^K d_{W_{G_n}}(x){\bm 1}\{d_{W_{G_n}}(x)\le M\}  \mathrm dx,
 \end{align*}
which converges, under the double limit, to $0$ by \eqref{eq:dl}, for every $K\geq 0$ fixed.  Hence, 
\begin{align*}
\lim_{K \rightarrow \infty} \lim_{n \rightarrow \infty} \frac{1}{2} \int_K^\infty \int_K^\infty W_{G_n}(x,y)\mathrm dx\mathrm dy & =  \lim_{M \rightarrow \infty} \lim_{n \rightarrow \infty}  \frac{1}{2} \int_0^\infty \int_0^\infty  W_{G_{n,M}}(x,y)\mathrm dx\mathrm dy \nonumber \\ 
&= \lim_{M \rightarrow \infty} \lim_{n \rightarrow \infty}  \E T_{n, M}=\lambda,
\end{align*}
verifying condition (a) of Theorem \ref{thm:poisson_quadratic}. This completes the proof of  $(b)\Rightarrow (c)$.

To prove $(a)\Rightarrow (b)$, recall the definition of $T_{n,M}$ from \eqref{eq:M}, and use Lemma \ref{Tnapprox} to note that $T_{n,M}\dto \dPois(\lambda)$, under the double limit as $n\rightarrow\infty$ followed by $M\rightarrow\infty$. 
Now, by a diagonalization argument, given any subsequence we can find a further subsequence $\{n_j\}_{j \geq 1}$ such that $\mu_{a, M}:=\lim_{j \rightarrow \infty}\E T_{n_j, M}^a$ converges, for all $a \geq 1$, by uniform integrability, since the moments $\sup_{n \in N}\E T_{n,M}^a  \lesssim_{M, a} 1$, are bounded (recall \eqref{eq:moments_TM}). Recall from the proof of Lemma \ref{lm:limit_M} that the moments $\{\mu_{a, M}\}_{a \geq 1}$ satisfy the Stieltjes moment condition. Therefore, along the subsequence, $T_{n_j, M} \rightarrow T_M$, in distribution and in moments, for some random variable $T_M$. Finally note that the random variables $T_{n_j,M}$ are non decreasing in $M$, and so the sequence $\{T_M\}_{M\ge 1}$ is stochastically increasing, and converges in distribution to $\dPois(\lambda)$.  Then, by the Monotone Convergence Theorem, $\E(T_{n_j,M}^a)$ converges to $\E(\dPois(\lambda)^a)$, for all integers $a \geq 1$, under the double limit. In particular, (b) follows from convergence of the first two moments.

\subsection{Proof of Corollary \ref{cor:quadratic_general}} Define $Y_i:=X_i \bm  1\{X_i\le 1\}$, for $1 \leq i \leq n$, and denote by 
$$T_n'=\frac{1}{2}\sum_{1\leq u, v \leq n}  a_{uv}(G_n) Y_u Y_v.$$
To begin with, note that the event $\{|T_n-T_n'|>0\}$ is contained in the following event: there exists  $(u, v)\in E(G_n)$ such that either $\{X_u \ge 2 \text{ and } X_v\ge 1\}$ or $\{X_u \ge 1 \text{ and } X_v \ge 2\}$. Therefore, by a union bound, 
\begin{align*}
\P(|T_n-T_n'|>0)\le 2\sum_{(u, v)\in E(G_n)}\P(X_u \ge 2)\P(X_v \ge 1) &  =2 |E(G_n)| \P(X_1 \ge 2)\P(X_1 \ge 1) \nonumber \\
& =  |E(G_n)| o(p_n^2),
\end{align*}
using $\P(X_1\geq 1) \leq \E(X_1) =O(p_n)$ and $2\P(X_1\geq 2) \leq \E(X_1)- \P(X_1=1) =o(p_n)$ by the assumption that $\lim_{n \rightarrow \infty}\frac{1}{p_n} \E X_1 =  1$.  Since $|E(G_n)|p_n^2=O(1)$ by assumption \eqref{eq:ETn}, it follows that $T_n-T_n' \pto 0$. Using $T_n' \dto Q_1+Q_2+Q_3$, by Theorem \ref{thm:poisson_quadratic}, the result follows. \\

\small{\noindent{\bf Acknowledgement:} The authors thank Shirshendu Ganguly for many illuminating discussions and helpful comments.}

\appendix

\normalsize 

\section{Approximation by Block Functions}
\label{sec:lemmas}

In this section we show that a $L$-block approximation of a $L_1$-integrable function converges to the function in $L_1$. This result has been used in the proof of Theorem \ref{thm:poisson_quadratic}.

\begin{ppn}\label{ppn:block}
Suppose $f:[0,1]^d\rightarrow \R$ is a bounded measurable function. For any integer $L\ge 1$ define the function $f_L:[0,1]^d\rightarrow \R$ as, 
$$f_L(x_1, x_2, \ldots, x_d)=L^d \prod_{i=1}^d  \int_{\frac{\lceil L x_i\rceil-1}{L}}^{\frac{\lceil L x_i\rceil }{L}} f(y_1, y_2, \ldots, y_d) \mathrm d y_1 \mathrm d y_2 \ldots \mathrm d y_d.$$
Then $||f_L-f||_{L_1([0,1]^d)}=\int_{[0,1]^d}|f_L({\bm x})-f({\bm x})|\mathrm d{\bm x} \rightarrow 0$, as $L \rightarrow \infty$. 
\end{ppn}

\begin{proof} Throughout the proof we abbreviate the norm $||\cdot||_{L_1([0,1]^d)}$ as $||\cdot||_1$. Now, fixing $\varepsilon>0$, by standard measure theory arguments, there exists a continuous function $g:[0,1]^d\rightarrow \R$ such that $\sup_{\bm x \in [0, 1]^d} |g(\bm x)| \le \sup_{\bm x \in [0, 1]^d} |f(\bm x)|$, and $||f-g||_1\le \varepsilon$. Then using Jensen's inequality, 
$||f_L-g_L||_1\le ||f-g||_1\le \varepsilon$. An application of triangle inequality then gives, 
$$||f-f_L||_1 \le ||f-g||_1 +||g-g_L||_1 +||g_L-f_L||_1 \le 2\varepsilon+||g-g_L||_1,$$
which implies $\limsup_{L\rightarrow\infty} ||f-f_L||_1 \le 2\varepsilon$, since $||g_L-g||_1\le ||g_L-g||_\infty \rightarrow 0$, by the continuity of $g$. This completes the proof since $\varepsilon>0$ is arbitrary. 
\end{proof}

\section{Stochastic Integration with Respect to a Poisson Process}
\label{sec:integral}

Let $\cX=[0, \infty)$, $\sB(\cX)$ the Borel sigma-algebra on $\cX$, and $\lambda$ is the Lebesgue measure on $(\cX, \sB(\cX))$. Denote  by $L_p(\cX^d)$ the set of all Borel measurable functions $f: \cX^d \rightarrow \R$ such that $\int_{\cX^d} |f(\bm x)|^p \mathrm d \bm x < \infty$, where $\bm x =(x_1, x_2, \ldots, x_d)$ and $\mathrm d \bm x=\mathrm dx_1\cdots \mathrm dx_d$, with respect to the Lebesgue measure on $\cX^d$. In this section, we define stochastic integration with respect to a Poisson process, for functions in $L_1(\cX^d)$. The theory of multiple stochastic integration for square integrable functions, with respect to a general centered Levy process is well-understood (see \cite{stochastic_levy} and the references therein). However, our applications require integration of functions in $L_1$ (for example, the function $\Delta(\cdot)$ in Theorem \ref{thm:poisson_quadratic} is in $L_1(\cX)$, but not in $L_2(\cX)$). In this section, we make the necessary modifications to the standard theory, extending stochastic integration  with respect to a Poisson process to $L_1$ functions.

Let $\{N(A), A \in \sB(\cX) \}$ be the  homogenous Poisson process of rate $1$ (that is, $N(A)\sim \dPois(\lambda(A))$, where $\lambda$ is the Lebesgue measure on $(\cX, \sB(\cX))$), defined on a probability space $(\Omega, \cF, \mu)$. Denote by $\cE_d$ the set of all It\^{o}-elementary functions, having the form
\begin{equation}\label{eq:itelty}
f(t_1, t_2, \ldots, t_d) =\sum_{i_1, i_2, \ldots, i_d=1}^m a_{i_1, i_2, \ldots, i_d} \bm 1_{A_{i_1}\times \cdots \times A_{i_d}} (t_1, t_2, \ldots,  t_d), 
\end{equation}
where $A_1, A_2, \ldots, A_m \in \sB(\cX)$ are pairwise disjoint, and $ a_{i_1, i_2, \ldots, i_d}$ is zero if two indices are equal. Note that an It\^{o}-elementary function need not necessarily be in $L_1(\cX^d)$. We begin by defining multiple It\^{o} integrals for functions in $\cE_d \cap L_1(\cX^d)$.

\begin{defn}\label{defn:integral_elementary} (Multiple It\^{o} integral for elementary functions) The $d$-dimensional It\^{o}-stochastic integral, with respect to the Poisson process $\{N(A), A \in \cB(\cX)\}$, for the function $f \in \cE_d \cap L_1(\cX^d)$ in  \eqref{eq:itelty} is defined as 
$$I_d(f):=\int f(x_1, x_2, \ldots, x_d) \prod_{a=1}^d \mathrm dN(x_a):=\sum_{i_1, i_2, \ldots, i_d=1}^m a_{i_1, i_2, \ldots, i_d} N(A_{i_1})\times \cdots \times N(A_{i_d}).$$ 
\end{defn}

It is easy to verify that this is well-defined, that is, if $f, g \in \cE_d \cap L_1(\cX^d)$, with $f=g$ almost everywhere Lebesgue, then $I_d(f) \stackrel{a.s.}{=} I_d(g)$. The multiple It\^{o} integral for elementary functions also satisfies the following two properties: 

\begin{itemize}

\item (Finiteness) $|I_d(f)| < \infty$ almost surely, for $f \in  \cE_d \cap L_1(\cX^d)$. To see this note that $\E[N(A_{i_1})\times \cdots \times N(A_{i_d})]=\lambda(A_{i_1})\times \cdots \times \lambda(A_{i_d})$, where $\lambda(A)$ denotes the 1-dimensional Lebesgue measure of the set $A$, whenever all the indices $i_1, i_2, \ldots, i_d$ are distinct. Therefore, 
\begin{align}\label{eq:iso}
\E[|I_d(f)|] \le \sum_{i_1, i_2, \ldots, i_d=1}^m |a_{i_1, i_2, \ldots, i_d}| \lambda(A_{i_1})\times \cdots \times \lambda(A_{i_d})= \int_{\cX^d} |f(\bm x)| \mathrm d \bm x < \infty.
\end{align}

\item (Linearity) Given two simple functions $f,g\in \cE_d \cap L_1(\cX^d)$, 
\begin{align}\label{eq:sum}
I_d(f+g) \stackrel{a.s.}{=} I_d(f)+I_d(g),
\end{align}
which is immediate from definitions. 

\end{itemize}

Now, we proceed to define multiple It\^{o} integral for general functions in $L_1(\cX^d)$. To this end, a straightforward modification of the proof of \cite[Theorem 2.1]{ito_integral}  shows that $\cE_d$ is dense in $L_1(\cX^d)$. Therefore, given $f \in L_1(\cX^d)$, there exists a sequence $\{f_n\}_{n \geq 1}$, with $f_n \in \cE_d$, such that $\lim_{n \rightarrow \infty}\int_{\cX^d}|f_n(\bm x)-f(\bm x)| \mathrm d \bm x=0$. (Note that this automatically implies $f_n\in \cE_d \cap L_1(\cX^d)$, for all $n$ large). 

\begin{ppn}\label{ppn:integral} Consider a sequence $\{f_n\}_{n \geq 1}$, with $f_n \in \cE_d$, such that $\lim_{n \rightarrow \infty}||f_n-f||_{L_1(\cX^d)}=0$. Then there exists a random variable $X$ defined on $(\Omega, \cF, \mu)$ such that $I_d(f_n) \stackrel{L_1} \rightarrow X$. Moreover, if $\{g_n\}_{n \geq 1}$, with $g_n \in \cE_d$, is another sequence such that $\lim_{n \rightarrow \infty}||g_n-f||_{L_1(\cX^d)}=0$, then the sequence of random variables $\{I_d(f_n)\}_{n \geq 1}$ and $\{I_d(g_n)\}_{n \geq 1}$ converge to the same limit in $L_1(\Omega)$.
\end{ppn}

\begin{proof}  Define the sequence $\{h_n\}_{n\ge 1}$ as follows: For $n \geq 1$, 
\begin{align*}
h_{2n-1}:=f_n \quad \text{and} \quad h_{2n}:=g_n.
\end{align*}
Note that $\lim_{n \rightarrow \infty}||h_n-f||_{L_1(\cX^d)}=0$. Therefore, given $\varepsilon>0$, there exists $N(\varepsilon)<\infty$ such that if $n_1, n_2 \ge N(\varepsilon)$, then 
$\int_{\cX^d} |h_{n_1}(\bm x)-h_{n_2}(\bm x)| \mathrm d \bm x<\varepsilon$.  This implies, 
\begin{align*}
\E |I_d(h_{n_1})-I_d(h_{n_2})|= &\E |I_d(h_{n_1}-h_{n_2})| \tag{by \eqref{eq:sum}}\\
\le &\int_{\cX^d} |h_{n_1}(\bm x)-h_{n_2}(\bm x)|d\bm x \tag{by \eqref{eq:iso}}\\
<&\varepsilon.
\end{align*}
This shows that $\{I_d(h_n)\}_{n\geq 1}$ is Cauchy in $L_1(\Omega)$, and by the completeness of
the space $L_1(\Omega)$, the result follows. 
\end{proof}

\begin{defn}\label{defn:integral_II} (Multiple It\^{o} integral for general $L_1$-functions) The $d$-dimensional It\^{o}-stochastic integral for a function $f \in L_1(\cX^d)$ (denoted as $I_d(f)$) is defined as the $L_1$ limit of the sequence $\{I_d(f_n)\}_{n \geq 1}$,  where $\{f_n\}_{n \geq 1}$ is a sequence such that $f_n \in \cE_d$ with $\lim_{n \rightarrow \infty}||f_n-f||_{L_1(\cX^d)}=0$.
\end{defn}

This is well-defined by Proposition \ref{ppn:integral}.  Also, as in the case of elementary functions, $I_d(f)$ satisfies the following properties: 
\begin{itemize}

\item (Finiteness) For any $f\in L_1(\cX^d)$, 
\begin{align}\label{eq:integral_bound}
\E|I_d(f)| \le  \int_{\cX^d} |f(\bm x)|\mathrm d\bm x.
\end{align}
To see this, let $\{f_n\}_{n\ge 1}$ be a sequence of elementary functions such that $\lim_{n \rightarrow \infty}||f_n-f||_{L_1(\cX^d)}=0$. Then using \eqref{eq:iso}, 
$$\E|I_d(f_n)|\le \int_{\cX^d} |f_n(\bm x)|\mathrm d\bm x.$$ The desired conclusion then follows on letting $n\rightarrow\infty$ on both sides of the above inequality, since $\E|I_d(f_n)| \rightarrow \E|I_d(f)|$, by Definition \ref{defn:integral_II}. 

\item (Linearity) For any two functions $f, g$ in $L_1(\cX^d)$, $I_d(f+g) \stackrel{a.s.}{=} I_d(f)+I_d(g)$, 
which is immediate from \eqref{eq:sum} and Definition \ref{defn:integral_II}. 

\end{itemize}

The following proposition shows the convergence of stochastic integrals for converging sequence of functions:

\begin{ppn}\label{ppn:limit} Consider a sequence $\{f_n\}_{n \geq 1}$ such that  $\lim_{n \rightarrow \infty}||f_n-f||_{L_1(\cX^d)}=0$. Then $I_d(f_n) \stackrel{L_1}\rightarrow I_d(f)$ in $(\Omega, \cF, \mu)$.
\end{ppn}

\begin{proof} Note that
\begin{align*}
\E |I_d(f_n)-I_d(f)|=\E |I_d(f_n-f)| \le \int |f_n(\bm x)-f(\bm x)| \mathrm d \bm x,
\end{align*}
where the first step uses linearity of stochastic integrals, and the second step uses \eqref{eq:integral_bound}. Taking limit as $n\rightarrow\infty$ on both sides, the result follows. 
\end{proof}

We conclude by computing the 2-dimensional It\^{o} stochastic integral of the block function \eqref{eq:block}.

\begin{example}\label{integralB} Fix $\kappa >0$ and consider the $B$-block function $f: [0, \kappa]^2 \rightarrow [0, 1]$ as defined in \eqref{eq:block}. Let $L \geq 1$ and define  
$$f^{(L)}(x, y)=\sum_{1\leq a \ne b \leq \ceil{\kappa L}}  r_f^{(L)}(a, b) \bm 1\left\{x \in \left[\frac{a-1}{L}, \frac{a}{L}\right]\right\} \bm 1\left\{y \in \left[\frac{b-1}{L}, \frac{b}{L}\right]\right\},$$
where  
\begin{align}\label{eq:fL}
r_f^{(L)}(a, b):= L^2 \int_{\frac{a-1}{L}}^{\frac{a}{L}} \int_{\frac{b-1}{L}}^{\frac{b}{L}} f(u,v)  \mathrm du \mathrm dv.
\end{align}
Note that the sum is over $a\ne b$, that is, $f^{(L)}(x, y)=0$ when $x, y \in [\frac{a-1}{L}, \frac{a}{L}]$, for some $1 \leq a \leq L$.  Therefore, this is the $L$-step piecewise constant approximation of $f$, with zeros on the diagonal blocks. By taking $L$ large enough, it follows that $f(x, y)=r_f^{(L)}(a, a) \lesssim 1$, for $x, y \in [\frac{a-1}{L}, \frac{a}{L}]$, which means 
$$\sum_{a=1}^L\int_{\frac{a-1}{L}}^{\frac{a}{L}} \int_{\frac{a-1}{L}}^{\frac{a}{L}} f(x, y) \mathrm dx \mathrm d y \lesssim \frac{1}{L} \rightarrow 0.$$ 
Then by Proposition \ref{ppn:block}, 
$\lim_{L \rightarrow \infty} ||f-f^{(L)}||_{L_1([0, \kappa]^2)} =0$, which means $$I_2(f^{(L)}) \stackrel{L_1}\rightarrow I_2(f),$$ by Proposition \ref{ppn:limit}. Now, let $\{N(t): 0 \leq t \leq \kappa \}$ be a Poisson process of rate $1$, and $\partial N(a):=N(\frac{a}{L})-N(\frac{a-1}{L})\sim \dPois(1/L)$.  Then taking $L$ large enough and Definition \ref{defn:integral_elementary},  
\begin{align}
I_2& (f^{(L)}) \nonumber \\  
&=\sum_{j=1}^B b_{jj} \sum_{\ceil{c_{j-1} L} \le a\ne b\le \ceil{c_{j} L}}\partial N(a) \partial N(b)  + 2 \sum_{1 \leq j < j' \leq B} b_{jj'} \sum_{\substack{ \ceil{c_{j-1}  L} \le a \le \ceil{c_j L} \\ \ceil{c'_{j-1}  L} \le b \le \ceil{c'_j L}}} \partial N(a) \partial N(b) + o_{L_1}(1) \nonumber \tag*{(where the $o_{L_1}(1)$-term goes to zero in $L_1$)}\\
& \stackrel{L_1}\rightarrow  2 \sum_{j=1}^B b_{jj} {N_j \choose 2} + 2 \sum_{1 \leq j < j' \leq B} b_{jj'} N_j  N_{j'},  \nonumber \\ 
& =I_2(f), 
\end{align} 
where $\{N_1, N_2, \ldots, N_B\}$ are independent with $N_j \sim \dPois(c_j-c_{j-1})$. 
\end{example}

\end{document}